%% file: 21n2.tex
\documentclass[multiarticle,nochapter,download]{msjmemoirs}
\usepackage{amsmath,amsthm,amssymb}

\usepackage[all]{xy}
\usepackage{bbm}

\usepackage{makeidx}
\makeindex

\usepackage[nottoc]{tocbibind}

\usepackage{etoolbox}
\patchcmd{\theindex}{\thispagestyle{plain}}{}{}{}

\allowdisplaybreaks

\newtheorem{thm}{Theorem}[section]
\newtheorem{prop}[thm]{Proposition}
\newtheorem{conj}[thm]{Conjecture}
\newtheorem{cor}[thm]{Corollary}
\newtheorem{lem}[thm]{Lemma}
\newtheorem{claim}[thm]{Claim}
\newtheorem{condition}[thm]{Condition}

\theoremstyle{definition}

\numberwithin{equation}{section}
\newtheorem{rem}[thm]{Remark}
\newtheorem{ex}[thm]{Example}

\newtheorem{defn}[thm]{Definition}

\newtheorem{assum}[thm]{Assumption}

\def\bbk{\mathbbm{k}}

\def\bbP{\mathbb{P}}
\def\bbQ{\mathbb{Q}}
\def\bbR{\mathbb{R}}
\def\bbT{\mathbb{T}}

\def\bbZ{\mathbb{Z}}

\def\bfB{\mathbf{B}}
\def\bfC{\mathbf{C}}
\def\bfF{\mathbf{F}}
\def\bfG{\mathbf{G}}

\def\bfSigma{\mathbf{\Sigma}}
\def\bfa{\mathbf{a}}
\def\bfb{\mathbf{b}}
\def\bfc{\mathbf{c}}

\def\bfe{\mathbf{e}}

\def\bfg{\mathbf{g}}
\def\bfm{\mathbf{m}}
\def\bfn{\mathbf{n}}

\def\bfv{\mathbf{v}}
\def\bfu{\mathbf{u}}
\def\bfw{\mathbf{w}}
\def\bfx{\mathbf{x}}
\def\bfy{\mathbf{y}}

\def\bfzero{\mathbf{0}}
\def\calA{\mathcal{A}}

\def\calF{\mathcal{F}}
\def\calH{\mathcal{H}}

\def\frakd{\mathfrak{d}}
\def\frakD{\mathfrak{D}}

\def\frakp{\mathfrak{p}}
\def\frakq{\mathfrak{q}}

\def\frakS{\mathfrak{S}}
\def\fraks{\mathfrak{s}}
\def\haty{\hat{y}}

\def\d{\delta}

\def\ss{\scriptstyle}

\usepackage{etoolbox}
\makeatletter
\pretocmd{\@startsection}{\gdef\thesectiontype{#1}}{}{}
\pretocmd{\@sect}{\@namedef{the\thesectiontype title}{#8}}{}{}
\pretocmd{\@ssect}{\@namedef{the\thesectiontype title}{#5}}{}{}
\makeatother

\usepackage{fancyhdr}
\begin{document}
\bibliographystyle{amsalpha}
\fontsize{11pt}{12.5pt}\selectfont

\mainmatter[page=1]
\title[\ ]{Cluster Algebras and Scattering Diagrams\\
\ \\
Part II\\
Cluster Patterns and Scattering Diagrams\thanks{
This is the final manuscript of
 Part II of the monograph ``Cluster Algebras and Scattering Diagrams'',
 MSJ Mem. 41 (2023) by the author. 
}}
\author[\ ]{Tomoki Nakanishi\\
\small Graduate School of Mathematics, Nagoya University}
\maketitle

\begin{abstract}
We review some important results by Gross, Hacking, Keel, and Kontsevich  on
cluster algebra theory, namely, the \emph{column sign-coherence of $C$-matrices\/}
and the \emph{Laurent positivity}, both of which were conjectured by Fomin and Zelevinsky.
We digest and reconstruct 
the proofs of these conjectures by Gross et al.
still based on their scattering diagram method,
however, without relying on toric/birational geometry.
 At the same time, we also give a detailed account of the correspondence
 between the notions of cluster patterns
and scattering diagrams.
 Most of the results  in this part are  found in or
translated from the known results
in the literature.
However, the approach, the construction of logic and proofs, and the
overall presentation
 are
new.
Also, as an application,
we show that there is a one-to-one correspondence
 between $g$-vectors and cluster variables
in cluster patterns with arbitrary coefficients.
\end{abstract}

\newpage
\tableofcontents[depth=2]

\newpage

\input part2.tex

\newpage
\pagestyle{myheadings}
\bibliography{../../biblist/biblist.bib}
\fontsize{10pt}{10.5pt}\selectfont
\printindex
\end{document}

%% file: part2.tex

\setcounter{section}{-1}
\pagestyle{fancy}
\fancyhead[CO]{\thesection.\ \thesectiontitle}

\section{Introduction to Part II}
\label{2sec:intro1}
In this part  we review some important results by Gross, Hacking, Keel, and Kontsevich \cite{Gross14} on
cluster algebra theory, namely, the \emph{column sign-coherence of $C$-matrices\/}
and the \emph{Laurent positivity}, both of which were conjectured by Fomin and Zelevinsky \cite{Fomin07, Fomin02}.

Let us start by recalling the background.
In the paper \cite{Fomin07} Fomin and Zelevinsky  developed a
new method to study 
\emph{cluster patterns}, which are the underlying algebraic and combinatorial
structure to define cluster algebras.
The basic ingredients of the method are \emph{$C$-matrices, $G$-matrices, and 
$F$-polynomials\/} associated with a given cluster pattern.
Through the \emph{separation formulas}, the cluster pattern
is recovered from them.
Moreover,
one can regard  $C$- and $G$-matrices as the \emph{tropical part\/} of
the cluster pattern, while $F$-polynomials as the \emph{nontropical part},
so that
the separation formulas naturally separate these two parts.
In many applications of cluster algebras or cluster patterns, this separation 
plays a central role.

In \cite{Fomin07} several conjectures on the properties of
 $C$- and $G$-matrices and $F$-polynomials were also proposed.
Among them, there are two prominent conjectures, namely,
\begin{itemize}
\item
\emph{The sign-coherence conjecture\/}:
 Every $C$-matrix is column sign-coherent.
\item
\emph{The Laurent positivity conjecture\/}: 
Every $F$-polynomial has no negative coefficients.
(It is equivalent to the earlier conjecture in \cite{Fomin02}
on the positivity of the coefficients in the Laurent polynomial expressions for cluster variables.)
\end{itemize}
Due to their importance,
these conjectures, together with other conjectures in \cite{Fomin07}, have been intensively
studied, and proved in some special cases by various methods.
Most notably, in  the skew-symmetric case,
the sign-coherence conjecture was proved 
by \cite{Derksen10, Plamondon10b, Nagao10},
and the Laurent positivity conjecture  was proved by \cite{Lee15, Davison16}.
However, the extending these proofs to the most general case, i.e., the skew-symmetrizable case, seemed difficult and demanding some new ideas.

In \cite{Gross14} Gross, Hacking, Keel, and Kontsevich finally proved 
both conjectures together in full generality by the \emph{scattering diagram method}.
The scattering diagrams  were originally introduced in \cite{Kontsevich06}
to study the toric degenerations of Calabi-Yau manifolds in homological mirror symmetry
in two dimensions,
 and they were generalized to higher dimensions by
\cite{Gross07}.

\begin{table}
\begin{center}
\begin{tabular}{ccc}
 cluster patterns & &scattering diagrams\\
\hline
 $G$-matrices & &chambers\\
 $C$-matrices & &normal vectors of walls\\
$F$-polynomials  && wall-crossing automorphisms\\
 cluster variables &&theta functions \\
 \hline
\end{tabular}
\end{center}
\caption{Correspondence between
the notions of
cluster patterns and scattering diagrams}
\label{2table:corres1}
\end{table}

Roughly speaking, we have the  correspondence between
the notions in
cluster patterns and scattering diagrams
as in Table \ref{2table:corres1}.
Let us admit this correspondence.
Then,
by the definition/construction of scattering diagrams,
any normal vector of a wall is either positive or negative.
This proves the sign-coherence conjecture.
On the other hand, any theta function has a
manifestly positive combinatorial description
by \emph{broken lines\/}
as shown in \cite{Gross14}.
This proves the Laurent positivity conjecture.
\emph{Therefore,  establishing the above correspondence
is the crucial step to prove the conjectures.}

In this part we digest and reconstruct the results 
and the proofs of the above  theorems
in \cite{Gross14} 
still based on their scattering diagram method,
however,
\emph{without relying on
toric/birational geometry therein}.
To be more specific,
we set the following guidelines while preparing the manuscript:
 \begin{itemize}
 \item We assume that readers are familiar with basic concepts in
 cluster algebras to some extent, for example, seeds and mutations, cluster variables,
 coefficients, the Laurent phenomenon, the finite type classification, etc.

 \item The goal of the part is to give reasonably self-contained  proofs of the sign-coherence
 and the Laurent positivity theorems
based on the formulation of cluster patterns by \cite{Fomin07}
with the help of some key properties of scattering diagrams from \cite{Gross14}.
 \item
 At the same time, we also give a detailed account  of the correspondence
 between the notions of cluster patterns
and scattering diagrams in  Table \ref{2table:corres1}.

 \item 
We employ the Fock-Goncharov decomposition
 of mutations \emph{with tropical sign\/}
 to work on $F$-polynomials.
This substitutes and avoids the toric/birational geometrical setting in \cite{Gross14},
 including cluster varieties,
 which plays a key role at several points in their original proofs of the  conjectures.

   \item We rely crucially on some basic properties of scattering diagrams in \cite{Gross14}
  \emph{that are independent of the  geometrical setting therein}.
  The details and the proofs of these results will be given in Part III
  of this monograph.

\item Our approach is entirely based on the formulation of cluster patterns
in \cite{Fomin07}.
However, after the paper \cite{Fomin07},
considerable progress has been made in the subject.
So,  we also take it into account and reorganize and/or streamline some part of the formulation.
Most notably, we  do not rely on the notion of  principal coefficients explicitly \emph{except
for\/}  the polynomial property of $F$-polynomials.
  \end{itemize}

Most of the results  in this part are  found in or
translated from the known results
 in \cite{Gross14},
and/or
in  other related literatures,
especially,  a series of papers by Reading
\cite{Reading11,
Reading12,
Reading17,
Reading18}.
However, the approach, the construction of logic and proofs, and the overall presentation  are
new.
For example, we prove the sign-coherence conjecture
inductively along the $n$-regular tree graph
\emph{together with\/}  establishing
the correspondence in Table \ref{2table:corres1}
for $C$- and $G$-matrices.
This approach is completely different from the toric geometrical argument in \cite{Gross14}.
On the other hand, it is close to the approach in  \cite{Reading12, Reading17}
using  \emph{mutation fans\/} therein.
In fact, many of our results can be also obtained from
the results on the mutations fans therein.
A subtle but important difference is that we work on $C$- and $G$-matrices
without referring to mutation fans at all.
Therefore,  our approach is more direct.
To summarize, our presentation
is complementary to the above references,
 and we believe it serves as an alternative guide for
the scattering diagram formalism for cluster algebras.

Finally let us mention that there are already several excellent applications
of the scattering diagram method to cluster algebra theory (e.g., \cite{Muller15,Cheung15, Bridgeland16, Cao17, Cao18, Zhou19, Qin19, Mou19, Davison19}).

\newpage
\fancyhead[CO]{\thesubsection.\ \thesubsectiontitle}

\section{Cluster patterns and separation formulas}

In this section we recall several basic notions and facts on  \emph{cluster patterns\/}
by following the formulation of Fomin and Zelevinsky especially in \cite{Fomin07}.
Most of the materials are taken or adapted from  \cite{Fomin07}, and 
we assume that the readers are familiar with them to some extent.
Thus, we skip the proofs of the results unless we think necessary.

\subsection{Seeds and mutations}

\pagestyle{fancy}

Below we fix a positive integer $n$,
which is called the \emph{rank\/} of the forthcoming seeds, cluster patterns, \emph{etc}.

\begin{defn}[Skew-symmetrizable matrix]
An $n\times n$  integer  matrix $B=(b_{ij})_{i,j=1}^n$
is said to be \emph{skew-symmetrizable\/} \index{matrix!skew-symmetrizable}
if there is a positive rational diagonal matrix $D=(d_i\delta_{ij})_{i,j=1}^n$
such that $DB$ is skew-symmetric,
i.e,
\begin{align}
\label{2eq:ss1}
d_i b_{ij}= - d_jb_{ji}
\end{align}
holds. The matrix $D$ is called a \emph{(left) skew-symmetrizer\/} of $B$.
\index{skew-symmetrizer}
\end{defn}

\begin{rem}
Often a skew-symmetrizer $D$ is assumed to be an \emph{integer\/} matrix.
However,  it is natural to use a rational one
in order to match with the convention in \cite{Gross14}.
 See \eqref{2eq:Dmat1}. So, we employ the condition here.
\end{rem}

For a skew-symmetrizable matrix $B$,
its diagonal entries $b_{ii}$ vanish
due to \eqref{2eq:ss1}.
Also,  the condition \eqref{2eq:ss1} is rephrased as
\begin{align}
\label{2eq:dbd1}
DBD^{-1}=-B^T,
\end{align}
where for any matrix $M$, $M^T$ stands for the transpose of $M$.

For our purpose,
we especially employ the following (nonstandard) definition of seeds.

\begin{defn}[Seed]
\label{2defn:seed1}
Let
$\calF_X$ and $\calF_Y$  be a pair of fields both of which are isomorphic to
the rational function field of $n$-variables
with  coefficients in $\mathbb{Q}$.
A \emph{(labeled) seed in $(\calF_X, \calF_Y)$} \index{seed} \index{seed!labeled}
is a triplet
$\Sigma=(\bfx,\bfy,B)$,
where 
$\bfx=(x_1,\dots,x_n)$
and $\bfy=(y_1,\dots,y_n)$
are  $n$-tuples of free  generating elements
in $\calF_X$ and $\calF_Y$, respectively,
and $B=(b_{ij})_{i,j=1}^n$
is an $n\times n$
skew-symmetrizable (integer) matrix.
We call $\bfx$, $\bfy$, and $B$,
respectively, 
the \emph{$x$-variables}, \index{$x$-variable (cluster variable)}
the \emph{$y$-variables},  \index{$y$-variable (coefficient)}
and the \emph{exchange matrix\/} of $\Sigma$. \index{matrix!exchange}
We also call the pair $(\calF_X, \calF_Y)$ the \emph{ambient fields\/}  \index{ambient field}
of a seed $\Sigma$
and the forthcoming cluster patterns, \emph{etc}.
\end{defn}

\begin{rem}
In the standard terminology of \cite{Fomin07},
the $x$-variables above are \emph{cluster variables without coefficients},
while the $y$-variables above are \emph{coefficients in the
universal semifield\/} $\mathbb{Q}_{\mathrm{sf}}(\mathbf{y})\subset \calF_Y$.
Here, $x$- and $y$-variables are on an equal footing and independent of each other.
In particular,
$y$-variables do \emph{not\/}  serve as coefficients for $x$-variables. 
\end{rem}

For any seed $\Sigma=(\bfx,\bfy,B)$,
we attach an $n$-tuple 
$\hat\bfy=(\hat{y}_1,\dots,\hat{y}_n)$
of elements in the ambient field $\calF_X$ defined by
\begin{align}
\label{2eq:yhat1}
\hat{y}_i
= \prod_{j=1}^n x_j^{b_{ji}}.
\end{align}
They play an important role in cluster algebra theory.
We call them  \emph{$\hat{y}$-variables}. \index{$\hat{y}$-variable}
They are algebraically independent if and only if $B$ is nonsingular.

For any integer $a$,
we define 
\begin{align}
[a]_+:=\max(a,0).
\end{align}
We have a useful identity:
\begin{align}
\label{2eq:pos1}
a=[a]_+ - [-a]_+.
\end{align}

\begin{defn}[Seed mutation]
\index{mutation!seed}
\label{2defn:mut1}
For any seed $\Sigma=(\bfx,\bfy,B)$
and $k\in \{1,\dots,n\}$,
we define a new seed 
$\Sigma'=(\bfx',\bfy',B')$
by the following rule:
\begin{align}
\label{2eq:xmut1}
x'_i
&=
\begin{cases}
\displaystyle
x_k^{-1}\Biggl(\, \prod_{j=1}^n x_j^{[-b_{jk}]_+}
\Biggr)
( 1+\hat{y}_k)& i=k,
\\
x_i
&i\neq k,
\end{cases}
\\
\label{2eq:ymut1}
y'_i
&=
\begin{cases}
\displaystyle
y_k^{-1}
& i=k,
\\
y_i y_k^{[b_{ki}]_+} (1+ y_k)^{-b_{ki}}
&i\neq k,
\end{cases}
\\
\label{2eq:bmut1}
b'_{ij}&=
\begin{cases}
-b_{ij}
&
\text{$i=k$ or $j=k$,}
\\
b_{ij}+
b_{ik} [b_{kj}]_+
+
[-b_{ik}]_+b_{kj}
&
i,j\neq k,
\end{cases}
\end{align}
where $\hat{y}_k$ in \eqref{2eq:xmut1} is the $\hat{y}$-variable in  \eqref{2eq:yhat1}.
The seed $\Sigma'$ is called the \emph{mutation of $\Sigma$ in direction $k$},
and denoted by $\mu_k(\Sigma)=\mu_k(\bfx,\bfy,B)$.
\end{defn}

The following facts
ensure that $(\bfx',\bfy',B')$ is indeed a seed.
\begin{itemize}
\item[(a).] If $D$ is a skew-symmetrizer of $B$,
then it is also a skew-symmetrizer of $B'$.
In particular,
the matrix $B'$ is skew-symmetrizable.
\item[(b).]
Regard the mutation $\mu_k$ as a map from the 
set of seeds to itself. Then, $\mu_k$ is an involution.
In particular, $\bfx'$ and $\bfy'$ are free
generating elements
in $\calF_X$ and $\calF_Y$, respectively.
\end{itemize}

Also, one can easily verify that $|B'|=|B|$,
so that $B'$ is nonsingular if and only if $B$ is nonsingular.

The following fact can be easily verified.
\begin{prop}[{\cite[Prop.~3.9]{Fomin07}}]
\label{2prop:yhat1}
The $\hat{y}$-variables mutate as
\begin{align}
\hat{y}'_i
=
\begin{cases}
\displaystyle
\hat{y}_k^{-1}
& i=k,
\\
\hat{y}_i \hat{y}_k^{[b_{ki}]_+} (1+ \hat{y}_k)^{-b_{ki}}
&i\neq k
\end{cases}
\end{align}
in the ambient field $\calF_X$
by the same rule as \eqref{2eq:ymut1}.
\end{prop}

There is some flexibility to write the mutation formulas
 \eqref{2eq:xmut1}--\eqref{2eq:bmut1} as follows.
(The choice $\varepsilon=1$ reduces to the one
in \eqref{2eq:xmut1}--\eqref{2eq:bmut1}.)

\begin{prop}[{\cite[Eqs.~(2.9), (2.10)]{Nakanishi11c},
\cite[Eq.~(3.1)]{Berenstein05}}]
\label{2prop:eps1}
Let $\varepsilon \in \{1,-1\}$.
Then, the right hand sides of the
following equalities do not depend on the
choice of $\varepsilon$:
\begin{align}
\label{2eq:xmut2}
x'_i
&=
\begin{cases}
\displaystyle
x_k^{-1}\Biggl(\, \prod_{j=1}^n x_j^{[-\varepsilon b_{jk}]_+}
\Biggr)
( 1+\hat{y}_k^{\varepsilon})
& i=k,
\\
x_i
&i\neq k,
\end{cases}
\\
\label{2eq:ymut2}
y'_i
&=
\begin{cases}
\displaystyle
y_k^{-1}
& i=k,
\\
y_i y_k^{[\varepsilon b_{ki}]_+} (1+ y_k^{\varepsilon})^{-b_{ki}}
&i\neq k,
\end{cases}
\\
\label{2eq:bmut2}
b'_{ij}&=
\begin{cases}
-b_{ij}
&
\text{$i=k$ or $j=k$,}
\\
b_{ij}
+
b_{ik} [\varepsilon b_{kj}]_+
+
[-\varepsilon b_{ik}]_+b_{kj}
&
i,j\neq k.
\end{cases}
\end{align}
\end{prop}
The claim can be easily checked with the identity \eqref{2eq:pos1}.
We call the above formulas  the \emph{$\varepsilon$-expressions\/} of mutations.
\index{$\varepsilon$-expression}

\subsection{Cluster patterns}

Let $\bbT_n$ be the \emph{$n$-regular tree} graph,
that is, a graph without cycles such that exactly $n$-edges are attached
to each vertex.
\index{regular tree}
We assume that the edges are labeled by $1$,\dots, $n$
such that the edges attached to each vertex 
have different labels.
By abusing the notation,
the set of vertices of $\bbT_n$ is also denoted by 
 $\bbT_n$.
 We say that a pair of vertices in $t$ and $t'$ in $\bbT_n$
 are ($k$-)\emph{adjacent}, or $t'$ is ($k$-)\emph{adjacent to $t$}, \index{$k$-adjacent}
 if they are connected with an edge labeled by $k$.

\begin{defn}[Cluster pattern/$B$-pattern]
\index{pattern!cluster}
\index{pattern!$B$-} \index{$B$-pattern}
A collection of seeds $\mathbf{\Sigma}=\{ \Sigma_t
=(\bfx_t,\bfy_t,B_t)
\}_{ t\in \bbT_n}
$ in $(\calF_X, \calF_Y)$
 indexed by $\bbT_n$
is called a \emph{cluster pattern\/} if,
for any vertices $t,t'\in \bbT_n$ that are $k$-adjacent,
the equality $\Sigma_{t'}=\mu_k( \Sigma_t)$ holds.
The collection of exchange matrices
 $\mathbf{B}=\{ B_t\}
_{t\in \bbT_n}
$ extracted from a cluster pattern $\bfSigma$ is called the \emph{$B$-pattern\/}
of $\bfSigma$.
\end{defn}
We arbitrary choose a distinguished
 vertex (the \emph{initial vertex\/}) $t_0$ in $\bbT_n$. \index{initial vertex}
Any cluster pattern
 $\mathbf{\Sigma}=\{ \Sigma_t
=(\bfx_t,\bfy_t,B_t)\}_{t\in \bbT_n }$ is uniquely determined from
the \emph{initial seed\/} $\Sigma_{t_0}$ at  the initial vertex $t_0$ \index{seed!initial}
by repeating mutations.

Note that one can take a common skew-symmetrizer $D$ for all $B_t$ in $\mathbf{B}$.
By \eqref{2eq:dbd1}, we have
\begin{align}
\label{2eq:DBD1}
DB_tD^{-1}=-B_t^T.
\end{align}
In particular, if $B_t$ is skew-symmetric for some $t$, it is  skew-symmetric for any $t$.
Such a cluster pattern or a $B$-pattern is said to be \emph{skew-symmetric}.
 \index{$B$-pattern!skew-symmetric}
 Similarly, if $B_t$ is nonsingular for some $t$, it is  nonsingular for any $t$.
Such a cluster pattern or a $B$-pattern is said to be \emph{nonsingular}. \index{$B$-pattern!nonsingular}

For  a seed $\Sigma_t=(\bfx_t,\bfy_t,B_t)$ ($t\in \bbT_n$)
in a cluster pattern $\mathbf{\Sigma}$,
we use the notation
\begin{align}
\label{2eq:xyt1}
\bfx_t=(x_{1;t},\dots,x_{n;t}),
\quad
\bfy_t=(y_{1;t},\dots,y_{n;t}),
\quad
B_t=(b_{ij;t})_{i,j=1}^n.
\end{align}
For the initial seed $\Sigma_{t_0}=(\bfx_{t_0},\bfy_{t_0},B_{t_0})$, we often drop the indices $t_0$ as
\begin{align}
\label{2eq:xyinit1}
\bfx_{t_0}=\bfx=(x_{1},\dots,x_{n}),
\quad
\bfy_{t_0}=\bfy=(y_{1},\dots,y_{n}),
\quad
B_{t_0}=B=(b_{ij})_{i,j=1}^n,
\end{align}
if there is no confusion.
We also use similar notations for $\hat{y}$-variables $\hat{\bfy}_t
=(\hat{y}_{1;t}, \dots, \hat{y}_{n;t})$
and the initial
$\hat{y}$-variables $\hat{\bfy}_{t_0}=\hat{\bfy}
=(\hat{y}_1,\dots,\hat{y}_n)$.

We recall the definition of  cluster algebras,
though
the  focus of this part is
 to study the properties of \emph{cluster patterns\/}
  rather than cluster algebras.
  
\begin{defn}[Cluster algebra]
\index{cluster algebra}
For a cluster pattern $\mathbf{\Sigma}$,
the associated \emph{cluster algebra\/} $\calA(\bfSigma)$
is the $\bbZ$-subalgebra of the ambient field $\calF_X$
generated by all $x$-variables (cluster variables) $x_{i,t}$ ($i=1,\dots,n; t\in \bbT_n$).
\end{defn}

\subsection{$C$-matrices, $G$-matrices, and $F$-polynomials}

Following \cite{Fomin07},
we introduce
$C$-matrices, $G$-matrices, and $F$-polynomials.
They play a crucial role  in the study  of
cluster patterns.

\subsubsection{Matrix notations}

We start by  introducing the following matrix notations (see, e.g., \cite{Nakanishi11a}).
For an $n\times n$ matrix $A$, let $[A]_+$ denote the matrix obtained from $A$ by replacing the entry $a_{ij}$
with $[a_{ij}]_+$.
Let $A^{\bullet k}$
denote the matrix obtained from $A$ by replacing the entry $a_{ij}$ ($j\neq k$)
with $0$.
Similarly, let $A^{ k\bullet}$
denote the matrix obtained from $A$ by replacing the entry $a_{ij}$ ($i\neq k$)
with 0.

For any $n\times n$ matrices,
the following facts are easily checked,
where for the last two equalities we use \eqref{2eq:pos1}.
\begin{gather}
\label{2eq:AB1}
AB^{\bullet k}=(AB)^{\bullet k},
\quad
A^{k \bullet} B=(AB)^{k \bullet },
\\
\label{2eq:AB2}
A^{\bullet k}B=A^{\bullet k}B^{k \bullet} =AB^{k \bullet},
\\
\label{2eq:AB3}
A[B]_+ + [-A]_+B=A[-B]_+ + [A]_+B,
\\
\label{2eq:AB4}
\text{If $AB=CD$, then $A[B]_+ - C[D]_+=A[-B]_+ - C[-D]_+$.}
\end{gather}

Let $J_k$ be the $n\times n$ diagonal matrix obtained from the identity matrix by replacing
the $k$th diagonal entry with $-1$.
The following properties will be repeatedly used.

\begin{lem}
\label{2lem:square1}
Let $A$ be any $n\times n$ matrix whose diagonal entries are 0.
Let $P=J_k + A^{k\bullet}$
or
$J_k + A^{\bullet k}$.
Then, we have
\begin{gather}
\label{2eq:AJ1}
A^{\bullet k}J_k=-A^{\bullet k},
\quad
J_k A^{\bullet k}=A^{\bullet k},
\quad
A^{k \bullet }J_k=A^{k \bullet },
\quad
J_k A^{k \bullet }=-A^{k \bullet },
\\
\label{2eq:Nsq1}
P^2=I,
\\
\label{2eq:Nsq2}
|P|=-1.
\end{gather}
\end{lem}

\begin{proof}
Let us prove \eqref{2eq:Nsq1} for $P=J_k + A^{k\bullet}$. By $a_{kk}=0$,
we have
\begin{align}
\begin{split}
P^2&=J_k^2 + J_k A^{k\bullet} + A^{k\bullet} J_k +A^{k\bullet}A^{k\bullet}
\\
&= I -A^{k\bullet} + A^{k\bullet} =I.
\end{split}
\end{align}
Also, we have $|P|=|J_k|=-1$.
\end{proof}

\begin{ex}
The first case of the matrix mutation \eqref{2eq:bmut1}
is written as
\begin{align}
\label{2eq:bmut3}
(B')^{k\bullet}=-B^{k\bullet},
\quad
(B')^{\bullet k}=-B^{ \bullet k}.
\end{align}
Also, the entire mutation  \eqref{2eq:bmut1} is written as
\cite[Eq.~(3.2)]{Berenstein05}
\begin{align}
\label{2eq:bmut4}
B'=(J_k + [-B]_+^{\bullet k})
B(J_k + [B]_+^{k \bullet }).
\end{align}
By \eqref{2eq:Nsq1}, we have the relation
\begin{align}
\label{2eq:bmut5}
(J_k + [-B]_+^{\bullet k})B'=
B(J_k + [B]_+^{k \bullet }).
\end{align}
\end{ex}

\subsubsection{$C$-matrices}

Let $\mathbf{\Sigma}$ be any cluster pattern,
and let $\mathbf{B}$ be the $B$-pattern of $\mathbf{\Sigma}$.
Let $t_0\in \mathbb{T}_n$ be a given initial vertex.

We  introduce a collection of 
matrices $\bfC^{t_0}=\bfC=\{ C_t\}_{ t\in \bbT_n}$
called a \emph{$C$-pattern}, \index{pattern!$C$-}\index{$C$-pattern}
which is uniquely determined  from  $\mathbf{B}$ and $t_0$, therefore, eventually only from $t_0$ and $B_{t_0}$.

\begin{defn}[$C$-matrices]
\index{matrix!$C$-}\index{$C$-matrix}
\label{2defn:Cmat1}
The \emph{$C$-matrices $C_t=(c_{ij;t})_{i,j=1}^n$ ($t\in \bbT_n$)
of a cluster pattern $\mathbf{\Sigma}$ (or a $B$-pattern $\bfB$) with an initial vertex $t_0$}
are $n\times n$
integer matrices 
that are uniquely determined 
by the following initial condition
and the mutation rule:
\begin{align}
\label{2eq:Cmat1}
C_{t_0}&=I,
\\
 \label{2eq:cmut2}
 C_{t'}&=
 C_tJ_k + C_t[B_t]_+^{k\bullet}
 + [-C_t]_+^{\bullet k}B_t,
 \end{align}
 or more explicitly,
  \begin{align}
 \label{2eq:cmutmat2}
 c_{ij;t'}&=
 \begin{cases}
 -c_{ik;t}
 &
 j= k,
 \\
 c_{ij;t} + c_{ik;t} [b_{kj;t}]_+
 + [-c_{ik;t}]_+ b_{kj;t}
 &
 j \neq k,
 \end{cases}
 \end{align}
where $t$ and $t'$ are $k$-adjacent.
Each column vector $\bfc_{i;t}=(c_{ji;t})_{j=1}^n$ of a matrix $C_t$ is called
a \emph{$c$-vector}. \index{vector!$c$-}\index{$c$-vector}
\end{defn}

\begin{rem}
The $C$-matrices appeared implicitly as the lower halves of the extended
 exchange matrices for the principal coefficients in \cite[Def.~3.1]{Fomin07}.
\end{rem}

The mutation \eqref{2eq:cmut2}, in particular, implies
\begin{align}
\label{2eq:ck1}
\bfc_{k;t'}=-\bfc_{k;t},
\end{align}
or equivalently,
\begin{align}
\label{2eq:ck2}
C_{t'}^{\bullet k}=-C_{t}^{\bullet k}.
\end{align}

One can easily check that the mutation \eqref{2eq:cmutmat2} is involutive
using \eqref{2eq:pos1}, \eqref{2eq:bmut1}, and \eqref{2eq:ck1}.
Alternatively, in the matrix form,
by multiplying $J_k$ from the right to \eqref{2eq:cmut2}, we have
\begin{align}
 C_{t'} J_k&=
C_t
+
 C_t ^{\bullet k} [B_t]_+^{k\bullet}
 + [-C_t]_+^{\bullet k}B_t^{k\bullet}.
 \end{align}
 By \eqref{2eq:bmut3} and \eqref{2eq:ck2}, 
 this can be written as
\begin{align}
C_t&=
 C_{t'} J_k +  C_{t'}[-B_{t'}]_+^{k\bullet}
 + [C_{t'}]_+^{\bullet k}B_{t'}.
 \end{align}
Then, 
by \eqref{2eq:AB3},
we have
\begin{align}
C_t&=
 C_{t'} J_k +C_{t'}  [B_{t'}]_+^{k\bullet}
 + [-C_{t'}]_+^{\bullet k}B_{t'},
 \end{align}
 which is the desired result.
Moreover, by \eqref{2eq:AB3} again,
we have 
 the following $\varepsilon$-expression of \eqref{2eq:cmut2}:
\begin{prop}[{\cite[Eq.~(2.4)]{Nakanishi11a}}]
\index{$\varepsilon$-expression!for $C$-matrices}
Let $\varepsilon \in \{1,-1\}$.
Then, the right hand side of the
following equality does not depend on the
choice of $\varepsilon$:
\begin{align}
 \label{2eq:cmut3}
 C_{t'}&=
 C_tJ_k +  C_t[\varepsilon B_t]_+^{k\bullet}
 + [-\varepsilon C_t]_+^{\bullet k} B_t.
 \end{align}
\end{prop}

\subsubsection{$G$-matrices}

Under the same assumption of $C$-matrices,
we  introduce another collection of 
matrices $\bfG^{t_0}=\bfG=\{ G_t\}_{ t\in \bbT_n}$
called a \emph{$G$-pattern}, \index{pattern!$G$-}\index{$G$-pattern}
which is  uniquely determined  from $\mathbf{B}$, $t_0$,
and $\bfC^{t_0}$ defined above.

\begin{defn}[$G$-matrices]
The \emph{$G$-matrices $G_t=(g_{ij;t})_{i,j=1}^n$ ($t\in \bbT_n$) \index{matrix!$G$-}\index{$G$-matrix}
of a cluster pattern $\mathbf{\Sigma}$  (or a $B$-pattern $\bfB$)  with an initial vertex $t_0$}
are $n\times n$
integer matrices 
that are uniquely determined 
by the following initial condition
and the mutation rule:
\begin{align}
\label{2eq:Gmat1}
G_{t_0}&=I,
\\
 \label{2eq:gmut1}
 G_{t'}&=
 G_tJ_k +G_t [- B_t]_+^{\bullet k}
 - B_{t_0}[- C_t]_+^{\bullet k},
 \end{align}
 or more explicitly,
 \begin{align}
 g_{ij;t'}&=
 \begin{cases}
 \displaystyle
 -g_{ik;t}
 + \sum_{\ell=1}^n g_{i\ell;t} [-b_{\ell k;t}]_+
 -  \sum_{\ell=1}^nb_{i\ell;t_0}  [-c_{\ell k;t}]_+ 
 &
 j= k,
 \\
 g_{ij;t}  &
 j \neq k,
 \end{cases}
 \end{align}
where $t$ and $t'$ are $k$-adjacent.
Each column vector $\bfg_{i;t}=(g_{ji;t})_{j=1}^n$ of a matrix $G_t$ is called
a \emph{$g$-vector}. \index{vector!$g$-}\index{$g$-vector}
\end{defn}

\begin{rem}
The $g$-vectors were originally introduced as
the degree vectors of cluster variables with principal coefficients in \cite{Fomin07}.
\end{rem}

The first nontrivial property so far  is the following relation
between $C$- and $G$-matrices.

\begin{prop}[{\cite[Eq.~(6.14)]{Fomin07}}]
 The following equality holds for any $t\in \bbT_n$.
 \par
 (First duality) \index{duality!first}
\begin{align}
\label{2eq:dual0}
G_t B_t = B_{t_0} C_t.
\end{align}
\end{prop}
The proof in \cite[Eq.~(6.14)]{Fomin07}
relied on the fact that $g$-vectors are the degree vectors  of cluster variables
 with principal coefficients.
We give an alternative proof based on the  definitions
of $C$- and $G$-matrices here.
Let $d(t,t')$ denote the distance in $\bbT_n$ between two vertices $t,t'
\in \bbT_n$.

\begin{proof}
We prove it  by 
 the induction on $t$ along $\bbT_n$ from $t_0$.
 For $t=t_0$, the equality \eqref{2eq:dual0} holds 
 because $C_{t_0}=G_{t_0}=I$.
 Suppose that the equality \eqref{2eq:dual0} holds for  $t\in \bbT_n$ such that
 $d(t_0,t)=d$. Let $t'$ be the vertex that is $k$-adjacent to $t$ such that $d(t_0,t')=d+1$.
 Then, we have
 \begin{align}
 \begin{split}
 G_{t'} B_{t'}
 &= 
  G_t(J_k + [- B_t]_+^{\bullet k})B_{t'}
 - B_{t_0}[- C_t]^{\bullet k}B_{t'}^{k \bullet }
\\
&
=
  G_tB_{t}(J_k + [B_t]_+^{k \bullet })
 + B_{t_0}[- C_t]^{\bullet k}B_{t}^{k \bullet }
 \\
 &
=
 B_{t_0} C_t(J_k + [B_t]_+^{k \bullet })
 + B_{t_0}[- C_t]^{\bullet k}B_{t}
 \\
 &=
  B_{t_0}C_{t'},
\end{split}
 \end{align}
 where in the second equality we used
 \eqref{2eq:bmut3} and \eqref{2eq:bmut5}.
\end{proof}

\begin{rem}
There has been no name on this equality so far in the literature; however, due to
its importance, we name it as  above,
where the duality means a relation between $C$- and $G$-matrices.
(Later we will have the second and the third dualities,
which are more nontrivial.)
\end{rem}

Thanks to
 \eqref{2eq:AB4} and  \eqref{2eq:dual0},
one can  easily check that the mutation \eqref{2eq:gmut1} is involutive
as before.
Moreover, 
by   \eqref{2eq:AB4} and \eqref{2eq:dual0} again,
we have
 the following $\varepsilon$-expression of \eqref{2eq:gmut1}.
 
\begin{prop}[{\cite[Eqs.~(6.12), (6.13)]{Fomin07}}]
\index{$\varepsilon$-expression!for $G$-matrices}
Let $\varepsilon \in \{1,-1\}$.
Then, the right hand side of the
following equality does not depend on the
choice of $\varepsilon$:
\begin{align}
 \label{2eq:gmut2}
  G_{t'}&=
 G_tJ_k +G_t [- \varepsilon B_t]_+^{\bullet k}
 - B_{t_0}[- \varepsilon C_t]^{\bullet k}. \end{align}
\end{prop}

\subsubsection{$F$-polynomials}

Let 
$\mathbf{y}=(y_1,\dots,y_n)$  be an $n$-tuple of formal variables.
Although it is not necessary at this moment,
it is natural to identify them with the initial $y$-variables $\mathbf{y}_{t_0}$
of $\mathbf{\Sigma}$
as we will see soon.
Under the same assumption of $C$-matrices,
we  introduce a collection of 
rational functions
 $\bfF^{t_0}=\bfF=\{ F_{i;t}(\bfy) \}_{ i=1,\dots,n;\, t\in \bbT_n}$
 in $\bfy$
 called a \emph{$F$-pattern}, \index{pattern!$F$-}\index{$F$-pattern}
 which is  uniquely determined  from $\mathbf{B}$, $t_0$,
and $\bfC^{t_0}$  as follows.

\begin{defn}[$F$-polynomials]
\index{$F$-polynomial}
\label{2defn:Fpoly1}
The \emph{$F$-polynomials $F_{i;t}(\bfy)$ ($i=1,\dots,n$; $t\in \bbT_n$)
of a cluster pattern $\mathbf{\Sigma}$  (or a $B$-pattern $\bfB$) with an initial vertex $t_0$}
are rational functions in
$\bfy$
with coefficients in $\bbQ$
that are uniquely determined 
by the following initial condition
and the mutation rule:
\begin{align}
\label{2eq:Finit1}
F_{i;t_0}(\bfy)&=1,
\\
 \label{2eq:Fmut1}
 F_{i;t'}(\bfy)&=
 \begin{cases}
\frac
 {
  \displaystyle
 M_{k;t}(\bfy)
 }
{ \displaystyle
 F_{k;t}(\bfy)
 }
   &
 i= k,
 \\
 F_{i;t}(\bfy)  &
 i \neq k,
 \end{cases}
 \end{align}
where $t$ and $t'$ are $k$-adjacent,
and $M_{k,t}(\bfy)$ is a rational function in $\bfy$
defined as follows:
\begin{align}
\label{2eq:M1}
 M_{k;t}(\bfy)
 &=
    \prod_{j=1}^{n}
  y_j^{[c_{jk;t}]_+}
    \prod_{j=1}^{n}
  F_{j;t}(\bfy)^{[b_{jk;t}]_+}
+
    \prod_{j=1}^{n}
  y_j^{[-c_{jk;t}]_+}
    \prod_{j=1}^{n}
  F_{j;t}(\bfy)^{[-b_{jk;t}]_+}.
\end{align}
\end{defn}

In the situation in \eqref{2eq:Fmut1},
we have
\begin{align}
 M_{k;t}(\bfy)
 =
  M_{k;t'}(\bfy).
\end{align}
Therefore, 
 the mutation 
\eqref{2eq:Fmut1} is involutive.
We do not discuss 
the $\varepsilon$-expression for \eqref{2eq:Fmut1},
because we do not use it here.

The $F$-polynomials were originally introduced as the specialized $x$-varia\-bles
with principal coefficients under the specialization $x_1=\dots=x_n=1$
 in \cite{Fomin07}.
The Laurent phenomenon applied to them yields the following important fact,
which justifies the name \emph{$F$-polynomials}.

\begin{thm}[{\cite[Prop.~3.6]{Fomin07}}]
\label{2thm:Fpoly1}
For any $i=1,\dots,n$ and $t\in \bbT_n$,
the rational function $F_{i;t}(\bfy)$ is a polynomial  in $\bfy$
with coefficients in $\bbZ$.
\end{thm}

\begin{rem}
\label{2rem:notation1}
In case we need to clarify the choice of the initial vertex $t_0$,
we write  $C$-matrices and others as, $C^{t_0}_t = (c^{t_0}_{ij;t})_{i,j=1}^n$,
$\bfc_{i;t}^{t_0}$,
$G^{t_0}_t = (g^{t_0}_{ij;t})_{i,j=1}^n$,
$\bfg_{i;t}^{t_0}$,
$F_{i;t}^{t_0}(\bfy)$.
Conversely,
if there is no confusion, we suppress the superscript $t_0$ as above for simplicity.
\end{rem}

\subsection{Separation formulas and tropicalization}
\label{2subsec:separation1}

Let us present   one of the most fundamental 
results on cluster patterns in \cite{Fomin07}.

\begin{thm}[{Separation Formulas \cite[Prop.~3.13, Cor.~6.3]{Fomin07}}]
\index{separation formula}
\label{2thm:sep1}
Let $\mathbf{\Sigma}$ be any cluster pattern,
and let $t_0\in \mathbb{T}_n$ be a given initial vertex.
Let $\bfC$, $\bfG$, and $\bfF$ be
the $C$-, $G$-, and $F$-patterns
of $\mathbf{\Sigma}$ with the initial vertex $t_0$.
Let $\bfx$, $\bfy$,  and  $\hat{\bfy}$
 be the initial $x$-, $y$-variables,
 and $\hat{y}$-variables
for the initial seed $\Sigma_{t_0}$.
Then,
the following formulas hold.
\begin{align}
\label{2eq:sep1}
x_{i;t}&=
\Biggl(\,
\prod_{j=1}^n
x_j^{g_{ji;t}}
\Biggr)
F_{i;t}(\hat{\bfy}),\\
\label{2eq:sep2}
y_{i;t}&=
\Biggl(\,
\prod_{j=1}^n
y_j^{c_{ji;t}}
\Biggr)
\prod_{j=1}^n
F_{j;t}(\bfy)^{b_{ji;t}}.
\end{align}
\end{thm}
\begin{proof}
Here is an alternative proof to the one in \cite{Fomin07}.
Using the mutations of $C$- and $G$-matrices, and $F$-polynomials,
one can easily verify
\eqref{2eq:sep1} and \eqref{2eq:sep2} 
by the induction on $t$ along $\bbT_n$ from $t_0$
with the help of \eqref{2eq:dual0}.
We leave the details as an exercise to the readers.
\end{proof}

Since  $\hat{y}$-variables mutate exactly in the same manner as $y$-variables,
 a parallel formula  holds for them:
\begin{align}
\label{2eq:sep3}
\hat{y}_{i;t}&=
\Biggl(\,
\prod_{j=1}^n
\hat{y}_j^{c_{ji;t}}
\Biggr)
\prod_{j=1}^n
F_{j;t}(\hat{\bfy})^{b_{ji;t}}.
\end{align}
Alternatively, one can  derive it from \eqref{2eq:sep1} with the help of
 the duality
\eqref{2eq:dual0}.

Thank to the separation formulas, the study of cluster patterns reduces to
the study of $C$- and $G$-matrices, and $F$-polynomials.
Moreover,
the formulas play a central role
in the interplay between the cluster algebraic structure
and various subjects.

The original meaning of the separation formula in \cite[Cor.~6.3]{Fomin07} is that
a general version of the formula \eqref{2eq:sep1} for \emph{$x$-variables
with coefficients\/} naturally ``separates'' the additions of $x$-variables and
 coefficients.
Unfortunately, we cannot observe it in \eqref{2eq:sep1}, because our $x$-variables
are \emph{without coefficients}.
Below we attach an additional meaning of ``separation'' to the formulas in view of the
 \emph{tropicalization}.

First, we briefly recall the notion of tropicalization in the context of semifields \cite{Fomin07}.
\begin{defn}
\index{semifield}
A \emph{semifield\/} $\mathbb{P}$ is
a multiplicative abelian group 
equipped with 
 a binary operation
$\oplus$ on $\mathbb{P}$
that is commutative, associative, and distributive,
i.e., $(a\oplus b) c = ac\oplus bc$.
The operation $\oplus$ is called the \emph{addition\/} in $\mathbb{P}$.
\end{defn}

Here we exclusively use the following two semifields.
\begin{ex}
\label{2ex:sf1}
(a). \emph{Universal semifield\/} $\mathbb{Q}_{\mathrm{sf}}(\bfu)$. \index{semifield!universal}
Let $\bfu=(u_1,\dots, u_n)$ be an $n$-tuple of 
commutative variables.
We say that a 
rational function $f(\bfu)\in \bbQ(\bfu)$ in $u$
has a \emph{subtraction-free expression\/} \index{subtraction-free expression}
if it is expressed as
$f(\bfu)=p(\bfu)/q(\bfu)$, where $p(\bfu)$ and $q(\bfu)$
are nonzero polynomials in $\bfu$ whose coefficients
are nonnegative.
Let   $\mathbb{Q}_{\mathrm{sf}}(\bfu)$
be the set of all rational functions in $\bfu$
having subtraction-free expressions.
Then, $\mathbb{Q}_{\mathrm{sf}}(\bfu)$
is a semifield by the usual 
multiplication and addition in $\bbQ(\bfu)$.
\par
(b). \emph{Tropical semifield\/} $\mathrm{Trop}(\bfu)$. \index{semifield!tropical}
Let $\bfu=(u_1,\dots, u_n)$ be an $n$-tuple of 
commutative variables.
Let $\mathrm{Trop}(\bfu)$ be the set of all
Laurent monomials of $\bfu$ with coefficient 1,
which is a multiplicative abelian group by the usual
multiplication.
We define the addition $\oplus$ by
\begin{align}
\label{2eq:ts1}
\prod_{i=1}^n u_i^{a_i} \oplus
\prod_{i=1}^n u_i^{b_i}
:=
\prod_{i=1}^n u_i^{\min(a_i,b_i)} .
\end{align}
Then,  $\mathrm{Trop}(\bfu)$ becomes a semifield.
The addition $\oplus$ is called the \emph{tropical sum}. \index{tropical sum}
\end{ex}

Consider $\mathbb{Q}_{\mathrm{sf}}(\bfu)$ and
$\mathrm{Trop}(\bfu)$ with common generating variables
$\bfu=(u_1,\dots, u_n)$.
Then, 
we have a unique semifield homomorphism
\begin{align}
\pi_{\mathrm{trop}}: \mathbb{Q}_{\mathrm{sf}}(\bfu)
\rightarrow \mathrm{Trop}(\bfu)
\end{align}
such that $\pi_{\mathrm{trop}}(u_i)=u_i$ for any $i=1,\dots,n$.
We call it the \emph{tropicalization homomorphism}.
\index{tropicalization!homomorphism}
\index{semifield!tropicalization homomorphism}
Roughly speaking, it extracts the ``leading monomial'' of 
a rational function $f(\bfu)$ in $ \mathbb{Q}_{\mathrm{sf}}(\bfu)$.

\begin{ex}
For $\bfu=(u_1,u_2,u_3)$,
\begin{align}
\pi_{\mathrm{trop}}
\left(
\frac{3u_1u_2^2u_3^2 + 2u_1^2u_2u_3}{3u_2^2 + u_1^2 u_2^2+
u_1u_2^3u_3}
\right)
=
\frac{u_1u_2u_3}{u_2^2}=u_1u_2^{-1}u_3.
\end{align}
\end{ex}

Now let $\bfSigma$ be any cluster pattern, and let $t_0\in \bbT_n$
be a given initial vertex.
Let $\bfy_{t_0}=\bfy$ be the initial $y$-variables in the ambient field $\mathcal{F}_Y$.
Then, the semifield generated by $\bfy$ in $\mathcal{F}_Y$ (by the usual multiplication and addition)
is identified with
the universal semifield 
$\mathbb{Q}_{\mathrm{sf}}(\bfy)$.
Moreover, any $y$-variable $y_{i;t}$ belongs to $\mathbb{Q}_{\mathrm{sf}}(\bfy)$
because the mutation \eqref{2eq:ymut1} is a subtraction-free operation.
Let us apply the tropicalization homomorphism $\pi_{\mathrm{trop}}^{t_0}:
\mathbb{Q}_{\mathrm{sf}}(\bfy)\rightarrow \mathrm{Trop}(\bfy)$
to $y$-variables. 
\begin{defn}[Tropical $y$-variable]
\index{tropical!$y$-variable}
\index{$y$-variable (coefficient)!tropical}
We call the image $\pi_{\mathrm{trop}}^{t_0}(y_{i;t})$ a
 \emph{tropical $y$-variable\/} with respect to the initial vertex $t_0$.
\end{defn}

The following fact 
 can be easily deduced from \eqref{2eq:Fmut1}
 by the induction on $t$ along $\bbT_n$ from $t_0$.

\begin{prop}[{\cite[Eq.~(5.5)]{Fomin07}}]
\begin{align}
\label{2eq:Ftrop1}
\pi_{\mathrm{trop}}^{t_0}(F_{i;t}(\bfy))=1.
\end{align}
\end{prop}
 
 \begin{rem}
The equality \eqref{2eq:Ftrop1} does not mean the polynomial
$F_{i;t}(\bfy)$ has a nonzero constant term.
For example, $\pi_{\mathrm{trop}}(y_1+y_2)=1$.
\end{rem}

Tropical $y$-variables are naturally identified with $c$-vectors as follows:
\begin{prop}[{\cite[Eq.~(3.14)]{Fomin07}}]
\label{2prop:ytrop1}
 We have
\begin{align}
\label{2eq:ytrop1}
\pi_{\mathrm{trop}}^{t_0}(y_{i;t})=y^{\bfc_{i;t}}:=\prod_{j=1}^n y_j^{c_{ji;t}}
.
\end{align}
\end{prop}
\begin{proof}
This follows from the separation formula \eqref{2eq:sep2}
and \eqref{2eq:Ftrop1}.
\end{proof}

\begin{rem}
In fact, \eqref{2eq:ytrop1} is the definition of $c$-vectors 
 in \cite[Eq. (5.8)]{Fomin07},
 where the tropical $y$-variables are identified with
\emph{principal coefficients}.
\end{rem}

Now looking back the separation formula \eqref{2eq:sep2} for $y$-variables,
we see that it naturally ``separates'' the \emph{tropical\/} and the \emph{nontropical\/} parts of
$y$-variables.

On the other hand, the tropicalization of $x$-variables 
with respect to the initial $x$-variables $\bfx$ does not yield  $g$-vectors;
rather, it yields the so-called \emph{$d$-vectors\/} (\emph{denominator vectors\/} in \cite{Fomin07}),
\index{vector!denominator ($d$-)}\index{$d$-vector}
which are the leading exponents of  the Laurent polynomial expression of $x$-variables.
However, if we regard the operation of setting all $F$-polynomials $F_{i;t}(\bfy)$ to 1
also as ``a kind of tropicalization'' for $x$-variables,
then 
 the separation formula \eqref{2eq:sep1} for $x$-variables
 can be viewed again as separating the ``tropical'' and the ``nontropical'' parts of
$x$-variables,
where 
a $g$-vector $\bfg_{i;t}=(g_{ji;t})_{j=1}^n$
is identified with  a \emph{tropical $x$-variable\/} defined by
\index{tropical!$x$-variable}\index{$x$-variable (cluster variable)!tropical}
\begin{align}
\label{2eq:xtrop1}
x^{\bfg_{i;t}}:=\prod_{j=1}^n x_j^{g_{ji;t}}.
 \end{align}

In the same token, we call 
\begin{align}
\label{2eq:yhattrop1}
\hat{y}^{\bfc_{i;t}}:=
\prod_{j=1}^n \hat{y}_j^{c_{ji;t}}
\end{align}
a  \emph{tropical $\hat{y}$-variable\/} with respect to the initial vertex $t_0$.
\index{tropical!$\hat{y}$-variable}
\index{$\hat{y}$-variable!tropical}
Thanks to the duality \eqref{2eq:dual0}, we see that the notion is compatible with a tropical
$x$-variable as
\begin{align}
\label{2eq:yhattrop2}
\hat{y}^{\bfc_{i;t}}=
\prod_{j=1}^n (x^{\bfg_{j;t}})^{b_{ji;t}}.
\end{align}
Alternatively, based on the duality \eqref{2eq:dual0},
we introduce a \emph{$\hat{c}$-vector\/} $\hat{\bfc}_{i;t}$ \index{vector!$\hat{c}$-}\index{$\hat{c}$-vector}
by
\begin{align}
\label{2eq:chat1}
\hat{\bfc}_{i;t}
&=B_{t_0} \bfc_{i;t}
=G_{t} \bfb_{i;t},
\end{align}
where $\bfb_{i;t}$ is the $i$th column vector of $B_t$.
In other words, it is the $i$th column vector of a \emph{$\hat{C}$-matrix\/} \index{matrix!$\hat{C}$-}\index{$\hat{C}$-matrix}
defined by
\begin{align}
\label{2eq:gbc2}
\hat{C}_{t}:=B_{t_0}C_{t}=G_tB_t.
\end{align}
Then,
the tropical $\hat{y}$-variable  in \eqref{2eq:yhattrop2}
is  expressed also as
\begin{align}
\label{2eq:yhattrop3}
\hat{y}^{\bfc_{i;t}}=
 {x}^{\hat{\bfc}_{i;t}}.
\end{align}
These expressions and notions play a major role in the scattering diagram method later. 

\subsection{Sign-coherence of $C$-matrices and Laurent positivity}

In the paper \cite{Fomin07} several conjectures on
the properties of cluster patterns were proposed.
Here we recall especially two (three, precisely speaking) conjectures among them,
which are the theme of this part.

Let $\bfC$ and $\bfF$ be
the $C$-  and $F$-patterns
of
any cluster pattern $\bfSigma$ with a given initial point  $t_0\in \bbT_n$.

The first conjecture is the following:
\begin{conj}[{Unit constant property \cite[Conj.~5.4]{Fomin07}}]
\label{2conj:Fconst1}
Every $F$-polynomial $F_{i;t}(\bfy)$ has constant term 1.
\end{conj}

This seemingly innocent conjecture turned out to be difficult to prove
 simply by the induction on $t$ along $\bbT_n$.
However, it was noticed by \cite{Fomin07} that Conjecture \ref{2conj:Fconst1}  is equivalent to another remarkable conjecture on $C$-matrices and $c$-vectors.

\begin{defn}[Row/Column sign-coherence]
We say that a  vector $\bfv \in \bbZ^n$ is \emph{positive\/}
(resp.~\emph{negative\/})
\index{vector!positive/negative}
if
it is nonzero vector and
 all  nonzero components  are positive (resp.~negative).
We say that a  matrix $M$ is \emph{column sign-coherent\/}
 (resp.~\emph{row sign-coherent\/}) if  each column vector (resp.~row vector) of $M$
is  either positive or negative.
\index{matrix!column/row sign-coherent}
\end{defn}

\begin{conj}[{Column sign-coherence of $C$-matrices \cite[Prop.~5.6]{Fomin07}}]
\index{sign-coherence!of $C$-matrix}
\index{sign-coherence!conjecture}
\label{2conj:Csign1}
Every $C$-matrix $C_t$ is column sign-coherent.
Equivalently,  every $c$-vector $\bfc_{i;t}$
 is either positive or negative.
\end{conj}

The equivalence between Conjectures \ref{2conj:Fconst1} and \ref{2conj:Csign1}
  easily follows from the mutation of $F$-polynomials  \eqref{2eq:Fmut1}
 \cite[Prop.~5.6]{Fomin07}.

The sign-coherence conjecture was proved in the skew-symmetric case
 by
\cite{Derksen10, Plamondon10b, Nagao10} with the representation/categorical methods,
and in general by \cite{Gross14} with the scattering diagram method.

\begin{thm}[{\cite[Cor.~5.5]{Gross14}}]
\label{2thm:Csign1}
Conjectures \ref{2conj:Fconst1}  and \ref{2conj:Csign1} hold
for any $\bfSigma$ and $t_0$.
\end{thm}

The second conjecture we consider  is the following:

\begin{conj}[{Laurent positivity \cite[\S 3]{Fomin07}, \cite[\S 3]{Fomin02}}]
\index{Laurent positivity}
\label{2conj:Fpositive1}
Every $F$-polynomial $F_{i;t}(\bfy)\in \bbZ[\bfy]$ has no negative 
coefficients.
\end{conj}

Through the separation formula, this is equivalent to the earlier conjecture in \cite{Fomin02} on the positivity of the coefficients
in the Laurent polynomial expressions for $x$-variables.
This conjecture was proved for  surface type
by \cite{Musiker09}, for  acyclic type by \cite{Kimura12},
in the  skew-symmetric case by \cite{Lee15,Davison16},
with various methods,
and in general by \cite{Gross14} with the scattering diagram method.

\begin{thm}[{\cite[Theorem 4.10]{Gross14}}]
\label{2thm:Fpositive1}
Conjecture \ref{2conj:Fpositive1} holds
for any  $\bfSigma$ and $t_0$.
\end{thm}

From these conjectures (now theorems)
many other conjectures
in \cite{Fomin07} and new results follow (e.g., \cite{Nakanishi11a, Cao17, Cao18, Nakanishi19}).
Therefore, they are considered to be in the heart of cluster algebra theory.

As mentioned in Introduction,
the purpose of the part is
to digest and reconstruct the proofs of these conjectures in \cite{Gross14}.
Therefore, we temporally forget Theorems
\ref{2thm:Csign1} and \ref{2thm:Fpositive1},
and keep Conjectures  \ref{2conj:Fconst1},
 \ref{2conj:Csign1},  and \ref{2conj:Fpositive1}
as conjectures until we prove them here.
 
 \begin{rem}
Theorem \ref{2thm:Csign1}
 was   proved in 
 \cite{Gross14} with some toric geometrical setting.
 In contrast,
  we prove Conjecture  \ref{2conj:Csign1}
 without any toric geometrical setting.
\end{rem}

\newpage
\section{More about $C$- and $G$-matrices}
In this section we study more about $C$- and $G$-matrices,
namely, the tropical part of a cluster pattern.
Here we only consider a $B$-pattern $\bfB$, and we do not refer to a cluster pattern $\bfSigma$.
Therefore, we regard that $C$- and $G$-matrices are directly  associated with
a $B$-pattern $\bfB$.

\subsection{Mutations and second duality}

Let $\bfB$ be any $B$-pattern, and let
$\bfC^{t_0}=\bfC$ and $\bfG^{t_0}=\bfG$
are the $C$- and $G$-patterns of $\bfB$ with
a given initial vertex $t_0$. 
In view of Proposition \ref{2prop:ytrop1}, we introduce the following notion.
\begin{defn}[Tropical sign]
\label{2defn:tropicalsign1}
Suppose that
 a $c$-vector $\bfc_{i;t}^{t_0}=\bfc_{i;t}$ is either positive or negative.
 Then, the common sign $\varepsilon_{i;t}^{t_0}=\varepsilon_{i;t}\in \{1,-1\}$ 
of all nonzero components of $\bfc_{i;t}$
is called the \emph{tropical sign\/} of $\bfc_{i;t}$. \index{tropical sign}
\end{defn}

\begin{ex}
(a).
For the initial vertex $t_0$, we have $C_{t_0}=I$. Therefore, we  have
\begin{align}
\varepsilon_{i;t_0}=1,
\quad
(i=1,\dots,n).
\end{align}
\par
(b). Let $t, t'\in \bbT_n$ be vertices that are $k$-adjacent.
Then, by \eqref{2eq:ck1},
we have
\begin{align}
\label{2eq:epk1}
\varepsilon_{k;t'}=-\varepsilon_{k;t}.
\end{align}
\end{ex}

\begin{prop}
\label{2prop:dual1}
Suppose that the sign-coherence conjecture holds.
Then, the following facts hold.
 \par
 (a).
 For any $t\in \bbT_n$,  we have
 \par
(Unimodularity \cite[Prop.~4.2]{Nakanishi11a})
\index{unimodularity}
 \begin{align}
 \label{2eq:unimod1}
 |C_{t}|=|G_{t}|\in \{1, -1\},
 \end{align}
 \par
(Second duality  \cite[Eq.~(3.11)]{Nakanishi11a}) \index{duality!second}
 \begin{align}
 \label{2eq:dual1}
 D^{-1}G_{t}^{T}DC_{t}=I,
 \end{align}
 or, equivalently,
 \begin{align}
 \label{2eq:dual12}
 D^{-1}C_{t}^{T}DG_{t}=I,
 \end{align}
 where 
 $D$ is a common skew-symmetrizer of $\bfB$.
 \par
 (b).
For any $t,t'\in \bbT_n$ that are  $k$-adjacent,
we have
\par
 (Mutations \cite[Prop.~1.3]{Nakanishi11a})
 \index{mutation!of $C$- and $G$-matrices}
\begin{align}
\label{2eq:cmut5}
C_{t'}&=C_t(J_k + [\varepsilon_{k;t}B_t]_+^{k \bullet}),
\\
\label{2eq:gmut5}
G_{t'}&=G_t(J_k + [-\varepsilon_{k;t}B_t]_+^{\bullet k}).
\end{align}
\end{prop}

Since we need some details of the proof later,
we present it below.

\begin{proof}
We first prove (b).
By the column sign-coher\-ence of $C_{t}$,
the tropical sign $\varepsilon_{k;t}$ is defined
for any $k$.
Then,
we have
\begin{align}
   [-\varepsilon_{k;t} C_t]_+^{\bullet k}=O.
 \end{align}
We set $\varepsilon=\varepsilon_{k;t}$
in the $\varepsilon$-expressions  of mutations \eqref{2eq:cmut3} and \eqref{2eq:gmut2}.
The resulting relations
 are
  \eqref{2eq:cmut5} and \eqref{2eq:gmut5}.

\par

We prove (a) by the induction on $t$
along  $\bbT_n$ starting from $t_0$.
Let us introduce  the following statement:
\begin{itemize}
\item[$(a)_d$.] The claim (a) holds for any $t \in \bbT_n$ such that $d(t_0,t)= d$.
\end{itemize}

First, $(a)_0$ holds because $C_{t_0}=G_{t_0}=I$.
Next, suppose that $(a)_d$ holds for some $d$.
Let $t, t'\in \bbT_n$ be vertices  that are $k$-adjacent such that
  $d(t_0,t)=d$  and $d(t_0,t')=d+1$.
Let us  temporally set the matrices in  \eqref{2eq:cmut5} and \eqref{2eq:gmut5} as
\begin{align*}
P=J_k + [\varepsilon_{k;t}B_t]_+^{k \bullet},
\quad
Q=J_k + [-\varepsilon_{k;t}B_t]_+^{\bullet k}.
\end{align*}
Then, by Lemma \ref{2lem:square1}, we have
\begin{align}
\label{2eq:detP1}
P^2=I,
\quad
|P|=|Q|=-1.
\end{align}
Thus, by the induction assumption  $ |C_{t}|=|G_{t}|=\pm1$, we have
$ |C_{t'}|=|G_{t'}|=\mp1$.
Also, by
\eqref{2eq:DBD1},
we have
\begin{align}
DPD^{-1}=Q^T.
\end{align}
Using the induction assumption $ D^{-1}(G_{t})^{T}DC_{t}=I$, we have
\begin{align}
 D^{-1}G_{t'}^{T}DC_{t'}=
 D^{-1} (Q^T G_{t}^{T})D(C_tP)
 =P(D^{-1}G_{t}^{T}DC_t)P=P^2=I.
\end{align}
Therefore, $(a)_{d+1}$ holds.
\end{proof}

\begin{rem} 
\label{2rem:sign1}
For the later use,
we record the following facts in the above proof.
\begin{itemize}
\item
To prove (b), we only use the sign-coherence of $C_t$ for $t$ therein.
\item
To prove $(a)_d \Longrightarrow (a)_{d+1}$,
we only use the sign-coherence of $C_t$ such that $d(t_0,t)=d$.
\end{itemize}
\end{rem}

The following  expressions of formulas
\eqref{2eq:cmut5} and \eqref{2eq:gmut5}
by $c$- and $g$-vectors are also useful.
\begin{prop}
Suppose that the sign-coherence conjecture holds.
Then, we have
\begin{align}
\label{2eq:cmut7}
\bfc_{i;t'}&=
\begin{cases}
\displaystyle
-\bfc_{k;t}
& i=k,
\\
\bfc_{i;t}
+[\varepsilon_{k;t}b_{ki;t}]_+\bfc_{k;t}
& i \neq k,
\end{cases}
\\
\label{2eq:gmut7}
\bfg_{i;t'}&=
\begin{cases}
\displaystyle
-\bfg_{k;t}
+
\sum_{j=1}^n
[-\varepsilon_{k;t}b_{j k;t}]_+\bfg_{j;t}
& i=k,
\\
\bfg_{i;t}
& i \neq k.
\end{cases}
\end{align}
\end{prop}

We note that
the matrices $C_t$ and $B_{t_0}$ in turn determine $B_t$.
\begin{prop}[{\cite[Eq.~(2.9)]{Nakanishi11a}}]
\label{2prop:cbc1}
Suppose that the sign-coherence conjecture holds.
Then, the following equality holds:
\begin{align}
\label{2eq:cbc1}
C_t^T(DB_{t_0})C_t
=
DB_t.
\end{align}
\end{prop}
\begin{proof}
By  the dualities \eqref{2eq:dual0} and \eqref{2eq:dual12}, we have
\begin{align}
C_t^TDB_{t_0}C_t
=C_t^TDG_tB_t
=D(D^{-1}C_t^TDG_t)B_t
=
DB_t.
\end{align}
\end{proof}

\subsection{Dual mutations and third duality}
\label{2subsec:dual1}
Let  $\bfB=\{ B_t\}_{t\in \bbT_n}$ be any $B$-pattern,
and let $t_0$ be a given initial vertex.
Since we now vary the initial vertex below,
we  write the corresponding $C$- and $G$-matrices 
as $C_{t}^{t_0}$ and $G_{t}^{t_0}$
 as in Remark
\ref{2rem:notation1}.

We note that the mutation \eqref{2eq:bmut2}
is compatible with the matrix transpose.
Namely,
if $B'=\mu_k(B)$, then $B'{}^T=\mu_k(B^T)$.
Therefore, ${\bfB}^T:=\{ B_t^T\}_{t\in \bbT_n}$
is also a $B$-pattern,
which is called the \emph{transpose of $\bfB$}. \index{$B$-pattern!transpose of}
We temporarily write the associated $C$- and $G$-matrices
as $\tilde{C}_{t}^{t_0}$, $\tilde{G}_{t}^{t_0}$.

We borrow the following notation from \cite{Fujiwara18}:
For any matrix $M$, $\varepsilon_{k\bullet}(M)$ (resp.~$\varepsilon_{\bullet k}(M)$)
stands for
the common sign of the $k$th row (resp.~ $k$th column) assuming that it is
either positive or negative.
For example, the tropical sign $\varepsilon_{i;t}$ in
Definition \ref{2defn:tropicalsign1}
is written as $\varepsilon_{\bullet i}(C^{t_0}_t)$.

\begin{prop}
\label{2prop:dual2}
Suppose that the sign-coherence conjecture holds.
Then, the following facts hold:
\par
(a). For any $t_0,t \in \bbT_n$,  we have
\par
 (Third duality \cite[Eq.~(1.13)]{Nakanishi11a})
 \index{duality!third}
\begin{align} 
\label{2eq:dual2}
C_{t}^{t_0} &= (\tilde{G}_{t_0}^{t})^T,\\
\label{2eq:dual3}
G_{t}^{t_0} &= (\tilde{C}_{t_0}^{t})^T.
\end{align}
In particular, each $G$-matrix $G_{t}^{t_0}$ is
row sign-coherent, and the equality 
\begin{align}
\varepsilon_{k\bullet}(G_{t}^{t_0}) 
= \varepsilon_{\bullet k}(\tilde{C}_{t_0}^{t})
\end{align}
holds.
\par
(b).
For any $t_0,t_1,t\in \bbT_n$ such that $t_0$ and $t_1$ are  $k$-adjacent,
we have
\par
 (Dual mutations \cite[Prop.~1.4, Eq.~(4.2)]{Nakanishi11a}, \cite[Prop.~3.6]{Fujiwara18})
 \index{mutation!dual}
\begin{align}
\label{2eq:cmut6}
C^{t_1}_{t}&=(J_k + [-{\varepsilon}_{k\bullet}(G^{t_0}_t )B_{t_0}]_+^{k \bullet})C^{t_0}_t,
\\
\label{2eq:gmut6}
G^{t_1}_{t}&=(J_k + [{\varepsilon}_{k\bullet}(G^{t_0}_t )B_{t_0}]_+^{\bullet k})G^{t_0}_t.
\end{align}
Here, ${\varepsilon}_{k\bullet}(G^{t_0}_t )$
is well-defined thanks to (a).

\end{prop}

Again, we need some details of the proof later.
Furthermore, the statement and the proof in \cite{Nakanishi11a} are a little indirect
in the present context.
Thus, we give a   more transparent proof, 
still in the  spirit of  \cite{Nakanishi11a}.

\begin{proof}
We prove the claims by the induction on the distance
$d(t_0,t)$ in $\bbT_n$, where we vary  $t_0$ for (a) and $t$ for (b).

Let us introduce  the following claims:
\begin{itemize}
\item[$(a)_d$.] The claim (a) holds for any $t_0,t \in \bbT_n$ such that $d(t_0,t)= d$.
\item[$(b)_d$.] The claim (b) holds for any $t_0,t_1,t \in \bbT_n$ such that $d(t_0,t)= d$
and $t_0$ and $t_1$ are  $k$-adjacent.
\end{itemize}

We prove the claims in the following order,
\begin{align}
\label{2eq:ind1}
(a)_0 \Longrightarrow 
(b)_0 \Longrightarrow 
(a)_1 \Longrightarrow 
(b)_1 \Longrightarrow 
(a)_2 \Longrightarrow 
\cdots,
\end{align}
assuming all  preceding claims.

First, we show that $(a)_0$ and $(b)_0$ hold.
The claim $(a)_0$, where $t_0=t$, trivially holds, because all relevant matrices are the identity matrix $I$.
Let us prove $(b)_0$, where $t_0=t$.
By assumption,  $t_0$ and $t_1$ are $k$-adjacent.
Then, by Proposition \ref{2prop:dual1} (b) and \eqref{2eq:bmut3},
we have
\begin{alignat}{3}
C^{t_1}_{t_0}&=C^{t_1}_{t_1}(J_k + [B_{t_1}]^{k\bullet}_+)
=(J_k + [-B_{t_0}]^{k\bullet}_+)C^{t_0}_{t_0},
\\
G^{t_1}_{t_0}&=G^{t_1}_{t_1}(J_k + [-B_{t_1}]^{\bullet k}_+)=(J_k + [B_{t_0}]^{\bullet k}_+)
G^{t_0}_{t_0}.
\end{alignat}
Thus,  $(b)_0$ holds.

Next, assuming the claims in \eqref{2eq:ind1} up to  $(b)_d$, we show $(a)_{d+1}$.
Assume $d(t_0,t)=d$ and $d(t_1,t)=d+1$ in \eqref{2eq:cmut6}. Take the transpose of \eqref{2eq:cmut6}.
Then, we apply \eqref{2eq:dual2},
 which is valid thanks to $(a)_{d}$,
  to its right hand side.
We have
\begin{align}
(C^{t_1}_t)^T
= \tilde{G}^t_{t_0}(J_k + [-\varepsilon_{\bullet k}(\tilde{C}^t_{t_0})
B_{t_0}^T]_+^{\bullet k}
).
\end{align}
The right hand side is $\tilde{G}^t_{t_1}$ by \eqref{2eq:gmut5}. Thus, we have 
\eqref{2eq:cmut6}. The other case is similar.

Finally, assuming the claims in \eqref{2eq:ind1} up to  $(a)_{d+1}$, we show $(b)_{d+1}$.
This is the  nontrivial part.
It is enough to concentrate on the following situation.
Let $t_0$, $t_1$, $t$, $t'\in \bbT_n$ be vertices  such that
$d(t_0,t)=d$, $d(t_0,t')=d+1$,
where
$t_1$ and $t_0$ are $k$-adjacent,
while
$t$ and $t'$ are $\ell$-adjacent.
They are depicted as follows, where
$t_1$ could be between $t_0$ and $t$.

\begin{align*}
\begin{picture}(120,0)(40,35)
\put(40,40){\circle*{3}}
\put(70,40){\circle*{3}}
\put(130,40){\circle*{3}}
\put(160,40){\circle*{3}}
\put(40,40){\line(1,0){120}}
\put(38,30){$t_1$}
\put(68,30){$t_0$}
\put(128,30){$t$}
\put(158,30){$t'$}
\put(54,45){$k$}
\put(144,45){$\ell$}
\end{picture}
\end{align*}

By  \eqref{2eq:cmut5} and \eqref{2eq:cmut6} with $(b)_{d}$,
we obtain
\begin{align}
\label{2eq:cc1}
\begin{split}
C^{t_1}_{t'}&=C^{t_1}_t(J_{\ell} + [\varepsilon_{\bullet \ell}(C^{t_1}_t)B_t]_+^{\ell \bullet})\\
&=(J_k + [-{\varepsilon}_{k\bullet}(G^{t_0}_t )B_{t_0}]_+^{k \bullet})C^{t_0}_t
(J_{\ell} + [\varepsilon_{\bullet \ell}(C^{t_1}_t)B_t]_+^{\ell \bullet}).
\end{split}
\end{align}
By  $(a)_{d+1}$, 
the sign ${\varepsilon}_{k\bullet}(G^{t_0}_{t'} )$ is well-defined.
Let us temporally assume the identity:
\begin{align}
\label{2eq:Jid1}
\begin{split}
&\
(J_k + [-{\varepsilon}_{k\bullet}(G^{t_0}_t )B_{t_0}]_+^{k \bullet})C^{t_0}_t
(J_{\ell} + [\varepsilon_{\bullet \ell}(C^{t_1}_t)B_t]_+^{\ell \bullet})\\
=&\
(J_k + [-{\varepsilon}_{k\bullet}(G^{t_0}_{t'} )B_{t_0}]_+^{k \bullet})C^{t_0}_t
(J_{\ell} + [\varepsilon_{\bullet \ell}(C^{t_0}_t)B_t]_+^{\ell \bullet}).
\end{split}
\end{align}
Then, the right hand side of \eqref{2eq:cc1}
reduces to
$(J_k + [-{\varepsilon}_{k\bullet}(G^{t_0}_{t'} )B_{t_0}]_+^{k \bullet})C^{t_0}_{t'}$,
which is the desired result for \eqref{2eq:cmut6} of $(b)_{d+1}$.
Similarly, from the identity
\begin{align}
\label{2eq:Jid2}
\begin{split}
&\
(J_k + [{\varepsilon}_{k\bullet}(G^{t_0}_t )B_{t_0}]_+^{\bullet k })G^{t_0}_t
(J_{\ell} + [-\varepsilon_{\bullet \ell}(C^{t_1}_t)B_t]_+^{ \bullet\ell})\\
=&\
(J_k + [{\varepsilon}_{k\bullet}(G^{t_0}_{t'} )B_{t_0}]_+^{ \bullet k})G^{t_0}_t
(J_{\ell} + [-\varepsilon_{\bullet \ell}(C^{t_0}_t)B_t]_+^{ \bullet\ell}),
\end{split}
\end{align}
we obtain  \eqref{2eq:gmut6} of $(b)_{d+1}$.
It remains to prove \eqref{2eq:Jid1} and \eqref{2eq:Jid2}.
To do it, we need to know a precise relation between signs
${\varepsilon}_{k\bullet}(G^{t_0}_t )$ and ${\varepsilon}_{k\bullet}(G^{t_0}_{t'} )$,
and also $\varepsilon_{\bullet \ell}(C^{t_1}_t)$ and $\varepsilon_{\bullet \ell}(C^{t_0}_t)$.

The following facts hold:
\begin{itemize}
\item[(i).]
By \eqref{2eq:gmut5}, $G^{t_0}_t $ and $G^{t_0}_{t'} $ differ only in the $\ell$th column.
\item[(ii).]
By $(b)_d$, $C^{t_0}_t$ and $C^{t_1}_t$ differ only in the $k$th row.
\item[(iii).]
By the duality \eqref{2eq:dual1}, 
the $k$th row of $G^{t_0}_t $ has only  a unique component $\varepsilon$ at $\ell$th column
if and only
the  $\ell$th column of $C^{t_0}_t $ has a unique nonzero component $\varepsilon'$ at $k$th row.
Furthermore, if it happens, $\varepsilon=\varepsilon'\in \{1, -1\}$,
thanks to the unimodularity \eqref{2eq:unimod1}.

\end{itemize}

Now we consider two cases.

\emph{Case 1.} Suppose that the  $k$th row of $G^{t_0}_t $ has a nonzero component at some column
other than the $\ell$th column. Then, from Facts (i)--(iii) above and the row and column sign-coherence
of $G$- and $C$-matrices in $(a)_{d+1}$, we conclude that
\begin{align}
{\varepsilon}_{k\bullet}(G^{t_0}_{t'} )={\varepsilon}_{k\bullet}(G^{t_0}_t ),
\quad
\varepsilon_{\bullet \ell}(C^{t_1}_t)= \varepsilon_{\bullet \ell}(C^{t_0}_t).
\end{align}
Therefore, \eqref{2eq:Jid1} and \eqref{2eq:Jid2} hold trivially.

\emph{Case 2.} 
Suppose that the $k$th row of $G^{t_0}_t $ has  only one nonzero component that is  at the $\ell$th column.
Let us write this component as $\varepsilon\in \{ 1, -1\}$.
As mentioned in Fact (iii), 
$\varepsilon$ coincides with 
the only one nonzero component 
in the  $\ell$th column of $C^{t_0}_t $ that is at the $k$th row.
Then, by \eqref{2eq:gmut5}  and $(b)_d$, we have
\begin{align}
{\varepsilon}_{k\bullet}(G^{t_0}_{t'} )=-\varepsilon,
\quad
{\varepsilon}_{k\bullet}(G^{t_0}_{t} )=\varepsilon,
\quad
\varepsilon_{\bullet \ell}(C^{t_1}_t)=-\varepsilon,
\quad
\varepsilon_{\bullet \ell}(C^{t_0}_t)=\varepsilon.
\end{align}
Let us prove \eqref{2eq:Jid1}, which is now written as
\begin{align}
\label{2eq:Jid3}
\begin{split}
&\
(J_k + [-{\varepsilon}B_{t_0}]_+^{k \bullet})C^{t_0}_t
(J_{\ell} + [-{\varepsilon}B_t]_+^{\ell \bullet})\\
=&\
(J_k + [{\varepsilon}B_{t_0}]_+^{k \bullet})C^{t_0}_t
(J_{\ell} + [\varepsilon B_t]_+^{\ell \bullet}).
\end{split}
\end{align}
By \eqref{2eq:Nsq1},
it is rewritten as
\begin{align}
\label{2eq:Jid4}
\begin{split}
&\
(J_k + [{\varepsilon}B_{t_0}]_+^{k \bullet})(J_k + [-{\varepsilon}B_{t_0}]_+^{k \bullet})C^{t_0}_t
\\
=&\
C^{t_0}_t
(J_{\ell} + [\varepsilon B_t]_+^{\ell \bullet})(J_{\ell} + [-{\varepsilon}B_t]_+^{\ell \bullet}).
\end{split}
\end{align}
After some manipulation,
it  reduces to
\begin{align}
\label{2eq:bc3}
(B_{t_0} C_{t}^{t_0})^{k \bullet }
=C_t^{t_0}(B_t)^{\ell\bullet}.
\end{align}
By the assumption, we have
$(C_t^{t_0})^{\bullet \ell}=(G_{t}^{t_0})^{k\bullet}=
\varepsilon E_{k\ell}$, where $E_{k\ell}$ is the matrix whose entries
are zero except for the $(k,\ell)$-entry that is 1.
Then,
the right hand side of \eqref{2eq:bc3} can be rewritten as
\begin{align}
(C_t^{t_0})^{\bullet \ell} B_t=\varepsilon E_{k\ell} B_t
=(G_{t}^{t_0})^{k\bullet} B_t
=(G_{t}^{t_0}B_t)^{k\bullet} .
\end{align}
Thus, \eqref{2eq:bc3} reduces to
the duality \eqref{2eq:dual0}.
Therefore, \eqref{2eq:Jid1} is proved.
Similarly, skipping details a little, \eqref{2eq:Jid2} reduces to
\begin{align}
\label{2eq:bc2}
(B_{t_0})^{ \bullet k} G_{t}^{t_0}
=(G_t^{t_0}B_t)^{\bullet \ell}.
\end{align}
Then, by repeating a similar argument,
again \eqref{2eq:bc2} reduces to the duality \eqref{2eq:dual0}.
Therefore, \eqref{2eq:Jid2} is proved.
\end{proof}

\begin{rem} 
\label{2rem:sign2}
For the later use we record the following facts in the above proof.
\begin{itemize}
\item
To prove $(a)_d$, $(b)_d$ $\Longrightarrow (a)_{d+1}$, we only use the sign-coherence of $\tilde{C}_{t_0}^{t}$ such that $d(t_0,t)=d$.
\item
To prove $(b)_d$, $(a)_{d+1}$ $\Longrightarrow (b)_{d+1}$,
we only use the sign-coherence of $C_{t}^{t_0}$,
$C_{t}^{t_1}$,  $\tilde{C}_{t_0}^{t}$, $\tilde{C}_{t_0}^{t'}$ such that $d(t_0,t)=d$,
$d(t_0,t')=d+1$, $d(t_1,t)=d\pm1$.
\end{itemize}
\end{rem}

\begin{rem}
Combining three dualities \eqref{2eq:dual0},  \eqref{2eq:dual1},  \eqref{2eq:dual2}, 
we have various relations among $C$- and $G$-matrices therein.
For example, by \eqref{2eq:dual0} and \eqref{2eq:dual2}, we have
\begin{align}
B_{t_0}C_{t}^{t_0}=(B_{t}^T \tilde{C}_{t_0}^{t})^T,
\quad
G_{t}^{t_0}B_{t_0}=(\tilde{G}_{t_0}^{t}B_{t}^T )^T.
\end{align}
\end{rem}

The $\hat{C}$-matrices in \eqref{2eq:gbc2} mutate as a ``hybrid'' of 
$C$- and $G$-matrices in the following sense.

\begin{prop}
Suppose that the sign-coherence conjecture holds.
Let $\hat{C}_{t}^{t_0}=B_{t_0}{C}_{t}^{t_0}$ and $\hat{C}_{t}^{t_1}=B_{t_1}{C}_{t}^{t_1}$.
Then, 
for any $t_0,t_1,t,t'\in \bbT_n$ such that $t_0$ and $t_1$,
and also, $t$ and $t'$ are  $k$-adjacent, respectively,
we have
\begin{align}
\label{2eq:cmut12}
\hat{C}_{t'}^{t_0}&=\hat{C}_t^{t_0}(J_k + [{\varepsilon}_{\bullet k}(C^{t_0}_t )B_t]_+^{k \bullet}),
\\
\label{2eq:gmut12}
\hat{C}^{t_1}_{t}&=(J_k + [{\varepsilon}_{k\bullet}(G^{t_0}_t )B_{t_0}]_+^{\bullet k})\hat{C}^{t_0}_t.
\end{align}

\end{prop}
\begin{proof}
The first formula is obtained from  \eqref{2eq:gbc2}
by multiplying $B_{t_0}$ from the left.
The second one is obtained from  \eqref{2eq:gmut6}
by multiplying $B_{t}$ from the right.
\end{proof}

\subsection{Principal extension of $B$-patterns}
\label{2subsec:principal1}

For any $B$-pattern $\bfB$ of rank $n$ and a given initial vertex $t_0\in \bbT_n$,
we introduce a $2n \times 2n$  skew-symmetrizable matrix 
\begin{align}
\label{2eq:principleB1}
\overline{B}_{t_0}=
\begin{pmatrix}
B_{t_0} & -I\\
I & O\\
\end{pmatrix},
\end{align}
which is called the \emph{principal extension of $B_{t_0}$}.
It is important that $\overline{B}_{t_0}$ is nonsingular.
Choose any point $\overline{t}_0\in \bbT_{2n}$.
We naturally identify $\bbT_n$ as a subtree of $\bbT_{2n}$
by identifying $t_0$ with $\overline{t}_0$, and also
the edges with
 labels $k=1,\dots,n$ in $\bbT_n$ and $\bbT_{2n}$.
Consider a $B$-pattern 
$\overline\bfB=\{ \overline{B}_t\}_{t\in \bbT_{2n}}$ of rank $2n$ generated
by $\overline{B}_{t_0}$,
which is called \emph{the principal extension of $\bfB$ with
the initial vertex $t_0$}. \index{$B$-pattern!principal extension}\index{principal extension!of a $B$-pattern}

The following   stability properties  have  repeatedly appeared
 in the  literatures in various forms.
Here, we quote from \cite{Fujiwara18}.

\begin{prop}[{e.g., \cite[Theorem~4.2, Remark~4.4]{Fujiwara18}}]
\label{2prop:principal1}
Let $C_t$, $G_t$, $F_{i;t}(\bfy)$ be $C$- and $G$-matrices, and $F$-polynomials
of $\bfB$ with the initial point $t_0\in \bbT_n$.
Let $\overline{C}_t$, $\overline{G}_t$, $\overline{F}_{i;t}(\overline{\bfy})$ be $C$- and $G$-matrices, and $F$-polynomials
of $\overline{\bfB}$ with the same  initial point $t_0 \in \bbT_n \subset \bbT_{2n}$,
where $\overline{\bfy}=(y_1,\dots,y_{2n})$.
Then, the following relations hold for any $t\in \bbT_n\subset \bbT_{2n}$:
\begin{align}
\label{2eq:Cbar1}
\overline{C}_t&=
\begin{pmatrix}
C_t & Z_t\\
O & I
\end{pmatrix},
\\
\label{2eq:Gbar1}
\overline{G}_t&=
\begin{pmatrix}
G_t & O\\
O & I
\end{pmatrix},
\\
\overline{F}_{i;t}(\overline\bfy)
&=
\begin{cases}F_{i;t}(\bfy)
& i=1,\dots,n,
\\
1
& i=n+1,\dots, 2n,
\end{cases}
\end{align}
where $Z_t$ is a certain matrix. Moreover,
$Z_t=O$ if the  sign-coherence of $C$-matrices for $\bfB$ holds.
\end{prop}

These equalities can be  proved by the induction on $t$ along $\bbT_n$ from $t_0$
 \cite[Theorem~4.2]{Fujiwara18}.

By Proposition \ref{2prop:principal1}, any information 
on $C_t$, $G_t$, $F_{i;t}(\bfy)$
can be extracted from the one on 
$\overline{C}_t$, $\overline{G}_t$, $\overline{F}_{i;t}(\overline{\bfy})$.

\subsection{$G$-cones and $G$-fans}

From now on,
we study  the geometric aspect of $G$-matrices.

We recall some  notions from convex geometry.

\begin{defn}[Cone]
Let $M\simeq \bbZ^n$ be a lattice of rank $n$.
Let $M_{\bbR}=M\otimes_{\bbZ} \bbR\simeq \bbR^n$ be the extension to an $n$-dimensional
vector space over $\bbR$.
Let $N=\mathrm{Hom}_{\bbZ}(M,\bbZ)$
and $N_{\bbR}=\mathrm{Hom}_{\bbR}(M_{\bbR},\bbR)$
be their duals.
Let $\langle \ , \ \rangle$ be  the canonical paring
$N_{\bbR}\times M_{\bbR} \rightarrow \bbR$.
\begin{itemize}
\item 
For given elements $a_1,\dots,a_r\in M \subset M_{\bbR}$,
the subset of $M_{\bbR}$ defined by 
\begin{align}
\sigma(a_1,\dots,a_r):=
\bbR_{\geq 0} a_1 + \cdots \bbR_{\geq 0} a_r
\end{align} 
is called a \emph{convex rational polyhedral cone}, or simply, a \emph{cone} 
generated by $a_1,\dots,a_r$.
\index{cone}\index{cone!convex rational polyhedral}
The set $\sigma(\emptyset):=\{0\}$ is also a cone.
\item
A cone $\sigma$ is said to be \emph{strongly convex\/} if \index{cone!strongly convex}
$\sigma\cap (-\sigma)=\{0\}$.
\item
A cone is \emph{simplicial\/} if it is generated \index{cone!simplicial}
by $\bbZ$-linearly independent elements in $M$.
Thus, a simplicial cone is strongly convex.
\item
A cone is \emph{nonsingular\/} if it is generated \index{cone!nonsingular}
by a subset of a $\bbZ$-basis in $M$.
Thus, a nonsingular cone is simplicial.
\item The \emph{dimension of a cone $\sigma$} is the dimension of the  subspace 
of $M_{\bbR}$ spanned by $\sigma$.
\item The \emph{dual cone\/} $\sigma^{\vee}$ of a cone $\sigma$ is  defined by \index{cone!dual}
$$
\sigma^{\vee}=\{ u\in N_{\bbR} \mid \langle u, a \rangle \geq 0\ \text{for any $a\in \sigma$})\},
$$
which is a cone in $N_{\bbR}$.
\item
A subset $\tau$ of a cone $\sigma$ is called a \emph{face\/} of $\sigma$ \index{face (of a cone)}
if there is some $u\in \sigma^{\vee}$ such that $\tau=\sigma \cap u^{\perp}$,
where $u^{\perp}=\{ a\in M_{\bbR} \mid \langle u, a \rangle = 0\}$.
In particular, 
for a simplicial cone $\sigma$
with a $\bbZ$-linearly independent
 generating set $S$,
any face of $\sigma$
 is
a cone generated by a subset of $S$.
\end{itemize}
\end{defn}

\begin{ex}
For $M=\bbZ^2$ and $M_{\bbR}=\bbR^2$, $\sigma(\bfe_1)$ and 
$\sigma(\bfe_1, \bfe_2)$ are strongly convex cones,
while $\sigma(\bfe_1,-\bfe_1)=\bbR\bfe_1$ and
$\sigma(\bfe_1,-\bfe_1,\bfe_2,-\bfe_2)=\bbR^2$ are cones
that are not strongly convex.
The cone $\sigma=\sigma(\bfe_1, \bfe_2)$
has four faces $\sigma$, $\sigma(\bfe_1)$, $\sigma(\bfe_2)$, $\{0\}$,
while
the cone $\sigma=\sigma(\bfe_1,-\bfe_1)$
has the unique face, which is $\sigma$.
\end{ex}

\begin{defn}[Fan]
\index{fan}
\par
\begin{itemize}
\item
A \emph{fan\/} $\Delta$ in $M_{\bbR}\simeq \bbR^n$ is a nonempty (possibly infinite) set of strongly convex  cones in $M_{\bbR}$
satisfying the following conditions:
\begin{itemize}
\item[(i).]
If $\tau$ is a face of $\sigma\in \Delta$, then  $\tau\in\Delta$.
\item[(ii).]
If $\sigma, \tau \in \Delta$, then $\sigma\cap \tau \in \Delta$.
\end{itemize}
\item
A subset $|\Delta|=\bigcup_{\sigma\in \Delta} \sigma$ of  $M_{\bbR}$ is called the \emph{support\/}
of a fan $\Delta$. 
\item
A fan $\Delta$ is \emph{simplicial\/} if all cones of $\Delta$ are simplicial. \index{fan!simplicial}
\item
A fan $\Delta$ is \emph{nonsingular\/} if all cones of $\Delta$ are nonsingular. \index{fan!nonsingular}
Thus, a nonsingular fan is simplicial.
\item
A fan $\Delta$ is \emph{complete\/} if $|\Delta|=M_{\bbR}$. \index{fan!complete}
\end{itemize}
\end{defn}

We apply these geometric notions to $G$-matrices following \cite{Reading11,Reading12}.
In the rest of this section
\emph{we temporarily assume the sign-coherence of $C$-matrices\/}
and study some basic properties of $G$-cones.

Let $\bfB$ be any $B$-pattern,
and let
 $\bfG^{t_0}=\bfG$ be the $G$-pattern of $\bfB$ with
a given initial point $t_0$.

\begin{defn}[$G$-cone]
For each  $G$-matrix $G_t$,
the cone spanned by its $g$-vectors
in $\bbR^n$
\begin{align}
\sigma(G_t):=\sigma(\bfg_{1;t}, \dots, \bfg_{n;t})
\end{align}
is called a \emph{$G$-cone}. \index{cone!$G$-}\index{$G$-cone}
The interior of $\sigma(G_t)$ is denoted by
$\sigma^{\circ}(G_t)$.
Each  face of $\sigma(G_t)$ spanned by
$\bfg_{j;t}$'s excluding  the $i$th $g$-vector
$\bfg_{i;t}$
is denoted by $\sigma_i(G_t)$ for $i=1,\dots,n$.
\end{defn}

By the unimodularity
\eqref{2eq:unimod1},
every $G$-cone $\sigma(G_t)$
is a nonsingular  cone,
and its faces $\sigma_i(G_t)$ ($i=1,\dots,n$)
are cones of codimension one.

From now on, let us especially take a common skew-symmetrizer
$D$ of the $B$-pattern $\bfB$
in the following form:
\begin{align}
\label{2eq:Dmat1}
D=\mathrm{diag}( {\d}_1^{-1} 
\dots, {\d}_n ^{-1}),
\end{align}
where $\d_1$, \dots, $\d_n$ are positive integers.
(This seemingly awkward choice is suitable to the formulation
of scattering diagrams  in \cite{Gross14} we use later. See Example \ref{2ex:seed1}.)
Accordingly, we introduce the following nonstandard inner product 
$(\ , \ )_{D}$
in $\bbR^n$:
\begin{align}
\label{2eq:Dinner1}
(\bfu, \bfv)_{D}=\bfu^T D \bfv
\quad
(\bfu,\bfv\in \bbR^n).
\end{align}
Below, for a given $B$-pattern $\bfB$,
we  always assume this inner product  $(\cdot, \cdot)_{D}$
 in $\bbR^n$.
For $\bfn\neq \bfzero \in \bbR^n$, $\bfn^{\perp}$
is a hyperplane defined by 
\begin{align}
\bfn^{\perp}:=\{ \bfv \in \bbR^n \mid (\bfn,\bfv )_{D}=0\}.
\end{align}
We say that a normal vector $\bfn$ of a face
$\sigma_i(G_t)$ of a cone $\sigma(G_t)$  is \emph{inward for $\sigma(G_t)$}
if $(\bfn, \bfg_{i;t})_{D}>0$.

In this geometrical setting,
 the duality \eqref{2eq:dual1}/\eqref{2eq:dual12} can be rephrased in the following manner.

\begin{prop}
\label{2prop:duality3}
Suppose that the sign-coherence conjecture holds.
Then,   the following dualities of vectors in $\bbR^n$ hold:
\begin{align}
\label{2eq:dcg1}
(\d_i \bfc_{i;t},\bfg_{j;t})_D&=\delta_{ij},
\\
\label{2eq:cdg1}
(\bfc_{i;t},\delta_j \bfg_{j;t})_D&=\delta_{ij}.
\end{align}
In particular, each  $c$-vector $\bfc_{i;t}$ is a normal vector of
the face  $\sigma_i(G_t)$ of  the cone $\sigma(G_t)$
with respect to the  inner product $(\cdot, \cdot)_{D}$.
Moreover, it is inward for $\sigma(G_t)$.
\end{prop}

Let $\Delta(\bfG^{t_0})=\Delta(\bfG)$ be the set of all faces of all $G$-cones 
$\sigma(G_t)$ ($t\in \bbT_n$) in the $G$-pattern $\bfG^{t_0}$.
Let $t, t'\in \bbT_n$ be vertices that are $k$-adjacent.
Then, by \eqref{2eq:gmut7},
two cones $\sigma(G_t)$ and $\sigma(G_{t'})$ intersect
 each other
in their common face $\sigma_k(G_t)=\sigma_k(G{_{t'}})$ with codimension one.
However, for a cone $\sigma(G_{t''})$ obtained  from $\sigma(G_t)$ after several mutations,
it is not clear at all that $\sigma(G_{t''})$ intersects  $\sigma(G_t)$ only in their common face so that the set $\Delta(\bfG)$ forms a fan.
It was conjectured by  \cite[Conj.~3.14]{Reading11} that this always happens.
Then, it was shown to be true by \cite{Reading12}
under the sign-coherence conjecture.

\begin{thm}[{\cite[Theorem~8.7]{Reading12}}]
\label{2thm:Gfan1}
Suppose that the sign-coherence conjecture holds.
Then,
for any $B$-pattern $\bfB$ and a given initial vertex $t_0$,
the set of cones  $\Delta(\bfG^{t_0})$ is a fan in $\bbR^n$.
\end{thm}

In \cite{Reading12} this was proved in a more general setting of \emph{mutation fans\/} therein.
We provide  a self-contained proof of Theorem \ref{2thm:Gfan1}
in 
 Section \ref{2subsec:proofGfan1} without referring to mutation fans,
 because the tools and  techniques we use therein are also relevant to 
 our main purpose.
 
The resulting fan $\Delta(\bfG^{t_0})$,
which is nonsingular by \eqref{2eq:unimod1}, is
 called the  \emph{$g$-vector fan\/} \index{$g$-vector fan|see{$G$-fan}}
in  \cite{Reading11}.
To simplify a little,
we call it the \emph{$G$-fan\/} of $\bfB$ (or $\bfSigma$) with the initial vertex $t_0$. \index{fan!$G$-}\index{$G$-fan}

Let $\#\bfG^{t_0}$ be the number of \emph{ distinct\/} $G$-matrices in $\bfG^{t_0}$,
which can be finite or infinite.

\begin{prop}
\label{2prop:finite1}
Suppose that the sign-coherence conjecture holds.
Then, the $G$-fan  $\Delta(\bfG^{t_0})$ is complete
if and only if $\#\bfG^{t_0}$ is finite.
\end{prop}
\begin{proof}
This is true by the following topological reason.
For each subset $T$ of $\bbT_n$, we attach a
subset $\Delta_T$ of $\Delta(\bfG^{t_0})$ consisting of the faces
of  all $\sigma(G_t)$ with  $t\in T$.
Let $|\Delta_T|^c$ be the complement of the support $|\Delta_T|$
in $\bbR^n$.
Let $S^{n-1}$ be the unit sphere in $\bbR^n$.
Let $D_T=|\Delta_T|^c\cap S^{n-1}$.
For $T=\{t_0\}$, $D_T$ is topologically a $(n-1)$-dimensional open disk.
We consider a strongly increasing sequence of subsets  of $\bbT_n$,
$T_1=\{t_0\}\subsetneq T_2\subsetneq T_3 \subsetneq \cdots$
with $|T_s|=s$.
Correspondingly, we have a weakly decreasing  sequence of subsets of $S^{n-1}$,
$D_{T_1}\supset D_{T_2}\supset D_{T_3}\supset \cdots$.
Then, $D_{T_s}$ is a disjoint union of a finite number of open disks,
such that
at each process of increasing $T_{s}$ to $T_{s+1}$, one of the followings occurs
 (by adding a $G$-cone):
\begin{itemize}
\item
$D_{T_{s+1}}=D_{T_s}$.
\item
An open disk shrinks.
\item
An open disk splits into two open disks.
\item
An open disk disappears.
\end{itemize}
If $\#\bfG^{t_0}$ is finite, the decreasing process of $D_{T_s}$ should terminate at some $s$.
This occurs only when $D_{T_s}=\emptyset$. Thus, $\Delta(\bfG^{t_0})$ is complete.
If $\#\bfG^{t_0}$  is infinite, the process never terminates.
Then, when $T_s$ converges to  $\bbT_n$, $D_{T_s}$ converges to the (possibly infinite) 
 union of points and (not open, but not necessarily closed) disks.
Therefore, $|\Delta(\bfG^{t_0})|^c$ is nonempty.
\end{proof}

\begin{rem}

We say that a cluster pattern $\bfSigma$ is \emph{of finite type\/} \index{cluster pattern!of finite type}
if there are only finitely many distinct seeds in $\bfSigma$.
It is known, for example, by
\cite[Theorem~5.2]{Nakanishi19},
that
a cluster pattern $\bfSigma$ is of finite type
if and only if $\#\bfG^{t_0}$ is finite for $\bfG^{t_0}$ of 
 $\bfSigma$.
Then, combining it with Proposition \ref{2prop:finite1},
we have that
a cluster pattern $\bfSigma$ is of finite type
if and only if the $G$-fan  $\Delta(\bfG^{t_0})$ is complete
for $\bfG^{t_0}$ of 
 $\bfSigma$.
 Also, we note that the proof of  \cite[Theorem~5.2]{Nakanishi19}
 is based on the sign-coherence of $C$-matrices and the Laurent positivity.
\end{rem}

\subsection{Rank 2 examples of $G$-fans}
\label{2subsec:rank2G1}

In the rank 2 case, one can  confirm Theorem
\ref{2thm:Gfan1} and Proposition \ref{2prop:finite1}
by explicitly calculating $G$-matrices.

Let us arrange a $B$-pattern of rank 2 in the following way:
\begin{align}
\label{2eq:seedmutseq3}
\cdots
\
\mathop{\leftrightarrow}^{\mu_{2}}
\
B_{t_{-2}}
\
\mathop{\leftrightarrow}^{\mu_{1}}
\
B_{t_{-1}}
\
\mathop{\leftrightarrow}^{\mu_{2}}
\
B_{t_{0}}
\
\mathop{\leftrightarrow}^{\mu_{1}}
\
B_{t_{1}}
\
\mathop{\leftrightarrow}^{\mu_{2}}
\
B_{t_{2}}
\
\mathop{\leftrightarrow}^{\mu_{1}}
\
\cdots.
\end{align}
We take $B_{t_{0}}$ as the initial exchange matrix.
Let us present the associated $G$-fan $\Delta(\bfG)$
in  $\bbR^2$ explicitly.
Below the type of a $B$-pattern refers to the type of 
the Cartan matrix $A(B)=(a_{ij})$ associated with  $B=B_{t_0}$ in the convention of \cite{Kac90}, where
\begin{align}
a_{ij}=
\begin{cases}
2& i=j,\\
-|b_{ij}| & i\neq j.
\end{cases}
\end{align}

(I). Finite type.
The $G$-fan $\Delta(\bfG)$ is complete
as stated in Proposition \ref{2prop:finite1}.

(a). Type $A_2$: Let 
\begin{align}
\label{2eq:B1}
B_{t_{0}}=
\begin{pmatrix}
0 & -1 \\
1 & 0\\
\end{pmatrix}.
\end{align}
Along the mutation sequence  $\mu_1$, $\mu_2$, $\cdots$, 
the $G$-matrices are explicitly calculated  as
\begin{align}
\label{2eq:A2G1}
\begin{split}
G_{t_{0}}&=
\begin{pmatrix}
1 & 0 \\
0 & 1\\
\end{pmatrix},
\quad
G_{t_{1}}=
\begin{pmatrix}
-1& 0 \\
0 & 1\\
\end{pmatrix},
\quad
G_{t_{2}}=
\begin{pmatrix}
-1& 0 \\
0 & -1\\
\end{pmatrix},
\\
G_{t_{3}}&=
\begin{pmatrix}
1& 0\\
-1 & -1\\
\end{pmatrix},
\quad
G_{t_{4}}=
\begin{pmatrix}
1& 1\\
-1 & 0\\
\end{pmatrix},
\quad
G_{t_{5}}=
\begin{pmatrix}
0& 1\\
1 & 0\\
\end{pmatrix}.
\end{split}
\end{align}
Both $G_{t_{5}}$ and $G_{t_{0}}$ define the same cone,
corresponding to
 the  celebrated \emph{pentagon periodicity}. \index{pentagon periodicity}
The $G$-fan $\Delta(\bfG)$ is depicted in Figure \ref{2fig:cone1} (a).

(b).  Type $B_2$: Let
\begin{align}
B_{t_{0}}=
\begin{pmatrix}
0 & -1 \\
2 & 0\\
\end{pmatrix}.
\end{align}
We start from the initial $G$-matrix $G_{t_{0}}=I$.
Along the mutation sequence  $\mu_1$, $\mu_2$, $\cdots$, 
the following $g$-vectors appear in this order, showing  periodicity:
\begin{align}
\begin{pmatrix}
-1 \\
0 \\
\end{pmatrix},
\begin{pmatrix}
 0 \\
 -1\\
\end{pmatrix},
\begin{pmatrix}
1\\
-2 \\
\end{pmatrix},
\begin{pmatrix}
 1 \\
 -1\\
\end{pmatrix},
\begin{pmatrix}
1\\
0 \\
\end{pmatrix},
\begin{pmatrix}
0\\
 1\\
\end{pmatrix}.
\end{align}
The $G$-fan $\Delta(\bfG)$ is depicted in Figure \ref{2fig:cone1} (b).

(c). Type $G_2$: Let 
\begin{align}
B_{t_0}=
\begin{pmatrix}
0 & -1 \\
3 & 0\\
\end{pmatrix}.
\end{align}
We start from the initial $G$-matrix $G_{t_0}=I$.
Along the mutation sequence  $\mu_1$, $\mu_2$, $\cdots$, 
the following $g$-vectors appear in this order, showing periodicity:
\begin{align}
\begin{pmatrix}
-1 \\
0 \\
\end{pmatrix},
\begin{pmatrix}
 0 \\
 -1\\
\end{pmatrix},
\begin{pmatrix}
1\\
-3 \\
\end{pmatrix},
\begin{pmatrix}
1\\
-2 \\
\end{pmatrix},
\begin{pmatrix}
2\\
-3 \\
\end{pmatrix},
\begin{pmatrix}
 1 \\
 -1\\
\end{pmatrix},
\begin{pmatrix}
1\\
0 \\
\end{pmatrix},
\begin{pmatrix}
0\\
 1\\
\end{pmatrix}.
\end{align}
The $G$-fan $\Delta(\bfG)$ is depicted in Figure \ref{2fig:cone1} (c).

\begin{figure}
\centering
\leavevmode
\begin{xy}
0;/r1.2mm/:,
(0,-14)*{\text{(a) $A_2$}},
(5, 5)*{\text{\small $G_{t_{0}}$}},
(-5, 5)*{\text{\small $G_{t_{1}}$}},
(-5, -5)*{\text{\small $G_{t_{2}}$}},
(4, -8)*{\text{\small $G_{t_{3}}$}},
(7.5, -3)*{\text{\small $G_{t_{4}}$}},
(0,0)="A",
\ar "A"+(0,0); "A"+(10,0)
\ar "A"+(0,0); "A"+(0,10)
\ar@{-} "A"+(0,0); "A"+(-10,0)
\ar@{-} "A"+(0,0); "A"+(0,-10)
\ar@{-} "A"+(0,0); "A"+(10,-10)
\end{xy}
\hskip40pt
\begin{xy}
0;/r1.2mm/:,
(0,-14)*{\text{(b) $B_2$}},
(0,0)="A"
\ar "A"+(0,0); "A"+(10,0)
\ar "A"+(0,0); "A"+(0,10)
\ar@{-} "A"+(0,0); "A"+(-10,0)
\ar@{-} "A"+(0,0); "A"+(0,-10)
\ar@{-} "A"+(0,0); "A"+(5,-10)
\ar@{-} "A"+(0,0); "A"+(10,-10)
\end{xy}
\hskip40pt
\begin{xy}
0;/r1.2mm/:,
(0,-14)*{\text{(c) $G_2$}},
(0,0)="A"
\ar "A"+(0,0); "A"+(10,0)
\ar "A"+(0,0); "A"+(0,10)
\ar@{-} "A"+(0,0); "A"+(-10,0)
\ar@{-} "A"+(0,0); "A"+(0,-10)
\ar@{-} "A"+(0,0); "A"+(3.33,-10)
\ar@{-} "A"+(0,0); "A"+(5,-10)
\ar@{-} "A"+(0,0); "A"+(6.66,-10)
\ar@{-} "A"+(0,0); "A"+(10,-10)
\end{xy}
\vskip10pt
\leavevmode
\begin{xy}
0;/r1.2mm/:,
(0,-14)*{\text{(d) $A_1^{(1)}$}},
(8.7,-10)*{\cdot},
(9.2,-10)*{\cdot},
(10,-8.7)*{\cdot},
(10,-9.2)*{\cdot},
(0,0)="A"
\ar "A"+(0,0); "A"+(10,0)
\ar "A"+(0,0); "A"+(0,10)
\ar@{-} "A"+(0,0); "A"+(-10,0)
\ar@{-} "A"+(0,0); "A"+(0,-10)
\ar@{-} "A"+(0,0); "A"+(5,-10)
\ar@{-} "A"+(0,0); "A"+(6.66,-10)
\ar@{-} "A"+(0,0); "A"+(7.5,-10)
\ar@{-} "A"+(0,0); "A"+(8,-10)
\ar@{-} "A"+(0,0); "A"+(10,-5)
\ar@{-} "A"+(0,0); "A"+(10,-6.66)
\ar@{-} "A"+(0,0); "A"+(10,-7.5)
\ar@{-} "A"+(0,0); "A"+(10,-8)

\ar@{--} "A"+(0,0); "A"+(10,-10)
\end{xy}
\hskip40pt
\begin{xy}
0;/r1.2mm/:,
(0,-14)*{\text{(e) $A_2^{(2)}$}},
(4.1,-10)*{\cdot},
(4.6,-10)*{\cdot},
(5.5,-10)*{\cdot},
(6,-10)*{\cdot},
(0,0)="A"
\ar "A"+(0,0); "A"+(10,0)
\ar "A"+(0,0); "A"+(0,10)
\ar@{-} "A"+(0,0); "A"+(-10,0)
\ar@{-} "A"+(0,0); "A"+(0,-10)
\ar@{-} "A"+(0,0); "A"+(2.5,-10)
\ar@{-} "A"+(0,0); "A"+(3.3,-10)
\ar@{-} "A"+(0,0); "A"+(3.75,-10)
\ar@{-} "A"+(0,0); "A"+(10,-10)
\ar@{-} "A"+(0,0); "A"+(7.5,-10)
\ar@{-} "A"+(0,0); "A"+(6.6,-10)
\ar@{--} "A"+(0,0); "A"+(5,-10)
\end{xy}
\hskip39pt
\begin{xy}
0;/r1.2mm/:,
(0,-14)*{\text{(f) non-affine}},
(0,0)="A"
\ar@{-} (5,-10); (6.3,-8.7)
\ar@{-} (3,-10); (5.5,-7.5)
\ar@{-} (2.4,-8.6); (4.5,-6.5)
\ar@{-} (1.8,-7.2); (3.8,-5.2)
\ar@{-} (1.6,-5.4); (3,-4)
\ar "A"+(0,0); "A"+(10,0)
\ar "A"+(0,0); "A"+(0,10)
\ar@{-} "A"+(0,0); "A"+(-10,0)
\ar@{-} "A"+(0,0); "A"+(0,-10)
\ar@{-} "A"+(0,0); "A"+(2,-10)
\ar@{-} "A"+(0,0); "A"+(2.5,-10)
\ar@{-} "A"+(0,0); "A"+(2.66,-10)
\ar@{-} "A"+(0,0); "A"+(2.72,-10)
\ar@{-} "A"+(0,0); "A"+(2.76,-10)
\ar@{-} "A"+(0,0); "A"+(10,-10)
\ar@{-} "A"+(0,0); "A"+(8,-10)
\ar@{-} "A"+(0,0); "A"+(7.5,-10)
\ar@{-} "A"+(0,0); "A"+(7.33,-10)
\ar@{-} "A"+(0,0); "A"+(7.27,-10)
\ar@{-} "A"+(0,0); "A"+(7.23,-10)
\end{xy}
\caption{$G$-fans of rank 2.}
\label{2fig:cone1}
\end{figure}
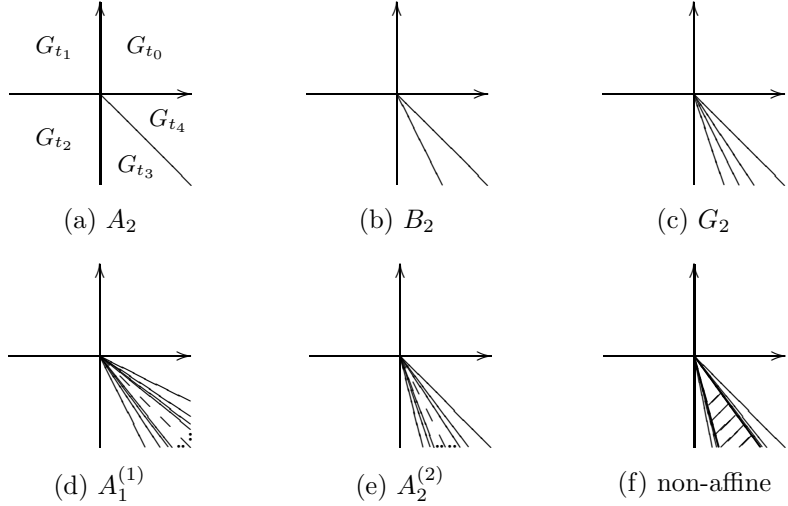

(II). Infinite type. 
The $G$-fan $\Delta(\bfG)$ is not complete.

(d). Type $A_1^{(1)}$: Let
\begin{align}
B_{t_0}=
\begin{pmatrix}
0 & -2 \\
2 & 0\\
\end{pmatrix}.
\end{align}
We start from the initial $G$-matrix $G_{t_0}=I$.
Along the mutation sequence  $\mu_1$, $\mu_2$, $\cdots$, 
the following $g$-vectors appear in this order:
\begin{align}
\begin{pmatrix}
-1 \\
0 \\
\end{pmatrix},
\begin{pmatrix}
 0 \\
 -1\\
\end{pmatrix},
\begin{pmatrix}
1\\
-2 \\
\end{pmatrix},
\begin{pmatrix}
2\\
-3 \\
\end{pmatrix},
\begin{pmatrix}
3\\
-4 \\
\end{pmatrix},
\cdots.
\end{align}
On the other hand,
along the mutation sequence  $\mu_2$, $\mu_1$, $\cdots$, 
the following $g$-vectors appear in this order:
\begin{align}
\begin{pmatrix}
2 \\
-1 \\
\end{pmatrix},
\begin{pmatrix}
 3 \\
 -2\\
\end{pmatrix},
\begin{pmatrix}
4\\
-3 \\
\end{pmatrix},
\begin{pmatrix}
5\\
-4 \\
\end{pmatrix},
\cdots.
\end{align}
The $G$-fan $\Delta(\bfG)$ is depicted in Figure \ref{2fig:cone1} (d).
It covers the region excluding the half line
$\mathbb{R}_{+}(1,-1)$.

(e). Type $A_{2}^{(2)}$: Let
\begin{align}
B_{t_0}=
\begin{pmatrix}
0 & -1 \\
4 & 0\\
\end{pmatrix}.
\end{align}
We start from the initial $G$-matrix $G_{t_0}=I$.
Along the mutation sequence  $\mu_1$, $\mu_2$, $\cdots$, 
the following $g$-vectors appear in this order:
\begin{align}
\begin{pmatrix}
-1 \\
0 \\
\end{pmatrix},
\begin{pmatrix}
 0 \\
 -1\\
\end{pmatrix},
\begin{pmatrix}
1\\
-4 \\
\end{pmatrix},
\begin{pmatrix}
1\\
-3 \\
\end{pmatrix},
\begin{pmatrix}
3\\
-8 \\
\end{pmatrix},
\begin{pmatrix}
2\\
-5 \\
\end{pmatrix},
\begin{pmatrix}
5\\
-12 \\
\end{pmatrix},
\begin{pmatrix}
3\\
-7 \\
\end{pmatrix},
\cdots.
\end{align}
On the other hand,
along the mutation sequence  $\mu_2$, $\mu_1$, $\cdots$, 
the following $g$-vectors appear in this order:
\begin{align}
\begin{pmatrix}
1 \\
-1 \\
\end{pmatrix},
\begin{pmatrix}
 3 \\
 -4\\
\end{pmatrix},
\begin{pmatrix}
2\\
-3 \\
\end{pmatrix},
\begin{pmatrix}
5\\
-8 \\
\end{pmatrix},
\begin{pmatrix}
3\\
-5 \\
\end{pmatrix},
\begin{pmatrix}
7\\
-12 \\
\end{pmatrix},
\cdots.
\end{align}
The $G$-fan $\Delta(\bfG)$ is depicted in Figure \ref{2fig:cone1} (e).
It covers the region excluding the half line
$\mathbb{R}_{+}(1,-2)$.

(f). Non-affine type: Let
\begin{align}
B_{t_0}=
\begin{pmatrix}
0 & -c \\
b & 0\\
\end{pmatrix},
\quad
(b,c>0,\  bc\geq 5).
\end{align}
We start from the initial $G$-matrix $G_{t_0}=I$.
Along the mutation sequence  $\mu_1$, $\mu_2$, $\cdots$, 
the following $g$-vectors appear in this order:
\begin{align}
\bfg_1=
\begin{pmatrix}
-1 \\
0 \\
\end{pmatrix},
\
\bfg_2=
\begin{pmatrix}
 0 \\
 -1\\
\end{pmatrix},
\
\bfg_3,
\
\bfg_4,
\cdots,
\end{align}
where $\bfg_i$'s obey the  recursion,
\begin{align}
\bfg_{i+2}=
\begin{cases}
-\bfg_i + b \bfg_{i+1} & \text{$i$ odd},\\
-\bfg_i + c \bfg_{i+1} & \text{$i$ even}.
\end{cases}
\end{align}
On the other hand,
along the mutation sequence  $\mu_2$, $\mu_1$, $\cdots$
from the initial $G$-matrix,
the following $g$-vectors appear in this order:
\begin{align}
\bfg'_1=
\begin{pmatrix}
c \\
-1 \\
\end{pmatrix},
\
\bfg'_2=
\begin{pmatrix}
 bc-1 \\
 -b\\
\end{pmatrix},
\
\bfg'_3,
\
\bfg'_4,
\cdots,
\end{align}
where $\bfg'_i$'s obey the recursion,
\begin{align}
\bfg'_{i+2}=
\begin{cases}
-\bfg'_i + c \bfg'_{i+1} & \text{$i$ odd},\\
-\bfg'_i + b \bfg'_{i+1} & \text{$i$ even}.
\end{cases}
\end{align}
It is known  \cite[Eqs.~(9.11), (9.12)]{Reading12}
 that,
 in the limit $i\rightarrow \infty$,
the rays  $\sigma(\bfg_i)$ 
and $\sigma(\bfg'_i)$
monotonically converge to
$\sigma(\bfv)$ and $\sigma(\bfv')$, respectively,
where
 \begin{align}
 \label{2eq:ds1}
 \bfv
 =
 \begin{pmatrix}
bc-\sqrt{bc(bc-4)} \\
- 2b\\
\end{pmatrix},
\quad
\bfv'
 =
 \begin{pmatrix}
bc+\sqrt{bc(bc-4)} \\
 -2b\\
\end{pmatrix}.
 \end{align}
The $G$-fan $\Delta(\bfG)$ is depicted in Figure \ref{2fig:cone1} (f).
It covers the region excluding $\sigma( \bfv, \bfv')\setminus \{0 \}$ 
 (the hatched region
in the figure).

\subsection{Piecewise linear isomorphisms between $G$-fans}
\label{2subsec:Gfan1}

In this subsection
we study a canonical bijection between the sets $\Delta(\bfG^{t_0})$
and $\Delta(\bfG^{t_1})$ for any  $t_0, t_1\in\bbT_n$.
The idea  comes from \cite{Gross14},
but
this is independent of their results.
Also, the result of this and the next subsections has considerable overlap
with the results in \cite{Reading12,Reading17},
in which $G$-fans were studied as a part of
 mutation fans.

Let  $\bfB$ be any $B$ pattern.
Let us start with the following consequence of the sign-coherence.

\begin{prop}
\label{2prop:int1}
Suppose that the sign-coherence conjecture holds.
Then, for any $t_0,t\in \bbT_n$, the cone $\sigma(G_t^{t_0})$ intersects 
the hyperplanes $\bfe_i^{\perp}$ $(i=1,\dots,n)$ only in the boundary
of $\sigma(G_t^{t_0})$.
\end{prop}
\begin{proof}
Suppose that a cone $\sigma(G_t^{t_0})$
intersects with $\bfe_i^{\perp}$ for some $i$ in its interior 
$\sigma^{\circ}(G_t^{t_0})$.
This means that the $i$th row of the $G$-matrix $G_t^{t_0}$
has both positive and negative components.
This contradicts the row sign-coherence of $G_t^{t_0}$
in Proposition \ref{2prop:dual2}.
\end{proof}

Let $t_0,t_1\in \bbT_n$ be  vertices that are $k$-adjacent.
Recall the formula
\eqref{2eq:gmut6},
which is valid under the assumption of the sign-coherence of $C$-matrices.
It is  written in the vector form as follows.
\begin{align}
\label{2eq:gmut8}
\bfg_{i;t}^{t_1}=(J_k + [{\varepsilon}_{k\bullet}(G^{t_0}_t )B_{t_0}]_+^{\bullet k})\bfg^{t_0}_{i;t},
\quad
(i=1,\dots,n).
\end{align}
Let us introduce the half spaces of $\bbR^n$ as
\begin{align}
\label{2eq:Rk1}
\bbR^n_{k,+}=\{ \bfv\in {\bbR}^n \mid v_k\geq 0\},
 \quad
\bbR^n_{k,-}=\{ \bfv\in {\bbR}^n \mid v_k\leq 0\}.
\end{align}
The matrix $(J_k + [{\varepsilon}_{k\bullet}(G^{t_0}_t )B_{t_0}]_+^{\bullet k})$
in \eqref{2eq:gmut8}
is regarded as a piecewise linear map
\begin{align}
\label{2eq:vphi1}
\begin{matrix}
\varphi_{t_0}^{t_1}:& \bbR^n&\rightarrow &\bbR^n\\
&\bfv &\mapsto&
\begin{cases}
(J_k + [B_{t_0}]_+^{\bullet k})\bfv 
& \bfv\in \bbR^n_{k,+},
\\
(J_k + [-B_{t_0}]_+^{\bullet k})\bfv 
&\bfv\in \bbR^n_{k,-},
\end{cases}
\end{matrix}
\end{align}
where we note that $\varphi_{t_0}^{t_1}(\bfv)=\bfv$ if $v_k=0$.
We have, by \eqref{2eq:gmut8} and \eqref{2eq:vphi1},
for any $t\in \bbT_n$ and $i=1,\dots,k$,
\begin{align}
\label{2eq:phigg1}
\varphi_{t_0}^{t_1}(\bfg^{t_0}_{i;t})=
\bfg_{i;t}^{t_1}.
\end{align}
The map $\varphi_{t_1}^{t_0}$ is also defined in the same way
by replacing $B_{t_0}$ with $B_{t_1}$.

\begin{prop} The following relation holds:
\label{2prop:varphiinv1}
\begin{align}
\label{2eq:phi1}
\varphi_{t_1}^{t_0}\circ \varphi_{t_0}^{t_1}&=
\varphi_{t_0}^{t_1}\circ \varphi_{t_1}^{t_0}
=\mathrm{id}.
\end{align}
\end{prop}
\begin{proof}
Note that, for $\bfv'=\varphi_{t_0}^{t_1}(\bfv)$,
we have
$v'_k=-v_k$ due to $b_{kk}^{t_0}=0$.
We also recall that $(B_{t_1})^{\bullet k}=
-(B_{t_0})^{\bullet k}$.
Then, for $\bfv\in \bbR^n_{k,+}$,
\begin{align}
\begin{split}
\varphi_{t_1}^{t_0}\circ \varphi_{t_0}^{t_1}(\bfv)
&=
(J_k + [-B_{t_1}]_+^{\bullet k})
(J_k + [B_{t_0}]_+^{\bullet k})\bfv\\
&=(J_k + [B_{t_0}]_+^{\bullet k})^2\bfv=
\bfv
\end{split}
\end{align}
thanks to Lemma \ref{2lem:square1}.
The other case is similar.

\end{proof}

\begin{prop}[{cf.~\cite[Prop.~8.13]{Reading12}}]
\label{2prop:Gfan1}
Suppose that the sign-coherence conjecture holds.
Then, we have
\begin{align}
\label{2eq:phi2}
\varphi_{t_0}^{t_1}(\sigma(G_{t}^{t_0}))&={\sigma}(G_{t}^{t_1}),
\quad
(t\in \bbT_n).
\end{align}
Therefore,
the correspondence 
 $\varphi_{t_0}^{t_1}: \bfg^{t_0}_{i;t} \mapsto \bfg^{t_1}_{i;t}$
 gives a bijection
between ${\Delta}(\bfG^{t_0})$ and ${\Delta}(\bfG^{t_1})$
preserving
 the intersection and the inclusion of cones.
\end{prop}

\begin{proof}
\par
We only need to prove
\eqref{2eq:phi2}.
By Proposition \ref{2prop:int1}, each cone $\sigma(G_{t}^{t_0})$
belongs to either $\bbR^n_{k,+}$ or $\bbR^n_{k,-}$,
so that it is \emph{linearly\/} transformed under $\varphi_{t_0}^{t_1}$.
Therefore, by \eqref{2eq:phigg1},
its image is $\sigma(G_{t}^{t_1})$.
\end{proof}

\begin{rem} 
\label{2rem:Gfan1}
For the later use,
we record the following fact in the above proof.

\begin{itemize}
\item
To prove \eqref{2eq:phi2}, we only use the formula
\eqref{2eq:gmut8} for $t_0$, $t_1$ and $t$ therein,
 and the row sign-coherence
of $G_t^{t_0}$  therein.
\end{itemize}
\end{rem}

\begin{ex}
\label{2ex:pl1}
Let us clarify 
Proposition 
\ref{2prop:Gfan1} explicitly for type $A_2$
based on the convention in Section \ref{2subsec:rank2G1}.
The fans ${\Delta}(\bfG^{t_{-1}})$, ${\Delta}(\bfG^{t_{0}})$, ${\Delta}(\bfG^{t_{1}})$
are depicted  in Figure \ref{2fig:Gmut2}.
Also, the piecewise linear maps $\varphi_{t_0}^{t_1}$ and $\varphi_{t_0}^{t_{-1}}$
are given by combining the following linear maps:
\begin{gather}
J_1 + [B_{t_0}]_+^{\bullet 1}=
\begin{pmatrix}
-1 & 0\\
1 & 1
\end{pmatrix},
\quad
J_1 + [-B_{t_0}]_+^{\bullet 1}=
\begin{pmatrix}
-1 & 0\\
0 & 1
\end{pmatrix},
\\
J_2 + [B_{t_0}]_+^{\bullet 2}=
\begin{pmatrix}
1 & 0\\
0 & -1
\end{pmatrix},
\quad
J_2 + [-B_{t_0}]_+^{\bullet 2}=
\begin{pmatrix}
1 & 1\\
0 & -1
\end{pmatrix}.
\end{gather}
Then, the correspondence in Proposition
\ref{2prop:Gfan1} can be easily confirmed
in Figure \ref{2fig:Gmut2}.
\end{ex}

\begin{figure}
\centering
\leavevmode
\begin{xy}
(20,-12)*{\text{\small$\varphi_{t_0}^{t_{-1}}$}},
(20,-17)*{\text{$\longleftarrow$}},
(0,-17)*{\text{${\Delta}(\bfG^{t_{-1}})$}},
(12, 5)*{\text{\small $G_{t_{4}}^{t_{-1}}=G_{t_{-1}}^{t_{-1}}$}},
(5, -5)*{\text{\small $G_{t_{0}}^{t_{-1}}$}},
(-5, -5)*{\text{\small $G_{t_{1}}^{t_{-1}}$}},
(-3.5, 8.5)*{\text{\small $G_{t_{3}}^{t_{-1}}$}},
(-8.5, 3.5)*{\text{\small $G_{t_{2}}^{t_{-1}}$}},
(0,0)="A",
\ar "A"+(0,0); "A"+(10,0)
\ar "A"+(0,0); "A"+(0,10)
\ar@{-} "A"+(0,0); "A"+(-10,0)
\ar@{-} "A"+(0,0); "A"+(0,-10)
\ar@{-} "A"+(0,0); "A"+(-10,10)
\end{xy}
\hskip14pt
\begin{xy}
(20,-12)*{\text{\small$\varphi_{t_0}^{t_{1}}$}},
(20,-17)*{\text{$\longrightarrow$}},
(0,-17)*{\text{${\Delta}(\bfG^{t_0})$}},
(5, 5)*{\text{\small $G_{t_{0}}^{t_0}$}},
(-5, 5)*{\text{\small $G_{t_{1}}^{t_0}$}},
(-5, -5)*{\text{\small $G_{t_{2}}^{t_0}$}},
(3.5, -8.5)*{\text{\small $G_{t_{3}}^{t_0}$}},
(8, -3.5)*{\text{\small $G_{t_{4}}^{t_0}$}},
(0,0)="A",
\ar "A"+(0,0); "A"+(10,0)
\ar "A"+(0,0); "A"+(0,10)
\ar@{-} "A"+(0,0); "A"+(-10,0)
\ar@{-} "A"+(0,0); "A"+(0,-10)
\ar@{-} "A"+(0,0); "A"+(10,-10)
\end{xy}
\hskip12pt
\begin{xy}
(0,-17)*{\text{${\Delta}(\bfG^{t_1})$}},
(5, 5)*{\text{\small $G_{t_{1}}^{t_1}$}},
(5, -5)*{\text{\small $G_{t_{2}}^{t_1}$}},
(-5, -5)*{\text{\small $G_{t_{3}}^{t_1}$}},
(-3.5, 8.5)*{\text{\small $G_{t_{0}}^{t_1}$}},
(-8, 3.5)*{\text{\small $G_{t_{4}}^{t_1}$}},
(0,0)="A",
\ar "A"+(0,0); "A"+(10,0)
\ar "A"+(0,0); "A"+(0,10)
\ar@{-} "A"+(0,0); "A"+(-10,0)
\ar@{-} "A"+(0,0); "A"+(0,-10)
\ar@{-} "A"+(0,0); "A"+(-10,10)
\end{xy}
\caption{Bijections between $G$-fans for type $A_2$.}
\label{2fig:Gmut2}
\end{figure}
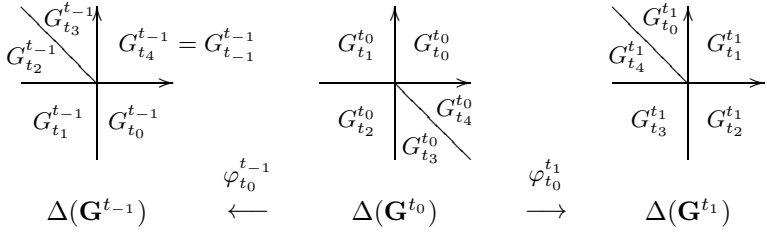

By applying the above isomorphism repeatedly along the tree $\bbT_n$,
we obtain the following important fact.

\begin{prop}
\label{2prop:Gfanphi1}
Suppose that the sign-coherence conjecture holds.
Then,
for any vertices $t_0,t_1\in\bbT_n$,
which are not necessarily adjacent,
 the correspondence 
 $\bfg^{t_0}_{i;t} \mapsto \bfg^{t_1}_{i;t}$
 gives a bijection
between ${\Delta}(\bfG^{t_0})$ and ${\Delta}(\bfG^{t_1})$
preserving
 the intersection and the inclusion
 of cones.

 \end{prop}

For the later use, we also describe the  relation between
$\varphi_{t_0}^{t_1}$ here and 
a certain piecewise linear isomorphism  in \cite{Gross14}.

Again, let $t_0, t_1\in \bbT_n$ be vertices that are $k$-adjacent.
First, for  $B_{t_0}$ and $k$, we introduce a  basis $\bfa_{1},\dots,\bfa_{n}$ of $\bbR^n$ as
\begin{align}
\label{2eq:ff1}
\bfa_{i}=
\begin{cases}
\displaystyle
-\bfe_i + 
\sum_{j=1}^n [-b_{jk;t_0}]_+ \bfe_j
& i = k,\\
\bfe_i
& i\neq k.
\end{cases}
\end{align}

In fact,
by specializing the formula in \eqref{2eq:gmut7} 
with $t=t_0$,
and applying the fact $\varepsilon_{k;t_0}=1$ therein,
we see that
\begin{align}
\label{2eq:gmut9}
\bfa_{i}=\bfg_{i;t_1}^{t_0}.
\end{align}
The change of a basis of $\bbR^n$
from $\bfe_1$, \dots, $\bfe_n$
to $\bfa_1$, \dots, $\bfa_n$
induces the following linear isomorphism
\begin{align}
\label{2eq:eta1}
\begin{matrix}
\eta_{t_0}^{t_1}:& \bbR^n&\rightarrow &\bbR^n\\
&\bfv&\mapsto&  (J_k+[-B_{t_0}]_+^{\bullet k})\bfv
\end{matrix}
 \end{align}
 such that $\eta_{t_0}^{t_1}(\bfa_i)=\bfe_i$.

Next, for the same $B_{t_0}$ and $k$,
we define a piecewise linear map $T_{k;t_0}$ on $\bbR^n$
following  \cite{Gross14}
as
\begin{align}
\label{2eq:Tk1}
\begin{matrix}
T_{k;t_0} :& \bbR^n &\rightarrow &\bbR^n\\
&\bfv &\mapsto
&
\begin{cases}
(I + (B_{t_0})^{\bullet k})\bfv 
& \bfv\in \bbR^n_{k,+},
\\
\bfv 
& \bfv\in \bbR^n_{k,-}.
\end{cases}
\end{matrix}
\end{align}

The following fact appeared in \cite{Muller15, Reading17}.

\begin{prop}[{\cite[Lemma~5.2.1]{Muller15}, \cite[Cor.~4.4]{Reading17}}]
\label{2prop:Gfan2}
 The above maps are related as follows:
\begin{align}
\label{2eq:phi6}
\varphi_{t_0}^{t_1}=
\eta_{t_0}^{t_1}\circ T_{k;t_0}.
\end{align}
\end{prop}

\begin{proof}
By the definitions of these maps,
\begin{align}
\eta_{t_0}^{t_1}\circ T_{k;t_0}
(\bfv)
&=
\begin{cases}
 (J_k+[-B_{t_0}]_+^{\bullet k})(I + (B_{t_0})^{\bullet k})\bfv 
& \bfv\in \bbR^n_{k,+},
\\
 (J_k+[-B_{t_0}]_+^{\bullet k})\bfv 
& \bfv\in \bbR^n_{k,-}.
\end{cases}
\end{align}
Thanks to \eqref{2eq:pos1}, we have
\begin{align*}
 (J_k+[-B_{t_0}]_+^{\bullet k})(I + (B_{t_0})^{\bullet k})
=J_k+ (B_{t_0})^{\bullet k}+[-B_{t_0}]_+^{\bullet k}
=J_k+[B_{t_0}]_+^{\bullet k}.
\end{align*}
\end{proof}

\subsection{Proof of Theorem \ref{2thm:Gfan1}}
\label{2subsec:proofGfan1}

Let us prove Theorem \ref{2thm:Gfan1}
under the assumption of the sign-coherence of $C$-matrices.

We recall that
the condition for $\Delta(\bfG^{t_0})$ to form a fan 
is stated as follows:

\begin{condition}
\label{cond:fan1}
For any  $t,t'\in \bbT_n$,
the cones  $\sigma(G_{t}^{t_0})$
and
$\sigma(G_{t'}^{t_0})$ intersect in  their common face.
\end{condition}

Thanks to  Proposition \ref{2prop:Gfan1},
Condition \ref{cond:fan1} is rephrased as follows,
where we used the bijection between $\Delta(\bfG^{t_0})$ and $\Delta(\bfG^{t})$
therein.

\begin{condition}
\label{2cond:fan1}
For any  $t,t'\in \bbT_n$,
the cones  $\sigma(G_{t'}^{t})$
and
$\sigma(I)$ intersect in  their common face,
where $\sigma(I)=\sigma(\bfe_1,\dots,\bfe_n)$.
\end{condition}

By changing the notation, we may prove Condition \ref{2cond:fan1} for $\sigma(G_{t}^{t_0})$
instead of $\sigma(G_{t'}^{t})$ therein.
Then, the condition can be further rephrased as follows:

\par\noindent
{\bf Claim.}
For any $t_0, t\in \bbT_n$,
if a $g$-vector $\bfg_{i;t}^{t_0}$ is contained in $\sigma(I)$, then $\bfg_{i;t}^{t_0}=\bfe_{\ell}$ for some $\ell$.

\par
Let us prove the claim.
For $J\subset \{1,\dots,n\}$, 
let $J^c=\{1,\dots,n\}\setminus J$ be its compliment,
and let $\sigma_{J}(G_{t}^{t_0})$ be the face of
$\sigma(G_{t}^{t_0})$ generated by $\bfg_{j;t}^{t_0}$'s ($j\in J^c$).
For example, we have  $\sigma_{\{1,\dots,n\}}(G_{t}^{t_0})=\{0\}$,
$\sigma_{\{i\}^c}(G_{t}^{t_0})=\sigma(\bfg_{i;t}^{t_0})$.

Thanks to Proposition  \ref{2prop:duality3}, we have
\begin{align}
\sigma_J(G_{t}^{t_0}) =
\sigma(G_{t}^{t_0}) 
\cap
\Biggl(\,
\bigcap_{j\in J} ({\bfc}_{j;t}^{t_0})^{\perp}
\Biggr).
\end{align}
On the other hand,
since ${\bfc}_{j;t}^{t_0}$ is either positive or negative,
we have
\begin{align}
\sigma(I)\cap
({\bfc}_{j;t}^{t_0})^{\perp}
=
\sigma(I)\cap
\bigcap_{k\in K[j]} {(\bfe_k)}^{\perp}
=\sigma_{K[j]}(I),
\end{align}
where $K[j]$ is the set of $k$ such that
 the $k$th component of ${\bfc}_{j;t}^{t_0}$ is nonzero.
Therefore, we have
\begin{align}
\label{2eq:sigma1}
\sigma_J(G_{t}^{t_0}) \cap \sigma(I)=
 \sigma(G_{t}^{t_0})\cap
\Biggl(\,
\bigcap_{j\in J}\sigma_{K[j]}(I)
\Biggr)
=
 \sigma(G_{t}^{t_0})\cap \sigma_{K[J]}(I),
\end{align}
where $K[J]$ is the set of $k$ such that
 the $k$th component of ${\bfc}_{j;t}^{t_0}$ is nonzero
 for some $j\in J$.
Now we set  $J=\{i\}^c$.
Then, we have 
 $\sigma_{\{i\}^c}(G_{t}^{t_0})=\sigma(\bfg_{i;t}^{t_0})$ and
 $K[J]=\{1,\dots,n\}$ or $\{\ell\}^c$ for some $\ell$.
 In the former case,
we see from \eqref{2eq:sigma1} that
$\sigma(\bfg_{i;t}^{t_0})$
  intersects $\sigma(I)$ trivially.
In the latter case,
 we have
 \begin{align}
 \sigma(\bfg_{i;t}^{t_0})\cap \sigma(I)
 =
  \sigma(G_{t}^{t_0})\cap\sigma(\bfe_{\ell}).
 \end{align}
 Therefore,
if    $\sigma(\bfg_{i;t}^{t_0})\cap
\sigma(I)\neq \{ 0\}$, we have   $\sigma(\bfg_{i;t}^{t_0})=\sigma(\bfe_{\ell})$.
  Then, by the unimodularity \eqref{2eq:unimod1},
we have   $\bfg_{i;t}^{t_0}=\bfe_{\ell}$ as desired.
   This
   completes the proof of Theorem \ref{2thm:Gfan1}.

\subsection{Properties of $\hat{c}$-vectors}

In this subsection we observe some properties
of $\hat{c}$-vectors that  are relevant to 
 the scattering diagram method.

The following properties do not depend on
 the sign-coherence conjecture.
\begin{prop}
The following facts hold:
\par
(a).
Any $\hat{c}$-vector $\hat{\bfc}_{i;t}$ is orthogonal to the 
 ${c}$-vector ${\bfc}_{i;t}$.
 \par
 (b).
Any $\hat{c}$-vector $\hat{\bfc}_{i;t}$ 
is in the subspace in $\bbR^n$ spanned by
$\sigma_i(G_t)$.
\end{prop}
\begin{proof}
(a).
Since $DB_{t_0}$ is skew-symmetric, we have
\begin{align}
\label{2eq:cc2}
({\bfc}_{i;t}, \hat{\bfc}_{i;t})_{D}
=
({\bfc}_{i;t})^{T} DB_{t_0} {\bfc}_{i;t}
=0.
\end{align}
\par
(b).
By the second expression of $\hat{c}$-vectors
in \eqref{2eq:chat1},
we have
\begin{align}
\label{2eq:chat2}
\hat{\bfc}_{i;t}
=
\sum_{j=1}^n
b_{ji;t}\bfg_{j;t},
\end{align}
where we recall that $b_{ii;t}=0$.
\end{proof}

\begin{defn}
\label{2defn:normalc1}
Suppose that the sign-coherence conjecture holds.
Let $\bfc_{i;t}$ be a $c$-vector, and
let $\varepsilon_{i;t}$ be its tropical sign
in Definition \ref{2defn:tropicalsign1}.
We call the vectors $\bfc_{i;t}^+=\varepsilon_{i;t}\bfc_{i;t}$
and $\hat{\bfc}_{i;t}^+=\varepsilon_{i;t}\hat{\bfc}_{i;t}$
a \emph{ $c^+$-vector\/} and a \emph{ $\hat{c}^+$-vector},
respectively.
\end{defn}

\begin{rem}
Any $c^+$-vector $\bfc_{i;t}^+$ is   \emph{positive\/} due to the sign-coherence
assumption,
while a $\hat{c}^+$-vector $\hat{\bfc}_{i;t}^+$ is not  so at all in general.
\end{rem}

The following fact is related to the \emph{incoming/outgoing property\/}
of walls appearing later in Definition \ref{2defn:wall1}.
\begin{prop}
\label{2prop:outgoing1}
Suppose that the sign-coherence conjecture holds.
Then, if
\begin{align}
\label{2eq:csigma1}
\hat{\bfc}_{i;t}^+\in
\sigma_i(G_t),
\end{align}
then ${\bfc}_{i;t}^+=\bfe_{\ell}$ for some $\ell$.
\end{prop}
\begin{proof}
In view of \eqref{2eq:chat2},
the condition \eqref{2eq:csigma1}
is equivalent to the following condition:
\begin{align}
\label{2eq:ebpos1}
\text{
The vector
$\varepsilon_{i;t}\bfb_{i;t}$ is either positive  or zero.
}
\end{align}
Suppose that \eqref{2eq:ebpos1} happens, where we replace $i$ with $k$ therein
for the sake of convenience.
Let $t'$ be the vertex that is
$k$-adjacent to $t$.
Then, the mutation \eqref{2eq:gmut7} is simplified by \eqref{2eq:ebpos1}
as
\begin{align}
\label{2eq:gmut10}
\bfg_{i;t'}&=
 \begin{cases}
 \displaystyle
 -\bfg_{k;t}
 &
 i= k,
 \\
 \bfg_{i;t}  &
 i \neq k.
 \end{cases}
 \end{align}
 Now we recall the row sign-coherence of $G$-matrices
 in Proposition \ref{2prop:dual2}.
 Then, 
 \eqref{2eq:gmut10} happens only when
 the following condition is satisfied:
 \begin{itemize}
 \item
 For each $\ell=1,\dots,n$,
 if the $\ell$th element of $\bfg_{k;t}$ is nonzero,
 then the $\ell$th element of other $\bfg_{i;t}$ ($i\neq k$) vanishes.
 \end{itemize}
Thanks to the unimodularity \eqref{2eq:unimod1} of $G_t$,
one can sharpen the condition as follows:
 \begin{itemize}
 \item
 There is some $\ell$ such that 
 $\bfg_{k;t}=\pm\bfe_{\ell}$,
 and the $\ell$th element of other $\bfg_{i;t}$ ($i\neq k$) vanishes.
 \end{itemize}
 Then, by the duality \eqref{2eq:dual1},
 $\bfc_{k;t}=\pm\bfe_{\ell}$.
 Thus,  $\bfc_{k;t}^+=\bfe_{\ell}$.
\end{proof}

\newpage
\section{Scattering diagrams and $C$- and $G$-matrices}

In this section we present a proof of the  sign-coherence conjecture
based on the results on scattering diagrams in \cite{Gross14}.
At the same time we also establish the  correspondence
between a $G$-fan and the support of a scattering diagram.

\subsection{Supports of scattering diagrams}

To write down a precise definition of scattering diagrams in \cite{Gross14} requires considerable pages,
and it will be postponed in Section \ref{2sec:scattering1}.
Here we summarize the minimal information and results  \cite{Gross14}
that are needed  to prove the sign-coherence conjecture.

Informally speaking,
for a certain vector space  $M_{\bbR} \simeq\bbR^n$,
 a \emph{scattering diagram\/} $\frakD=\{\bfw_{\lambda}=(\frakd_{\lambda}, f_{{\lambda}})\}_{\lambda\in \Lambda}$ in $M_{\bbR}$
is a finite or countably infinite
collection of  \emph{walls},
where $\Lambda$ is an index set,
and
each wall $\bfw_{\lambda}=(\frakd_{\lambda}, f_{{\lambda}})$ is a pair such that
$\frakd_{\lambda}$ is a (not necessarily strongly convex) cone of codimension one in $M_{\bbR}$
called the \emph{support} of $\bfw_{\lambda}$,
and $f_{{\lambda}}$ is a certain formal power series called the \emph{wall function} of $\bfw_{\lambda}$.
The union of the supports of walls
\begin{align}
\mathrm{Supp}(\frakD)=\bigcup_{\lambda\in \Lambda}
\frakd_{\lambda}
\subset M_{\bbR}
\end{align}
is called the \emph{support of a scattering diagram $\frakD$.}

For any \emph{nonsingular\/} skew-symmetrizable  matrix $B$,
one can construct a  scattering diagram $\frakD_{\fraks}$
called a \emph{cluster scattering diagram},
where $\fraks$ is a \emph{seed data\/} in \cite{Gross14} corresponding to $B$.
There is a canonical  isomorphism $\phi_{\fraks}:M_{\bbR} \buildrel {\sim} \over {\rightarrow}\bbR^n$
such that the image of  $\mathrm{Supp}(\frakD_{\fraks})$ in $ \bbR^n$ only depends on $B$.
We write this image as $\frakS(B)$,
and we call it  the \emph{support diagram\/} \index{support diagram} for $B$, for simplicity.
A connected component of $\bbR^n\setminus\frakS(B)$
is called a \emph{chamber\/} of $\frakS(B)$. \index{chamber}

From the viewpoint of cluster patterns,
the supports and the wall functions of   walls
 contain the information of the tropical  and the nontropical part
 of a cluster pattern,
 respectively.
 Since  at this moment we are interested only in the tropical part, i.e., $C$- and $G$-matrices,
 we may safely forget wall functions, and concentrate on the support diagram
 $\frakS(B)$.

Let $D$ be a skew-symmetrizer of $B$ in the form
\eqref{2eq:Dmat1},
and let
$(\ , \ )_{D}$  be the inner product in $\bbR^n$ defined in \eqref{2eq:Dinner1}.
The following facts are direct consequences of
the definition/theorem of the scattering diagram $\frakD_{\fraks}$ in \cite[Theorems 1.12 \&1.28]{Gross14}, which will appear later as
Theorem \ref{2thm:scat1}.

\begin{prop}[{\cite[Theorems 1.12 \&1.28]{Gross14}}]
\label{2prop:SD1}
For any nonsingular skew-symmetrizable matrix $B$,
the support diagram $\frakS(B)$ has the following properties.
\par
(a). The support $\frakd_{\lambda}$ of each wall of $\frakS(B)$ has a positive normal vector
$\bfn_{\lambda}$ with respect to the inner product $(\ , \ )_{D}$.
\par
(b). 
Each hyperplane $\bfe_i^{\perp}$  ($i=1,\dots, n$)
with respect to the inner product $(\ , \ )_{D}$
is a subset of $\frakS(B)$.
\par
(c). 
The interior $\sigma^{\circ}(I)$ of the 
cone $\sigma(I)=\sigma(\bfe_1,\dots, \bfe_n)$ is a \emph{chamber\/} of 
$\frakS(B)$.
In other words,  $\sigma^{\circ}(I)$ does not intersect $\frakS(B)$,
 while
 its boundary belongs to $\frakS(B)$.
(This is a consequence of (a) and (b).)
\end{prop}

Also, the following property of the support diagram $\frakS(B)$
 is  obtained from {\cite[Theorem 1.24]{Gross14}},
 which will appear later as
Theorem \ref{2thm:scat2},
together with 
Proposition \ref{2prop:Gfan2}.
See the explanation after Theorem \ref{2thm:scat2}
for details.

\begin{prop}[{\cite[Theorem 1.24]{Gross14},
\cite[Lemma~5.2.1]{Muller15},\cite[Cor.~4.4]{Reading17}}]
\label{2prop:SD2}
Let $\bfB$ be any nonsingular $B$-pattern.
Let $t_0, t_1\in \bbT_n$ be vertices that are $k$-adjacent.
Let
$\varphi_{t_0}^{t_1}$
 be
 the map
defined in \eqref{2eq:vphi1}.
Then, the support diagrams $\frakS(B_{t_0})$
and $\frakS(B_{t_1})$ are related by
\begin{align}
\varphi_{t_0}^{t_1}(\frakS(B_{t_0}))&=\frakS(B_{t_1}).
\end{align}
\end{prop}

This is clearly a  parallel result for $G$-fans in Proposition
\ref{2prop:Gfan1}, which was proved
under the assumption of the sign-coherence of $C$-matrices.

\subsection{Proof of sign-coherence conjecture}
Let us present a proof of Conjecture \ref{2conj:Csign1}
relying on Propositions \ref{2prop:SD1} and \ref{2prop:SD2}.
Our strategy is to establish both
\begin{itemize}
\item
 the identification of
the interior  $\sigma^{\circ}(G^{t_0}_t)$ of each $G$-cone
as a chamber in  $\frakS(B_{t_0})$,
\item
the column sign-coherence of $C$-matrices $C_t^{t_0}$,
\end{itemize}
 \emph{simultaneously\/} by the induction 
 on the distance $d(t_0, t)$ in  $\bbT_n$,
 where both $t_0$ and $t$ vary.

Now we prove the following theorem.
See \cite[Theorem~4.10]{Reading17} for a closely related result for mutation fans.
\begin{thm}
\label{2thm:SCprove1}
Let $\bfB=\{B_t\}_{t\in \bbT_n}$ be any nonsingular $B$-pattern.
Then, for any  $t_0,t\in\bbT_n$, the following facts hold:
\par
(a). The set
$\sigma^{\circ}(G_t^{t_0})$ is a chamber of $\frakS(B_{t_0})$.
\par
(b). Each $c$-vector $\bfc_{i;t}^{t_0}$
is a normal vector of  the support of some wall of $\frakS(B_{t_0})$.
\par
(c). The column sign-coherence holds for the $C$-matrix $C_t^{t_0}$.
\end{thm}

\begin{proof}
As in Section \ref{2subsec:dual1},
we introduce the transpose $\bfB^T=\{(B_t)^T\}_{t\in \bbT_n}$ of $\bfB$.
Let $\tilde{C}_{t}^{t_0}$, $\tilde{G}_{t}^{t_0}$
be  $C$- and $G$-matrices of $\bfB^T$.

Consider the following claims $(a)_d$, $(b)_d$, $(c)_d$ for $d=0,1,2$, \dots.
\par
$(a)_d$.
For any $t_0, t\in \bbT_n$ such that $d(t_0,t)=d$, the following facts hold:
\begin{align}
 \label{2eq:unimod11}
 |C_{t}^{t_0}|&=|G_{t}^{t_0}|\in \{1,-1\},
\\
 \label{2eq:dual11}
 D^{-1}(G_{t}^{t_0})^{T}DC_{t}^{t_0}&=I,
 \\
\label{2eq:dual21}
C_{t}^{t_0} &= (\tilde{G}_{t_0}^{t})^T,\\
\label{2eq:dual31}
G_{t}^{t_0} &= (\tilde{C}_{t_0}^{t})^T.
\end{align}
\par
$(b)_d$.
For any $t_0, t\in \bbT_n$ such that $d(t_0,t)=d$, the following facts hold:
\begin{itemize}
\item[(i).]
The set
$\sigma^{\circ}(G_t^{t_0})$ is a chamber of $\frakS(B_{t_0})$.
\item[(ii).]
Each $c$-vector $\bfc_{i;t}^{t_0}$
is a normal vector of the support of a wall of $\frakS(B_{t_0})$.
\item[(iii).]
The column sign-coherence holds for $C_{t}^{t_0}$.

\item[(iv).]
The row sign-coherence holds for $G_{t}^{t_0}$.
\end{itemize}
\par
$(c)_d$.
For any $t_0,t_1, t,t'\in \bbT_n$ such that
\begin{itemize}
\item
 $d(t_0,t)=d$,
 \item
$t_0$ and $t_1$  are $k$-adjacent,
\item
$t$ and $t'$  are $k$-adjacent,
 \end{itemize}
 the following facts hold:
\begin{align}
\label{2eq:cmut51}
C_{t'}^{t_0}&=C_t^{t_0}(J_k + [{\varepsilon}_{\bullet k}(C^{t_0}_t )B_t]_+^{k \bullet}),
\\
\label{2eq:gmut51}
G_{t'}^{t_0}&=G_t^{t_0}(J_k + [-{\varepsilon}_{\bullet k}(C^{t_0}_t )B_t]_+^{\bullet k}),
\\
\label{2eq:cmut61}
C^{t_1}_{t}&=(J_k + [-{\varepsilon}_{k\bullet}(G^{t_0}_t )B_{t_0}]_+^{k \bullet})C^{t_0}_t,
\\
\label{2eq:gmut61}
G^{t_1}_{t}&=(J_k + [{\varepsilon}_{k\bullet}(G^{t_0}_t )B_{t_0}]_+^{\bullet k})G^{t_0}_t.
\end{align}

We prove the claims in the following order,
\begin{align}
\label{2eq:ind2}
(a)_0 \Longrightarrow 
(b)_0 \Longrightarrow 
(c)_0 \Longrightarrow 
(a)_1 \Longrightarrow 
(b)_1 \Longrightarrow 
(c)_1 \Longrightarrow 
(a)_2 \Longrightarrow 
\cdots,
\end{align}
assuming all  preceding claims.
Moreover, we run the same induction procedure for 
  $C$- and $G$-matrices
 $\tilde{C}_{t}^{t_0}$, $\tilde{G}_{t}^{t_0}$ of $\bfB^T$ in the background.

We first prove $(a)_0$, $(b)_0$, and $(c)_0$.
The claims $(a)_0$ and $(c)_0$ were proved in Propositions \ref{2prop:dual1} and \ref{2prop:dual2}.
The claim $(b)_0$ holds because $C_{t_0}^{t_0}=G_{t_0}^{t_0}=I$,
 and Proposition \ref{2prop:SD1} (c),
 where we have already received an important input from
  scattering diagrams.

Now, assuming the claims in 
\eqref{2eq:ind2} up to $(c)_{d}$, we prove $(a)_{d+1}$.
This was already proved in Propositions \ref{2prop:dual1} and \ref{2prop:dual2}
by taking account of Remarks \ref{2rem:sign1} and \ref{2rem:sign2}.

Next, assuming the claims in 
\eqref{2eq:ind2} up to $(a)_{d+1}$,
we prove $(b)_{d+1}$.
(This is the highlights of the proof of the sign-coherence conjecture.)
Let $t_0,t_1,t\in  \bbT_n$
be vertices such that $d(t_0,t)=d$, $d(t_1,t)=d+1$, and
$t_0$ and $t_1$ are $k$-adjacent.
By $(b)_{d}$, 
$\sigma^{\circ}(G_t^{t_0})$ is a chamber in $\frakS(B_{t_0})$.
By  $(a)_d$, $(c)_d$, Proposition \ref{2prop:Gfan1},
and Remark \ref{2rem:Gfan1},
we have $\sigma(G_t^{t_1})=\varphi_{t_0}^{t_1}(\sigma(G_t^{t_0}))$.
Meanwhile,
by Proposition
\ref{2prop:SD2},
the piecewise linear isomorphism $\varphi_{t_0}^{t_1}$,
which is homeomorphism,
maps a chamber in  $\frakS(B_{t_0})$
to a chamber in  $\frakS(B_{t_1})$.
Therefore, $\sigma^{\circ}(G_t^{t_1})$ is  a chamber in  $\frakS(B_{t_1})$.
This proves (i).
By (i),
each face
$\sigma_i(G_t^{t_1})$ ($i=1,\dots,n$),
which is a cone of codimension one by $(a)_{d+1}$,
is a subset of the union  of the supports of some walls of $\frakS(B_{t_1})$.
Meanwhile,
by $(a)_{d+1}$,
each $c$-vector $\bfc_{i;t}^{t_1}$ is a normal vector of
the face $\sigma_i(G_t^{t_1})$
with respect to the inner product $(\ ,\ )_D$.
Therefore, 
each $c$-vector $\bfc_{i;t}^{t_1}$
is a normal vector of the support of some wall of $\frakS(B_{t_1})$.
See Proposition \ref{2prop:duality3}.
This proves (ii).
Then, by Proposition \ref{2prop:SD1} (a),
the $c$-vector $\bfc_{i;t}^{t_1}$ is either positive or negative.
This proves (iii), i.e., the column sign-coherence of the $C$-matrix $C_{t}^{t_1}$.
The same result is proved in the background for  the $C$-matrix  $\tilde{C}_{t_1}^{t}$
of the $B$-pattern $\bfB^T$.
Thus, by $(a)_{d+1}$,  the $G$-matrix $G_{t}^{t_1}$ is row sign-coherent.
This proves (iv).

Finally, 
 assuming the claims in 
\eqref{2eq:ind2} up to $(b)_{d+1}$,
we prove $(c)_{d+1}$.
This was already proved in Propositions \ref{2prop:dual1} and \ref{2prop:dual2}
by taking account of Remarks \ref{2rem:sign1} and \ref{2rem:sign2}.
\end{proof}

We may roughly describe the results in  Theorem \ref{2thm:SCprove1} (a)
that \emph{the $G$-fan $\Delta(\bfG^{t_0})$ is embedded in the corresponding cluster scattering diagram $\frakD_{\fraks}$}.

 We immediately extend Theorem \ref{2thm:SCprove1} (c) 
for any (possibly singular) $B$-pattern by
 Proposition \ref{2prop:principal1},

\begin{thm}
Let $\bfB=\{B_t\}_{t\in \bbT_n}$ be any (possibly singular) $B$-pattern.
Then, for any  $t_0,t\in\bbT_n$, 
the column sign-coherence for the $C$-matrix $C_t^{t_0}$ holds.
\end{thm}
\begin{proof}
Let $\overline{C}_t$ and $\overline{G}_t$
be $C$- and $G$-matrices
of  the principal extension $\overline{\bfB}$ of $\bfB$ with the initial vertex $t_0$
in Section \ref{2subsec:principal1}.
The matrix  $\overline{C}_t$ is column sign-coherent by Theorem \ref{2thm:SCprove1} (c).
Then, the claim follows from \eqref{2eq:Cbar1}.
\end{proof}

This completes the proof of Conjecture \ref{2conj:Csign1}.

From now on, we use the column sign-coherence of $C$-matrices
not as an assumption but as a theorem.

\begin{rem}
In Part III we show that Propositions \ref{2prop:SD1} and \ref{2prop:SD2}
actually
holds for any (possibly singular) $B$-pattern.
Thus, Theorem \ref{2thm:SCprove1} also holds for any (possibly singular) $B$-pattern
by the same proof.
This gives a direct proof of the sign-coherence without using 
 Proposition \ref{2prop:principal1}.
\end{rem}

\newpage
\section{More about $F$-polynomials}
\label{2sec:Fpoly1}

Let us turn to study $F$-polynomials,
namely, the nontropical part of a cluster pattern.

\subsection{Fock-Goncharov decomposition}

Let $\bfSigma$ be any cluster pattern.
We come back to the mutations of $x$-
and $y$-variables, especially in the $\varepsilon$-expression
\eqref{2eq:xmut2} and \eqref{2eq:ymut2}.
Let $t, t'\in \bbT_n$ be vertices  that are $k$-adjacent.
One can regard
mutations \eqref{2eq:xmut2} and \eqref{2eq:ymut2}
as  isomorphisms between rational function fields as follows,
where $\bfx_{t}$, $\bfx_{t'}$, $\bfy_{t}$, $\bfy_{t'}$ are $n$-tuple of formal variables:
\begin{align}
&\mu_{k;t}:\bbQ(\bfx_{t'}) \rightarrow \bbQ(\bfx_t),
\
\bbQ(\bfy_{t'}) \rightarrow \bbQ(\bfy_t),
\\
\label{2eq:xmut3}
&\mu_{k;t}(x_{i;t'})=
\begin{cases}
\displaystyle
x_{k;t}^{-1}\Biggl(\, \prod_{j=1}^n x_{j;t}^{[-\varepsilon b_{jk;t}]_+}
\Biggr)
( 1+\hat{y}_{k;t}^{\varepsilon})
& i=k,
\\
x_{i;t}
&i\neq k,
\end{cases}
\\
\label{2eq:ymut3}
&\mu_{k;t}(
y_{i;t'}
)
=
\begin{cases}
\displaystyle
y_{k;t}^{-1}
& i=k,
\\
y_{i;t} y_{k;t}^{[\varepsilon b_{ki;t}]_+} (1+ y_{k;t}^{\varepsilon})^{-b_{ki;t}}
&i\neq k.
\end{cases}
\end{align}
Note that the map is in the opposite direction, namely,  from $t'$ to $t$.
Here and below, we use common symbols for maps for
$x$- and $y$-variables as above in view of the parallelism between them.

To make use of the  flexibility of the choice of 
sign $\varepsilon$, we first choose a given
initial vertex $t_0$.
Since we have already established the sign-coherence
of $C$-matrices $C_t^{t_0}=C_t$,
the tropical sign $\varepsilon_{i;t}^{t_0}=\varepsilon_{i;t}$ in
Definition \ref{2defn:tropicalsign1} is defined for any $t$ and $i$.
Now we set $\varepsilon=\varepsilon_{k;t}$ 
in \eqref{2eq:xmut3} and \eqref{2eq:ymut3}.

Then, modifying the idea of \cite[\S2.1]{Fock03} with tropical signs, 
we consider the following decompositions of the maps $\mu_{k;t}$:
\begin{align}
\label{2eq:decom1}
\mu_{k;t}=\rho_{k;t}\circ \tau_{k;t}.
\end{align}
Here, $\tau_{k;t}^{t_0}=\tau_{k;t}$ are the following isomorphisms,
\begin{align}
&\tau_{k;t}:\bbQ(\bfx_{t'}) \rightarrow \bbQ(\bfx_t),
\
\bbQ(\bfy_{t'}) \rightarrow \bbQ(\bfy_t),
\\
\label{2eq:tau1}
&\tau_{k;t}(x_{i;t'})=
\begin{cases}
\displaystyle
x_{k;t}^{-1}\prod_{j=1}^n x_{j;t}^{[-\varepsilon_{k;t} b_{jk;t}]_+}
& i=k,
\\
x_{i;t}
&i\neq k,
\end{cases}
\\
\label{2eq:tau2}
&\tau_{k;t}(
y_{i;t'}
)
=
\begin{cases}
\displaystyle
y_{k;t}^{-1}
& i=k,
\\
y_{i;t} y_{k;t}^{[\varepsilon_{k;t} b_{ki;t}]_+} 
&i\neq k,
\end{cases}
\end{align}
while $\rho_{k;t}^{t_0}=\rho_{k;t}$ are the following automorphisms,
\begin{align}
&\rho_{k;t}:\bbQ(\bfx_{t}) \rightarrow \bbQ(\bfx_t),
\
\bbQ(\bfy_{t}) \rightarrow \bbQ(\bfy_t),
\\
\label{2eq:rho1}
&\rho_{k;t}(x_{i;t})=
x_{i;t}(1+\hat{y}_{k;t}^{\varepsilon_{k;t}})^{-\delta_{ik}},
\\
\label{2eq:rho2}
&\rho_{k;t}(
y_{i;t}
)
=y_{i;t}
(1+y_{k;t}^{\varepsilon_{k;t}})^{-b_{ki;t}}.
\end{align}
Note that the maps $\tau_{k;t}$ and $\rho_{k;t}$ depend on the choice
of the initial vertex $t_0$
through  tropical signs.
Also, they are defined only after establishing the sign-coherence of $C$-matrices.
We  call the decomposition \eqref{2eq:decom1} the \emph{Fock-Goncharov decomposition\/}
\index{Fock-Goncharov decomposition}
of a mutation ${\mu}_{k;t}$ (\emph{with tropical sign\/}) with respect to the initial vertex $t_0$.
We also call $\tau_{k;t}$ and $\rho_{k;t}$ the \emph{tropical part\/} \index{Fock-Goncharov decomposition!tropical part}
and the \emph{nontropical part\/} of ${\mu}_{k;t}$, respectively. \index{Fock-Goncharov decomposition!nontropical part}
Indeed,
 it is clear that \eqref{2eq:tau1} and \eqref{2eq:tau2} are the exponential form
 of the mutations of $g$-vectors \eqref{2eq:gmut7} and $c$-vectors \eqref{2eq:cmut7}, respectively.

\begin{rem}
The decomposition \eqref{2eq:decom1} was introduced in \cite{Fock03}
for $\varepsilon=1$.
The signed version here  were used, for example, in the application to
the  Stokes automorphisms
  \cite{Iwaki14a} and
dilogarithm identities 
 \cite{Gekhtman16}.
\end{rem}

\begin{prop}
\label{2prop:inv1}
Let $t,t'\in \bbT_n$ be vertices that are $k$-adjacent.
Then, the following relations hold.
\begin{align}
\label{2eq:mm1}
\mu_{k;t'}\circ \mu_{k;t}&=\mathrm{id},
\\
\label{2eq:tt1}
\tau_{k;t'}\circ \tau_{k;t}&=\mathrm{id}.
\end{align}
\end{prop}
\begin{proof}
The equality \eqref{2eq:mm1} is nothing but the involution of the mutations of $x$- and $y$-variables.
The equality \eqref{2eq:tt1} for both $x$- and $y$-variables can be verified by the fact,
\begin{align}
\varepsilon_{k;t'}=-\varepsilon_{k;t},
\quad
b_{ki;t'}=-b_{ki;t},
\quad
b_{jk;t'}=-b_{jk;t}.
\end{align}
\end{proof}
For any $t\in \bbT_n$,
consider a sequence of vertices $t_0$, $t_1$, \dots, $t_{r+1}=t\in \bbT_n$
such that
they  are sequentially adjacent with edges labeled by $k_0$, \dots, $k_{r}$.
Then, we define isomorphisms
\begin{align}
\label{2eq:mucom1}
\mu_{t}^{t_0}&:=\mu_{k_0;t_0}\circ\mu_{k_1;t_1}\circ\dots \circ \mu_{k_{r};t_{r}}
:\bbQ(\bfx_{t}) \rightarrow \bbQ(\bfx),
\
\bbQ(\bfy_{t}) \rightarrow \bbQ(\bfy),
\\
\label{2eq:mucom2}
\tau_{t}^{t_0}&:=\tau_{k_0;t_0}\circ\tau_{k_1;t_1}\circ\dots \circ \tau_{k_{r};t_{r}}
:\bbQ(\bfx_{t}) \rightarrow \bbQ(\bfx),
\
\bbQ(\bfy_{t}) \rightarrow \bbQ(\bfy),
\end{align}
where we set $\bfx_{t_0}=\bfx$, $\bfy_{t_0}=\bfy$.
Thanks to Proposition \ref{2prop:inv1}, $\mu_{t}^{t_0}$ and
$\tau_{t}^{t_0}$ depend only on $t_0$ and $t$.
Namely, we do not have to care about the redundancy $k_{s+1}=k_s$ 
in the sequence $k_0,\dots,k_r$.

Recall the notions of tropical $x$- and $y$-variables  in 
\eqref{2eq:xtrop1} and \eqref{2eq:ytrop1}.
The following proposition  tells
that the tropical parts $\tau_{t}^{t_0}$ are nothing but the
mutations of  tropical $x$- and $y$-variables.

\begin{prop}
\label{2prop:trop1}
The following formulas hold:
\begin{align}
\label{2eq:mux1}
\mu_{t}^{t_0}(x_{i;t})&=x^{\bfg_{i;t}}F_{i;t}(\hat\bfy),\\
\label{2eq:muy1}
\mu_{t}^{t_0}(y_{i;t})&=y^{\bfc_{i;t}}\prod_{j=1}^n F_{j;t}(\bfy)^{b_{ji;t}},\\
\label{2eq:tg1}
\tau_{t}^{t_0}(x_{i;t})&=x^{\bfg_{i;t}},\\
\label{2eq:tc1}
\tau_{t}^{t_0}(y_{i;t})&=y^{\bfc_{i;t}}.
\end{align}
\end{prop}
\begin{proof}
The first two equalities are  the separation formulas
in Theorem \ref{2thm:sep1}.
Let us prove the equality \eqref{2eq:tg1} for $x$-variables by the induction 
on $t$
along $\bbT_n$ starting from $t_0$.
For $t=t_0$, $x^{\bfg_{i;t_0}}=x^{\bfe_i}=x_i$.
Therefore,  \eqref{2eq:tg1}  holds.
Suppose that it is true for  $t$ with $d(t_0,t)=d$.
Let $t'$ be the vertex that is
$k$-adjacent to $t$
 such that $d(t_0,t')=d+1$.
Then, we have
\begin{align}
\begin{split}
&\quad\ \tau_{t'}^{t_0}(x_{i;t'})
=(\tau_{t}^{t_0}\circ \tau_{k;t})(x_{i;t'})
\\
&=
\begin{cases}
\displaystyle
\tau_{t}^{t_0}\Biggl(
x_{k;t}^{-1} \prod_{j=1}^n x_{j;t}^{[-\varepsilon_{k;t} b_{jk;t}]_+}
\Biggr)
=x^{ -\bfg_{k;t}+\sum_{j=1}^n
[-\varepsilon_{k;t} b_{jk;t}]_+
\bfg_{j;t}}
&
i=k,
\\
\tau_{t}^{t_0}(
x_{i;t})
=x^{\bfg_{i;t}}
& i\neq k
\end{cases}
\\
&
=x^{\bfg_{i;t'}},
\end{split}
\end{align}
where we used \eqref{2eq:gmut7} in the last equality.
Therefore,  \eqref{2eq:tg1}  holds for $t'$.
Similarly,
the equality \eqref{2eq:tc1} for $y$-variables follows from \eqref{2eq:cmut7}.
\end{proof}

Under these maps, $\hat{y}$-variables transform in the same way as $y$-variables as expected.
\begin{prop}
\label{2prop:tauyhat1}
The following formulas hold:
\begin{align}
\label{2eq:tau3}
\tau_{k;t}(
\hat{y}_{i;t'}
)
&=
\begin{cases}
\displaystyle
\hat{y}_{k;t}^{-1}
& i=k,
\\
\hat{y}_{i;t} \hat{y}_{k;t}^{[\varepsilon_{k;t} b_{ki;t}]_+} 
&i\neq k,
\end{cases}
\\
\label{2eq:rho3}
\rho_{k;t}(
\hat{y}_{i;t}
)
&=\hat{y}_{i;t}
(1+\hat{y}_{k;t}^{\varepsilon_{k;t}})^{-b_{ki;t}},
\\
\label{2eq:thc1}
\tau_{t}^{t_0}(\hat{y}_{i;t})&=\hat{y}^{\bfc_{i;t}}
=x^{\hat\bfc_{i;t}},
\end{align}
where $\hat\bfc_{i;t}=B_{t_0}\bfc_{i;t}$ is a $\hat{c}$-vector in
\eqref{2eq:chat1}.
\end{prop}
\begin{proof}
The equality \eqref{2eq:tau3}
follows from
 \eqref{2eq:tau1} and \eqref{2eq:bmut2}.
 The equality \eqref{2eq:rho3} 
 follows from \eqref{2eq:rho1}.
 The equality  \eqref{2eq:thc1}
 follows from  \eqref{2eq:tg1} and \eqref{2eq:yhattrop2}, 

\end{proof}

\subsection{Nontropical parts and $F$-polynomials}
Observing Proposition \ref{2prop:trop1},
it is clear that the nontropical parts $\rho_{k;t}$ are responsible to
generate and
mutate $F$-polynomials.
Let us make this statement more manifest.

In addition to the inner product $(\ , \ )_{D}$
in \eqref{2eq:Dinner1},
we introduce a  skew-symmetric form
\begin{align}
\label{2eq:Dinner12}
\{\bfu, \bfv\}_{DB}:=\bfu^T DB \bfv=(\bfu, B\bfv)_D,
\end{align}
where $B=B_{t_0}$.
Then, we introduce the following automorphisms  $\frakq_{k;t}$ for the initial $x$- and $y$-variables,
\begin{align}
\frakq_{k;t}:\bbQ(\bfx) &\rightarrow \bbQ(\bfx),
\
\bbQ(\bfy) \rightarrow \bbQ(\bfy),
\\
\label{2eq:fp1}
\frakq_{k;t}(x^{\bfm})&=
x^{\bfm}(1+
\hat{y}^{\bfc_{k;t}^+}
)^{-(\d_k \bfc_{k;t},\bfm)_D}\\
&=
x^{\bfm}(1+
x^{\hat\bfc_{k;t}^+}
)^{-(\d_k \bfc_{k;t},\bfm)_D}
,
\\
\label{2eq:fp2}
\frakq_{k;t}(
y^{\bfn}
)
&=y^{\bfn}
(1+y^{{\bfc}_{k;t}^+}
)^{-\{\d_k \bfc_{k;t},\bfn\}_{DB}},
\end{align}
where $\bfm,\bfn\in \bbZ^n$, and
$\bfc_{k;t}^+=\varepsilon_{k;t}\bfc_{k;t}$ and
$\hat{\bfc}_{k;t}^+=\varepsilon_{k;t}\hat{\bfc}_{k;t}$
are a  $c^+$-vector
 and 
a  $\hat{c}^+$-vector,
 respectively,
 in Definition \ref{2defn:normalc1}.

The  seemingly asymmetric definitions in
 \eqref{2eq:fp1} and \eqref{2eq:fp2}
are justified by the following properties.

\begin{prop}
\label{2prop:taurho1}
The following facts hold:
\par
(a). We have the formulas
\begin{align}
\label{2eq:fp3}
&\frakq_{k;t}(x^{\bfg_{i;t}})=
x^{\bfg_{i;t}}(1+
\hat{y}^{\bfc_{k;t}^+}
)^{-\delta_{ik}},
\\
\label{2eq:fp4}
&\frakq_{k;t}(
y^{\bfc_{i;t}}
)
=y^{\bfc_{i;t}}
(1+y^{{\bfc}_{k;t}^+}
)^{-b_{ki;t}}.
\end{align}
\par
(b). The following relation holds:
\begin{align}
\label{2eq:taurho1}
\tau^{t_0}_t \circ \rho_{k;t}
= \frakq_{k;t}\circ \tau^{t_0}_t.
\end{align}
\par
(c). If $t'$ and $t$ are $k$-adjacent, we have
\begin{align}
\frakq_{k;t'}
\circ
\frakq_{k;t}
=\mathrm{id}.
\end{align}
\end{prop}
\begin{proof}
(a). 
The claim follows from the following equalities:
\begin{align}
\label{2eq:gc1}
(\d_k \bfc_{k;t},\bfg_{i;t})_D&=\delta_{ki},\\
\label{2eq:gc2}
\{\d_k \bfc_{k;t}, \bfc_{i;t}\}_{DB}&=b_{ki;t}.
\end{align}
The  equality \eqref{2eq:gc1} is
the duality in \eqref{2eq:dcg1}.
The equality \eqref{2eq:gc2} follows from
\eqref{2eq:cbc1} as follows:
\begin{align}
D^{-1}(C_t^TDB_{t_0}C_t)
=
D^{-1} (DB_t)=B_t.
\end{align}
\par
(b).
Thanks to \eqref{2eq:fp3} and Propositions \ref{2prop:trop1} and \ref{2prop:tauyhat1},
we have
\begin{align}
&x_{i;t} 
\ {\buildrel \rho_{k;t} \over \mapsto}\ 
x_{i;t}(1+\hat{y}_{k;t}^{\varepsilon_{k;t}})^{-\delta_{ik}}
\ {\buildrel \tau_{t}^{t_0} \over \mapsto}\ 
x^{\bfg_{i;t}}(
1+
\hat{y}^{\bfc_{k;t}^+}
)^{-\delta_{ik}},
\\
&x_{i;t} 
\ {\buildrel \tau_{t}^{t_0} \over \mapsto}\ 
x^{\bfg_{i;t}}
\ {\buildrel  \frakq_{k;t} \over \mapsto}\ 
x^{\bfg_{i;t}}(1+
\hat{y}^{\bfc_{k;t}^+}
)^{-\delta_{ik}}.
\end{align}
Therefore, the equality \eqref{2eq:taurho1} holds for $x$-variables. Similarly, we have
\begin{align}
&y_{i;t} 
\ {\buildrel \rho_{k;t} \over \mapsto}\ 
y_{i;t}(1+{y}_{k;t}^{\varepsilon_{k;t}})^{-b_{ki;t}}
\ {\buildrel \tau_{t}^{t_0} \over \mapsto}\ 
y^{\bfc_{i;t}}(
1+
{y}^{\bfc_{k;t}^+}
)^{-b_{ki;t}},
\\
&y_{i;t} 
\ {\buildrel \tau_{t}^{t_0} \over \mapsto}\ 
y^{\bfc_{i;t}}
\ {\buildrel  \frakq_{k;t}\over \mapsto}\ 
y^{\bfc_{i;t}}(1+
{y}^{\bfc_{k;t}^+}
)^{-b_{ki;t}}.
\end{align}
Therefore, the equality \eqref{2eq:taurho1} holds also for $y$-variables.
\par
(c).
This follows from $\bfc_{k;t'}=-\bfc_{k;t}$.
\end{proof}

For the same
$t_0$, $t_1$, \dots, $t_{r+1}=t\in \bbT_n$
for $\mu_{t}^{t_0}$ and $\tau_{t}^{t_0}$ in 
\eqref{2eq:mucom1} and \eqref{2eq:mucom2},
we define the automorphism
\begin{align}
\label{2eq:fq1}
\frakq_{t}^{t_0}&:=
\frakq_{k_{0};t_{0}}
\circ\frakq_{k_1;t_1}\circ
\dots
\circ
\frakq_{k_{r};t_{r}}
:
\bbQ(\bfx) \rightarrow \bbQ(\bfx),
\
\bbQ(\bfy) \rightarrow \bbQ(\bfy).
\end{align}
Again, by Proposition \ref{2prop:taurho1} (c),
it depends only on $t_0$ and $t$.

One can  separate the tropical and the nontropical
parts of  $\mu^{t_0}_t $
as follows:
\begin{prop}
\label{2prop:sep3}
The following decomposition holds:
\begin{align}
\label{2eq:taurho2}
\mu^{t_0}_t 
= \frakq_t^{t_0}\circ \tau^{t_0}_t
:
\bbQ(\bfx_t) \rightarrow \bbQ(\bfx),
\
\bbQ(\bfy_t) \rightarrow \bbQ(\bfy).
\end{align}
\end{prop}
\begin{proof}
First, we note that
\begin{align}
&\rho_{k_0;t_0}(x_{i})=
x_{i}(1+\hat{y}_{k_0})^{-\delta_{ik_0}}
=
\frakq_{k_0;t_0}(x_{i})
,
\\
&\rho_{k_0;t_0}(
y_{i}
)
=y_{i}
(1+y_{k_0})^{-b_{k_0i;t_0}}
=
\frakq_{k_0;t_0}(y_{i}).
\end{align}
For $x$-variables, for example,
by Proposition \ref{2prop:taurho1} (b),
we have the following commutative diagram:
\begin{align}
\raisebox{70pt}
{
\begin{xy}
(0,0)*+{\bbQ(\bfx_t)}="aa";
(20,0)*+{\bbQ(\bfx_{t_{r}})}="ba";
(20,-12)*+{\bbQ(\bfx_{t_{r}})}="bb";
(40,-12)*+{\bbQ(\bfx_{t_{r-1}})}="cb";
(40,-24)*+{\bbQ(\bfx_{t_{2}})}="cc";
(60,-24)*+{\bbQ(\bfx_{t_{1}})}="dc";
(60,-36)*+{\bbQ(\bfx_{t_{1}})}="dd";
(80,0)*+{\bbQ(\bfx_{t_{0}})}="ea";
(80,-12)*+{\bbQ(\bfx_{t_{0}})}="eb";
(80,-24)*+{\bbQ(\bfx_{t_{0}})}="ec";
(80,-36)*+{\bbQ(\bfx_{t_{0}})}="ed";
(80,-48)*+{\bbQ(\bfx_{t_{0}})}="ee";
(10,3)*+{\text{\small $\tau_{k_r;t_r}$}};
(50,3)*+{\text{\small $\tau^{t_0}_{t_r}$}};
(60,-9)*+{\text{\small $\tau^{t_0}_{t_{r-1}}$}};
(30,-9)*+{\text{\small $\tau_{k_{r-1};t_{r-1}}$}};
(50,-21)*+{\text{\small $\tau_{k_{1};t_{1}}$}};
(70,-21)*+{\text{\small $\tau_{k_{0};t_{0}}$}};
(70,-33)*+{\text{\small $\tau_{k_{0};t_{0}}$}};
(4,-6)*+{\text{\small $\mu_{k_r;t_r}$}};
(44,-30)*+{\text{\small $\mu_{k_{1};t_{1}}$}};
(64,-42)*+{\text{\small $\mu_{k_0;t_0}$}};
(26,-6)*+{\text{\small $\rho_{k_r;t_r}$}};
(65,-30)*+{\text{\small $\rho_{k_{1};t_{1}}$}};
(86,-6)*+{\text{\small $\frakq_{k_r;t_r}$}};
(86,-30)*+{\text{\small $\frakq_{k_{1};t_{1}}$}};
(92,-42)*+{\text{\small $\rho_{k_0;t_0}=\frakq_{k_0;t_0}$}};
(80,-17)*+{\vdots};
(40,-17)*+{\vdots};
(30,-17)*+{\ddots};
{\ar "aa";"ba"};
{\ar "ba";"bb"};
{\ar "aa";"bb"};
{\ar "ba";"ea"};
{\ar "bb";"cb"};
{\ar "cb";"eb"};
{\ar "cc";"dc"};
{\ar "cc";"dd"};
{\ar "ea";"eb"};
{\ar "dc";"dd"};
{\ar "dc";"ec"};
{\ar "ec";"ed"};
{\ar "dd";"ed"};
{\ar "ed";"ee"};
{\ar "dd";"ee"};
\end{xy}
}
\end{align}
Therefore, the claim holds.
\end{proof}

We conclude
that the automorphisms $ \frakq_t^{t_0}$ generate
the nontropical parts of
$x$- and $y$-variables 
in the following manner.
 
 \begin{thm}
 The following formulas hold:
  \begin{align}
 \label{2eq:fp5}
 \frakq_t^{t_0}(x^{\bfg_{i;t}})
 &=
 x^{\bfg_{i;t}}
 F_{i;t}(\hat\bfy)
 ,\\
  \label{2eq:fp51}
 \frakq_t^{t_0}(y^{\bfc_{i;t}})
 &=
 y^{\bfc_{i;t}}
 \prod_{j=1}^n
 F_{j;t}(\bfy)^{b_{ji;t}}
,\\
 \label{2eq:fp9}
 \frakq_t^{t_0}(\hat{y}^{\bfc_{i;t}})
 &=
 \hat{y}^{\bfc_{i;t}}
 \prod_{j=1}^n
 F_{j;t}(\hat\bfy)^{b_{ji;t}}.
 \end{align}
 \end{thm}
 \begin{proof}
 The  equalities  \eqref{2eq:fp5} and  \eqref{2eq:fp51} 
 follow from Propositions
\ref{2prop:trop1} and \ref{2prop:sep3}.
For example,
for $x$-variables,
\begin{align}
 \frakq_t^{t_0}(x^{\bfg_{i;t}})
=
\frakq_t^{t_0}(\tau_{t}^{t_{0}}(x_{i;t}))
=
\mu_t^{t_0}(x_{i;t})
=
 x^{\bfg_{i;t}}
 F_{i;t}(\hat\bfy).
\end{align}
Similarly,
 the  equality  \eqref{2eq:fp9} 
 follows from Propositions
\ref{2prop:tauyhat1} and \ref{2prop:sep3}.
 \end{proof}
 
  \subsection{Detropicalization}
 \label{2subsec:synchro1}
 
 Let $\bfSigma$ be any cluster pattern of rank $n$,
 and let $t_0$ be a given initial vertex.
 Let $S_n$ be the symmetric group of degree $n$.
 We define the (left) action of a permutation $\nu\in S_n$
 on $x$- and $y$-variables, etc., as follows.
 \begin{alignat}{3}
 \label{2eq:sigmax1}
\bfx'&=\nu \bfx_t, &\quad
x'_i&=x_{\nu^{-1}(i);t},\\
\bfy'&=\nu \bfx_t, &\quad
y'_i&=y_{\nu^{-1}(i);t},\\
B'&=\nu B_t, &\quad
b'_{ij}&=b_{\nu^{-1}(i)\nu^{-1}(j);t},\\
C'&=\nu C_t, &\quad
c'_{ij}&=c_{i\nu^{-1}(j);t},\\
 \label{2eq:sigmag1}
G'&=\nu G_t, &\quad
g'_{ij}&=g_{i\nu^{-1}(j);t}.
\end{alignat}
Note that these actions are compatible with the separation formulas
in Theorem \ref{2thm:sep1}.

For any $t,t'\in \bbT_n$ and any permutation $\nu\in S_n$,
we define the following isomorphisms.
\begin{align}
&\nu_{t}^{t'}:\bbQ(\bfx_{t}) \rightarrow \bbQ(\bfx_{t'}),
\
\bbQ(\bfy_{t}) \rightarrow \bbQ(\bfy_{t'}),
\\
\label{2eq:sigma3}
&\nu_{t}^{t'}(x_{i;t})=
x_{\nu^{-1}(i);t'},
\\
\label{2eq:sigma4}
&\nu_{t}^{t'}(
y_{i;t}
)
=y_{\nu^{-1}(i);t'}.
\end{align}
Then, 
by
Proposition \ref{2prop:trop1},
one can rephrase the \emph{periodicity condition\/}
for $G$-matrices and $x$-variables as follows:
\begin{align}
\label{2eq:period1}
G_{t}^{t_0}=\nu G_{t'}^{t_0}
&\quad\Longleftrightarrow\quad \tau_t^{t_0}=\tau_{t'}^{t_0}\circ\nu_{t}^{t'} \ \mbox{(for $x$-variables)},
\\
\label{2eq:period2}
\bfx_{t}=\nu \bfx_{t'}
&\quad\Longleftrightarrow\quad  \mu_t^{t_0}=\mu_{t'}^{t_0}\circ\nu_{t}^{t'} \ \mbox{(for $x$-variables)}.
\end{align}
Also,  the parallel statement holds for $C$-matrices and $y$-variables.
\begin{prop}
\label{2prop:detrop1}
Let $t_0, t, t'$ be any vertices in $\bbT_n$.
Suppose that 
\begin{align}
G_{t}^{t_0}=\nu G_{t'}^{t_0}
\end{align}
holds.
Then, we have
\begin{align}
\label{2eq:period3}
\bfx_{t}=\nu \bfx_{t'} 
\quad
&\Longleftrightarrow \quad \frakq_t^{t_0}=\frakq_{t'}^{t_0} \ \mbox{\rm(for $x$-variables)}.
\end{align}
Also,  the parallel statement holds for $C$-matrices and $y$-variables.
\end{prop}
\begin{proof}
The claim  follows from  \eqref{2eq:period1}, \eqref{2eq:period2},
and
Proposition \ref{2prop:sep3}.
\end{proof}

Note that the permutation $\nu_{t}^{t'}$  drops off in  the condition for $ \frakq_t^{t_0}$
in   \eqref{2eq:period3}.

The properties
\begin{align}
\label{2eq:detro1}
G_{t}^{t_0}=\nu G_{t'}^{t_0}
&\quad\Longrightarrow\quad \bfx_{t}=\nu \bfx_{t'},
\\
\label{2eq:detro2}
C_{t}^{t_0}=\nu C_{t'}^{t_0}
&\quad\Longrightarrow\quad \bfy_{t}=\nu \bfy_{t'}
\end{align}
are called  the \emph{detropicalization\/} \index{detropicalization}
of $x$- and $y$-variables, respectively.
They are  the core of the \emph{synchronicity phenomenon\/} occurring
in cluster patterns
systematically studied in \cite{Nakanishi19}.

\begin{rem}
\label{rem:detrop1}
The implications \eqref{2eq:detro1} and \eqref{2eq:detro2}
were proved for $\nu=\mathrm{id}$
by \cite[Lemma~2.4 \& Theorem~2.5]{Cao17} 
and  for general $\nu$ by \cite[Theorem~5.2]{Nakanishi19}.
However, the proofs therein depend on the Laurent positivity  in
Theorem \ref{2thm:Fpositive1}, which we temporarily dismiss.
We come back to the point later
in Section \ref{2subsec:synchro2}.
\end{rem}

 \newpage
\section{Scattering diagrams}
\label{2sec:scattering1}

In this section we present a  precise formulation of scattering diagrams
following \cite{Gross13,Gross14}.
To do this, we set the following guideline.
\begin{itemize}
\item[(a).] Since in our approach we do not need
scattering diagrams with \emph{frozen variables\/}
(corresponding to \emph{coefficients of geometrical type\/} in \cite{Fomin07}),
we skip the relevant formulation.
This makes a considerable simplification
of the formulation.
\item[(b).] We omit the description of the underlying Lie algebra $\frak{g}$ and the group
$G$ in \cite{Gross14} for scattering diagrams.
They are crucial for the construction of
scattering diagrams, but we can skip them for our purpose.
\item[(c).] We clarify and establish the connection between  cluster patterns
and scattering diagrams in detail by step-by-step examples.
\item[(d).] We mostly follow the notations and the conventions in \cite{Gross14}.
However, we take the \emph{transpose\/} of  the skew-symmetric forms therein
to match the convention of the exchange matrices in \cite{Fomin07}.
\end{itemize}

\subsection{Fixed data and seed}

\begin{defn}[Fixed data]
A \emph{fixed data\/} \index{fixed data}
 $\Gamma$ consists of the following:
\begin{itemize}
\item
A lattice $N\simeq \bbZ^n$ of rank $n$ with
a skew-symmetric bilinear form
\begin{align}
\label{2eq:bili1}
\{\cdot, \cdot \}: N \times N \rightarrow \bbQ.
\end{align}
\item
A sublattice $N^{\circ}\subset N$ of finite index (equivalently, of rank $n$) such that
\begin{align}
\label{2eq:bili2}
\{N^{\circ}, N\}\subset \bbZ.
\end{align}
\item
Positive integers $\d_1$, \dots, $\d_n$ 
such that
there is a basis $(e_1, \dots, e_n)$ of $N$,
where
$(\d_1e_1, \dots, \d_n e_n)$ is  a basis of $N^{\circ}$.
\item
$M=\mathrm{Hom}(N,\bbZ)$, $M^{\circ}=\mathrm{Hom}(N^{\circ},\bbZ)$,
 $M_{\bbR}=M\otimes_{\bbZ} \bbR$,
and we regard
\begin{align}
M\subset M^{\circ} \subset M_{\bbR}.
\end{align}
\end{itemize}
\end{defn}

For a given $N^{\circ}\subset N$ of finite index,
the above $\d_1,\dots, \d_n$ always exist; for example, take the
 elementary divisors of the embedding $N^{\circ}\hookrightarrow N$.
 However, other choices are equally good.

Let $\langle n, m \rangle: N^{\circ}\times M^{\circ} \rightarrow \bbZ$ be
the canonical paring.
We also write its linear extension 
$N\times M_{\bbR} \rightarrow \bbR$
by the same symbol $\langle n, m \rangle$.
(Though it is a little confusing, the symbol $n$ is used  to denote  both  the rank and an element in $N$.)

\begin{rem}
(a).
In \eqref{2eq:bili2} we already took the transpose of the form in \cite{Gross14}
as mentioned in the above guideline (d).
\par
(b).
In \cite{Gross14}
it is assumed that
the numbers $\d_1$, \dots, $\d_n$ are  coprime.
Here we do not require it.
See also Remark \ref{22rem:equivalence1}.
\end{rem}

\begin{defn}[Seed]
\label{2defn:seed2} \index{seed!for a fixed data}
A \emph{seed\/} $\fraks=(e_1,\dots,e_n)$ for a fixed data $\Gamma$
is a basis of $N$ such that
$(\d_1e_1, \dots, \d_n e_n)$ is  a basis of $N^{\circ}$.
\end{defn}

Let  $\fraks=(e_1,\dots,e_n)$ be a seed for a fixed data $\Gamma$,
and 
let $(e^*_1, \dots, e^*_n)$ be the dual basis of $M$.
Let $f_i=\d_i^{-1}e_i^*$ ($i=1,\dots,n$).
Then,
 $(f_1,\dots,f_n)$ is a basis of $M^{\circ}$.
 
 Accordingly,
 we define isomorphisms, all of which are denoted by the same symbol for simplicity,
 \begin{align}
 \label{2eq:phi10}
& \begin{matrix}
 \phi_{\fraks}:&N& \rightarrow &\bbZ^n,
 \\
  & e_i &\mapsto &\bfe_i,
  \end{matrix}
  \\
 \label{2eq:phi11}
&  \begin{matrix}
   \phi_{\fraks} :&  M^{\circ} & \rightarrow &\bbZ^n,
   & 
 M_{\bbR} & \rightarrow& \bbR^n,\\
  & f_i &\mapsto &\bfe_i.
  \end{matrix}
 \end{align}
 The identifications under the above isomorphisms are written as,
 $N\simeq_{\fraks} \bbZ^n$,  
 $M^{\circ}\simeq_{\fraks} \bbZ^n$,   $M_{\bbR} \simeq_{\fraks} \bbR^n$,
 respectively.

 Let us identify
 $N\simeq_{\fraks} \bbZ^n$ and $M^{\circ}\simeq_{\fraks} \bbZ^n$.
For $n\in N^{\circ}\subset N$ and $m\in M^{\circ}$,
let $\bfn\in \bbZ^n$ and $\bfm\in \bbZ^n$ be the corresponding vectors.
Then,
the canonical paring $\langle n, m \rangle$
for $N^{\circ}\times M^{\circ}$ is given by
\begin{align}
\label{2eq:can1}
\langle n, m \rangle
=(\bfn, \bfm)_D:=\bfn^T D\bfm,
\quad
D=\mathrm{diag}(\d_1^{-1}, \dots, \d_n^{-1}).
\end{align}
This symmetric bilinear form $(\cdot, \cdot)_D$ agrees with the one
in \eqref{2eq:Dinner1}.
We also regard it
 as a  pairing 
$N \times M_{\bbR}\rightarrow \bbR$, depending on the context.

\begin{ex}
\label{2ex:seed1}
A seed $\fraks$ for a fixed data $\Gamma$ is identified
with a seed $(\bfx, \bfy, B)$ (in the sense of Definition \ref{2defn:seed1}) in the following way:
Let $x^m$ ($m\in M^{\circ})$ and $y^n$ ($n \in N)$ be the monomial expressions
with formal symbols $x$ and $y$
such that $x^{m+m'}=x^m x^{m'}$, $y^{n+n'}=y^n y^{n'}$.
Let $\calF_X$ and $\calF_Y$ be the rational function fields generated by
all formal exponentials
$x^m$ and $y^n$, respectively.
Then, a seed $(\bfx_{\fraks}, \bfy_{\fraks}, B_{\fraks})=(\bfx, \bfy, B)$ in $(\calF_X, \calF_Y)$ is defined by
\begin{align}
\label{2eq:initials1}
x_i= x^{f_i},
\quad
y_i= y^{e_i},
\quad
b_{ij}=\{ \d_i e_i, e_j\}.
\end{align}
In particular,
the matrix $D$ in \eqref{2eq:can1}
is a (left) skew-symmetrizer of $B=(b_{ij})$.
Moreover, 
under the identification $N\simeq_{\fraks} \bbZ^n$, the skew-symmetric bilinear form
in \eqref{2eq:bili1} is given by
\begin{align}
\label{2eq:DB2}
\{n, n'\}=\{\bfn, \bfn'\}_{DB}:=\bfn^T DB \bfn'.
\end{align}
This skew-symmetric bilinear form agrees with the one
in \eqref{2eq:Dinner12}.
This is the reason why we take $D$ as
\eqref{2eq:Dmat1} throughout.
For a given seed $\fraks$ (the \emph{initial seed\/}) for a fixed data $\Gamma$ \index{seed!initial}
and a given initial vertex $t_0\in \bbT_n$,
we associate a cluster pattern
$\bfSigma_{\fraks,t_0}$ 
whose initial  seed $\Sigma_{t_0}$ is given
by the  seed $(\bfx, \bfy, B)$ in \eqref{2eq:initials1}.
Below we consider the
$C$-, $G$-, and $F$-patterns
of $\bfSigma_{\fraks,t_0}$ with the initial vertex $t_0$
unless otherwise mentioned.
\end{ex}

\begin{rem}
In the scattering diagram formalism, 
it is natural to consider the diagonal matrix 
with positive integer diagonals
$\Delta=\mathrm{diag}(\delta_1,\dots,\delta_n)$
and the  skew-symmetric rational matrix
$\Omega = (\omega_{ij})_{i,j=1}^n$,
$\omega_{ij}=\{e_i,e_j\}$.
Then, 
the matrices in  \eqref{2eq:can1}--\eqref{2eq:DB2}
are given by
\begin{align}
D=\Delta^{-1},
\quad
B=\Delta \Omega,
\quad
DB=\Omega.
\end{align}
\end{rem}
\subsection{Walls}
Let us introduce the homomorphism
\begin{align}
\label{2eq:p*1}
\begin{matrix}
p^* : &N &\rightarrow &M^{\circ}\\
& n & \mapsto & \{ \cdot, n\}
\end{matrix}.
\end{align}

\begin{ex}[continued]
Using the matrix $B$ in
 \eqref{2eq:initials1}, 
we have
\begin{align}
\label{2eq:pe1}
 p^*(e_j)=\sum_{i=1}^n b_{ij}f_i.
 \end{align}
 Therefore, 
 under our identification $N\simeq_{\fraks} \bbZ^n$ and $M^{\circ}\simeq_{\fraks} \bbZ^n$,
 the matrix representation of $p^*$ is just given by $B$.
Also, again by \eqref{2eq:initials1} and \eqref{2eq:pe1},
\begin{align}
\label{2eq:pe2}
x^{p^*(e_i)}=\prod_{j=1}^n x_j^{b_{ji}}=\hat{y}_i.
\end{align}
\end{ex}

From now on we assume the following condition for a fixed data $\Gamma$:

\begin{assum}[Injectivity Assumption] \index{Injectivity Assumption}
\label{2assum:inj1}
The map $p^*$ is injective.
\end{assum}

\begin{ex}[continued.]
Injectivity Assumption
 is equivalent to assume that
 the skew-symmetric form $\{\cdot,\cdot\}$
 in \eqref{2eq:bili1} is nondegenerate.
 Thus,
  $B$ in  \eqref{2eq:initials1} is nonsingular,
  and   $\hat{y}_1$, \dots, $\hat{y}_n$ 
in \eqref{2eq:pe2}
are algebraically independent.
\end{ex}

Define
\begin{gather}
\label{2eq:N+1}
N^+ = N_{\fraks}^+:=
\biggl\{ \sum_{i=1}^n a_i e_i \biggm | a_i \in \bbZ_{\geq 0},\ \sum_{i=1}^n a_i \neq 0\biggr\},
\\
\label{2eq:P1}
 P=P_{\fraks}:=
 \biggl\{ \sum_{i=1}^n a_i p^*(e_i) \biggm |  a_i \in \bbZ_{\geq 0}\biggr\}
 =p^*(N^+)\sqcup \{0\}
 \subset M^{\circ}.
 \end{gather}
 We view $P$ as a monoid.

\begin{ex}[continued.]
Under the identification $M^{\circ}\simeq_{\fraks} \bbZ^n$,
we have
\begin{align}
 P=
 \biggl\{ \sum_{i=1}^n a_i \bfb_i \biggm | a_i \in \bbZ_{\geq 0}\biggr\},
\end{align}
where $\bfb_i$ is the $i$th column of $B$.
\end{ex}

Let $\bbk$ be any field of characteristic zero, 
e.g., $\bbk=\bbQ$.
Let $\bbk[P]$ be the monoid algebra of $P$ over $\bbk$.
Let $J$  be the maximal ideal of $\bbk[P]$
generated by $P\setminus \{0\}$,
and let 
\begin{align}
\bbk[[P]]=\lim_{\longleftarrow} \bbk[P]/J^{\ell}
\end{align}
be the completion with respect to $J$.
Using the formal exponentials $x^m$ ($m\in P$),
an element of $\bbk[P]$ (resp. $\bbk[[P]]$)  is written as a  polynomial 
(resp. a formal power series) in $x$ as 
\begin{align}
f=
\sum_{m\in P} c_m x^m,
\quad
(c_m \in \bbk).
\end{align}

\begin{rem}
In \cite{Gross14} the notation $z^m$ is used.
We use  $x^m$ because of the forthcoming identification with $x$-variables.
This notation is also compatible with the one in Example \ref{2ex:seed1}.
\end{rem}

\begin{ex}[continued.]
Under the identification \eqref{2eq:pe2}, 
$\bbk[P]$ is  the polynomial algebra $\bbk[\hat{y}_1, \dots, \hat{y}_n]$.
Also, 
$\bbk[[P]]$ is   the algebra of  formal power series   $\bbk[[\hat{y}_1, \dots, \hat{y}_n]]$.
\end{ex}

\begin{defn}[Normalized automorphism]
For $n\in N^+$ and $f=1+\sum_{k=1}^{\infty} c_k x^{kp^*(n)}
\in \bbk[[P]]$,
we define an automorphism $\frakp_{n,f}$ of $\bbk[[P]]$ by
\begin{align}
\label{2eq:fpm1}
\frakp_{n,f}(x^m)=x^m f^{\langle \d(n)n, m\rangle},
\quad
(m\in P),
\end{align}
where $\d(n)$ is the smallest positive rational number 
such that $\d(n)n\in N^{\circ}$.
We call $\d(n)$ the \emph{normalization factor\/} of $n$ \index{normalization factor}
and $\frakp_{n,f}$ the \emph{normalized automorphism\/} by $n$ and $f$. \index{normalized automorphism}
\end{defn}

\begin{ex}[continued.]
\label{2ex:wfunct1}
Let us identify $N\simeq_{\fraks} \bbZ^n$ and $M^{\circ}\simeq_{\fraks} \bbZ^n$.

\par
(a). Let $n=\bfe_i$, where $p^*(n)=B\bfe_i=\bfb_i$.
Then, we have $\d(\bfe_i)=\d_i$.
Let $f=1+x^{\bfb_i}=1+\hat{y}_i$.
Then, the  automorphism $\frakp_{n,f}$ is given by
\begin{align}
\frakp_{n,f}(x^{\bfm})=x^{\bfm} (1+\hat{y}_i)^{( \d_i \bfe_i, \bfm)_D}.
\quad
(\bfm\in P\subset M^{\circ}).
\end{align}
\par
(b).
Let $n=\bfc_{i;t}^+$,
where $p^*(n)=B\bfc_{i;t}^+=\hat\bfc_{i;t}^+$.
We first claim that
\begin{align}
\label{2eq:n'1}
\d(\bfc_{i;t}^+)=\d_i.
\end{align}
We prove it by the induction on $t$ along $\bbT_n$ starting from $t_0$.
For $t=t_0$, this is the case  (a) above.
Assume that \eqref{2eq:n'1} holds for some $t$.
Let $t'$ be the vertex that is $k$-adjacent to $t$.
By \eqref{2eq:cmut5}, we have
\begin{align}
\begin{split}
C_{t'}D^{-1}&=C_t(J_k + [\varepsilon_{k;t}B_t]_+^{k \bullet})D^{-1}\\
&=C_tD^{-1}(J_k + [-\varepsilon_{k;t}(B_t)^T]_+^{k \bullet}),
\end{split}
\end{align}
where \eqref{2eq:DBD1} is used for the last equality.
Since $|J_k + [-\varepsilon_{k;t}(B_t)^T]_+^{k \bullet}|=-1$ by Lemma 
\ref{2lem:square1}, each column vector $\d_i\bfc_{i;t'}^+$ belongs to $N_0$.
Furthermore, $\d_i$ is the smallest
positive rational number satisfying this property
because $\d_1\bfc_{1;t'}^+$, \dots, $\d_n\bfc_{n;t'}^+$ are a $\bbZ$-basis
of $N^{\circ}$.
This proves \eqref{2eq:n'1}.
Now, let $f=1+x^{\hat\bfc_{i;t}^+}=1+\hat{y}^{\bfc_{i;t}^+}$.
Then, we have
\begin{align}
\frakp_{n,f}(x^{\bfm})=x^{\bfm} (1+\hat{y}^{\bfc_{i;t}^+})^{( \d_i \bfc_{i;t}^+, \bfm)_D},
\quad
(\bfm\in P\subset M^{\circ}).
\end{align}
Observe that 
\begin{align}
\label{2eq:pq1}
\frakp_{n,f}^{-\varepsilon_{i;t}}=\frakq_{i;t},
\end{align}
where  $\frakp_{n,f}^{-1}$ is the inverse of $\frakp_{n,f}$,
and $\frakq_{i;t}$ is the one  in  \eqref{2eq:fp1}
 viewed as
an automorphism of
$\bbk[[P]]$.
The factor $-\varepsilon_{i;t}$ in  \eqref{2eq:pq1}
is important for the integrity of the whole picture.
\end{ex}

For any  $n\in N^+$, we define
a hyperplane in $M_{\bbR}$
\begin{align}
n^{\perp}:=
\{ m\in M_{\bbR} \mid
\langle n, m \rangle = 0 \}
\subset M_{\bbR}.
\end{align}
We say that $n\in N^+$ is \emph{primitive\/} if \index{primitive}
there is no pair $j \in \bbZ_{> 1}$ and  $n'\in N^+$
such that $n=j n'$.
Let $N_{\mathrm{pr}}^+$ be the set of all primitive elements
in $N$.
 
Now we introduce the  fundamental ingredient of scattering diagrams.

\begin{defn}[Wall] \index{wall}
\label{2defn:wall1}
A \emph{wall\/} for a seed $\fraks$
 is a triplet $\bfw=(\frakd, f)_{n}$,
 where
 \begin{itemize}
\item
$n\in N^+_{\rm pr}$,
\item
$\frakd\subset n^{\perp}$ is a (not necessarily strongly convex) cone in $M_{\bbR}$ of codimension one,
\item
$f$ is an element of  $\bbk[[x^{p^*(n)}]]\subset \bbk[[P]]$ with
 constant term 1,
\begin{align}
\label{2eq:fd1}
f=1+\sum_{k=1}^{\infty} c_k x^{kp^*(n)}.
\end{align}
\end{itemize}
The cone $\frakd$  is called the \emph{support\/} of a wall. \index{support!of a wall}
No names were explicitly given for $f$  and $n$ in \cite{Gross14}.
Since this is inconvenient, here we call them the \emph{wall function\/} \index{wall function}
and  the \emph{normal vector\/} of a wall. \index{normal vector (of a wall)}
We say that a wall $\bfw=(\frakd, f)_{n}$ is \emph{incoming\/} if \index{wall!incoming}
\begin{align}
\label{2eq:outgoing1}
p^*(n)\in \frakd.
\end{align}
Otherwise, we say it is \emph{outgoing}. \index{wall!outgoing}
\end{defn}

\begin{rem}
Since $\{n,n\}=0$ by skew-symmetry, we have
\begin{align}
p^*(n)\in n^{\perp}.
\end{align}
\end{rem}

\begin{ex}[continued]
\label{2ex:wall1}
Let us identify $N\simeq_{\fraks} \bbZ^n$ and $M^{\circ}\simeq_{\fraks} \bbZ^n$.

\par
(a). 
Let $n=\bfe_i$, $\frakd=\bfe_i^{\perp}$, and $f=1+\hat{y}_i$.
Then, $\bfb_i =B\bfe_i\in \frakd$, and $\bfw_{i;t_0}=(\bfe_i^{\perp},1+\hat{y}_i)_{\bfe_i}$ is an \emph{incoming\/} wall.
\par
(b).
Let $n=\bfc_{i;t}^+=\varepsilon_{i;t}\bfc_{i;t}$, 
$\frakd=\sigma_i(G_t)$, and $f=1+\hat{y}^{\bfc_{i;t}^+}$.
Firstly, $\bfc_{i;t}^+$ is primitive due to the unimodularity
\eqref{2eq:unimod1} of $C$-matrices.
Secondly, $\bfc_{i;t}^+$ is a normal vector of $\sigma_i(G_t)$
thanks to the duality in Proposition \ref{2prop:duality3}.
Therefore, $\bfw_{i;t}=(\sigma_i(G_t), 1+\hat{y}^{\bfc_{i;t}^+})_{\bfc_{i;t}^+}$ is a wall.
Moreover,  due to
Proposition \ref{2prop:outgoing1},
 if $\bfc_{i;t}^+\neq \bfe_{\ell}$ for any $\ell$,   it is an \emph{outgoing\/} wall.
Note that, if $t,t'\in \bbT_n$ are $k$-adjacent,
 then two walls $\bfw_{k;t}$ and $\bfw_{k;t'}$
 are identical.
\end{ex}

\subsection{Scattering diagrams}

Let $\widehat{J}$ be the maximal ideal of $\bbk[[P]]$
generated by $P\setminus \{0\}$.
 We say that $f=1+\sum_{m\in P, m\neq 0} c_m x^m
  \in \bbk[[P]]$ is \emph{trivial modulo $\widehat{J}^{\ell}$} if its image in $\bbk[[P]]/\widehat{J}^{\ell}$
  is 1.

\begin{ex}[continued]
Under the identification $\bbk[[P]]=\bbk[[\hat{y}_1, \dots, \hat{y}_n]]$,
$f$ is trivial  modulo $\widehat{J}^{\ell}$ if and only if
$f-1$ contains no monomial in $\hat{y}_1$, \dots, $\hat{y}_n$  whose total degree is 
less than $\ell$.
\end{ex}

Now we can give the definition of a scattering diagram.
\begin{defn}[Scattering diagram] \index{scattering diagram}
\label{2defn:scat1}
A \emph{scattering diagram\/} $\frakD$  for
a seed $\fraks$  is a  collection
of walls $\{ \bfw_{\lambda}=(\frakd_{\lambda}, f_{{\lambda}})_{n_{\lambda}}\}_{\lambda\in \Lambda}$
 with respect to
a seed $\fraks$,
where  $\Lambda$ is a  finite or countably infinite index set,
satisfying the following \emph{finiteness condition\/}: \index{finiteness condition}
\begin{itemize}
\item
For each positive integer $\ell$, there are only finitely many
walls whose wall functions are not trivial modulo $\widehat{J}^{\ell}$.
\end{itemize}
\end{defn}

\begin{defn}[Support and singular locus]
For a scattering diagram $\frakD=\{ (\frakd_{\lambda}, f_{{\lambda}})_{n_{\lambda}}\}_{\lambda\in \Lambda}$,
the \emph{support\/} and the \emph{singular locus\/} of $\frakD$ are defined by \index{support!of a scattering diagram}
\index{singular locus (of a scattering diagram)}
\begin{align}
\mathrm{Supp}(\frakD)&=\bigcup_{\lambda\in \Lambda} \frakd_{\lambda},\\
\mathrm{Sing}(\frakD)&=\bigcup_{\lambda\in \Lambda} \partial\frakd_{\lambda}
\cup
\bigcup_{\ss \lambda, \lambda'\in \Lambda \atop \ss \dim 
\frakd_{\lambda}\cap \frakd_{\lambda'}=n-2
} \frakd_{\lambda}\cap \frakd_{\lambda'}.
\end{align}
\end{defn}

\begin{ex}[continued]
\label{2ex:scatter1}
Let us identify $N\simeq_{\fraks} \bbZ^n$ and $M^{\circ}\simeq_{\fraks} \bbZ^n$.
For the cluster pattern $\bfSigma_{\fraks,t_0}$,
we define a set of walls
\begin{align}
\label{2eq:wall1}
\frakD(\bfSigma_{\fraks,t_0}):=
\{\bfw_{i;t}=
(\sigma_i(G_t),1+\hat{y}^{\bfc_{i;t}^+})_{\bfc_{i;t}^+}
\mid i=1,\dots,n;\, t\in \bbT_n
\},
\end{align}
where we discard all duplicate walls so that $\frakD(\bfSigma_{\fraks,t_0})$ consists
of mutually distinct walls.
We may take a different label $(i;t)$ freely for  identical walls.
The collection of walls $\frakD(\bfSigma_{\fraks,t_0})$
is not necessarily a scattering diagram, because the finiteness condition  
 is not  guaranteed.
In fact, there are examples where $\frakD(\bfSigma_{\fraks,t_0})$
does  not satisfy the finiteness condition.
See Remark \ref{2rem:finite1}.
\end{ex}

\subsection{Wall-crossing automorphisms}
\begin{defn}[Admissible curve] A curve $\gamma: [0,1] \rightarrow M_{\bbR}$
is \emph{admissible\/} for a scattering diagram $\frakD$
if it satisfies the following properties:
\begin{itemize}
\item $\gamma$ does not intersect $\mathrm{Sing}(\frakD)$.
\item The end points of $\gamma$ are in $M_{\bbR}\setminus \mathrm{Supp}(\frakD)$.
\item It is a smooth curve, and it intersects  $\mathrm{Supp}(\frakD)$ transversally.
\end{itemize}
\end{defn}
\begin{defn}[Wall-crossing automorphism] \index{wall-crossing automorphism}
For any scattering diagram $\frakD$ and any admissible curve $\gamma$,
we define the associated \emph{wall-crossing automorphism along $\gamma$}, $\frakp_{\gamma,\frakD}\in \mathrm{Aut}( \bbk[[P]])$ as follows:
For each degree $\ell$, take all walls whose wall functions are not trivial modulo $\widehat{J}^{\ell}$.
There are only finitely many such walls.
Let $(\frakd_s, f_{s})_{n_s}$ ($s=1,\dots, r$)
be the walls among them which $\gamma$  intersects.
We may assume that the curve $\gamma(t)$ intersects  $\frakd_s$ at  time $t=t_s$
such that
\begin{align}
0<t_1\leq t_2 \leq \dots \leq t_r<1.
\end{align}
For each $s$, we define the \emph{intersection sign} \index{intersection sign}
\begin{align}
\label{2eq:factor1}
\epsilon_s
=
\begin{cases}
1 & \langle n_{s}, \gamma'(t_s)\rangle<0,\\
-1 & \langle n_{s}, \gamma'(t_s)\rangle>0,\\
\end{cases}
\end{align}
where $\gamma'(t)$ is the velocity vector of $\gamma(t)$ at $t$.
Then, we define
\begin{align}
\label{2eq:pdef1}
\frakp_{\gamma,\frakD}^{\ell}&=
\frakp_{n_r,f_{r}}^{\epsilon_r}\circ 
\cdots
\circ
\frakp_{n_2,f_{2}}^{\epsilon_2}\circ 
\frakp_{n_1,f_{1}}^{\epsilon_1}, \\
\label{2eq:pdef2}
\frakp_{\gamma,\frakD}&=
\lim_{\ell\rightarrow\infty}\frakp_{\gamma,\frakD}^{\ell}.
\end{align}
\end{defn}

Note that $\frakp_{\gamma,\frakD}$ only depends on the homotopy class
of $\gamma$ in $M_{\bbR}\setminus\mathrm{Sing}(\frakD)$. 

\begin{rem}
In \cite{Gross14},
 $\frakp_{\gamma,\frakD}$ is called a \emph{path-ordered product\/} of 
  automorphisms $\frakp_{n_s,f_{s}}^{\epsilon_s}$.
\end{rem}

\begin{ex}[continued]
\label{2ex:p1}
Let $\frakD(\bfSigma_{\fraks,t_0})$ be the collection of walls
in Example \ref{2ex:scatter1}.
Even though 
the finiteness condition  
does not hold in general,
the wall-crossing automorphism is still well-defined
for a curve crossing only finitely many walls.
Let $t_0$, $t_1$, \dots, $t_{r+1}=t\in \bbT_n$ be a sequence of vertices such that 
they are sequentially adjacent with edges labeled by $k_0$, $k_1$,\dots, $k_r$.
Accordingly, we consider an admissible curve $\gamma_{t}^{t_0}$ further satisfying
the following conditions:
\begin{itemize}
\item
It is in the interior of the support $|\Delta(\bfG^{t_0})|$ 
of the $G$-fan $\Delta(\bfG^{t_0})$.
\item
It starts in $\sigma^{\circ}(G_{t_{r+1}})$.
\item
It sequentially intersects  walls
\begin{align}
(\sigma_{k_r}(G_{t_r}),1+\hat{y}^{\bfc_{k_r;t_r}^+})_{\bfc_{k_r;t_r}^+},
\dots,
(\sigma_{k_0}(G_{t_0}),1+\hat{y}^{\bfc_{k_0;t_0}^+})_{\bfc_{k_0;t_0}^+}.
\end{align}
\item
It ends in  $\sigma^{\circ}(G_{t_0})$.
\end{itemize}
By Proposition \ref{2prop:duality3},
for each $s=0,\dots,r$,
the $c$-vector $\bfc_{k_s;t_s}$
is inward for the cone $\sigma_{k_s}(G_{t_s})$.
Thus,
when $\gamma_t^{t_0}$ crosses
$\sigma_{k_s}(G_{t_s})$ to \emph{get  into\/} $\sigma(G_{t_s})$,
the velocity  vector $\gamma'$ and the $c$-vector $\bfc_{k_s;t_s}$
are always in the \emph{same\/} direction with respect to $\sigma_{k_s}(G_{t_s})$.
Since
the normal vector of the wall is $n_s=\varepsilon_{k_s;t_s} \bfc_{k_s;t_s}$,
 the factor in \eqref{2eq:factor1}
is given by  $-\varepsilon_{k_s;t_s}$.
Then, by \eqref{2eq:fq1} and \eqref{2eq:pq1},
 the wall-crossing automorphism along $\gamma_t^{t_0}$ is
 identified with $\frakq_{t}^{t_0}$ in \eqref{2eq:fq1}
 as follows:
\begin{align}
\label{2eq:pq2}
\begin{split}
\tilde\frakp_{t}^{t_0}:=
\frakp_{\gamma_t^{t_0},\frakD(\bfSigma_{\fraks,t_0})}
&=
\frakp_{\bfc_{k_0;t_0}^+, 1+\hat{y}^{\bfc_{k_0;t_0}^+}}
^{-\varepsilon_{k_0,t_0}}
\circ
\cdots
\circ
\frakp_{\bfc_{k_r;t_r}^+, 1+\hat{y}^{\bfc_{k_r;t_r}^+}}
^{-\varepsilon_{k_r,t_r}}
\\
&=
\frakq_{k_0;t_0}
\circ
\cdots
\circ
\frakq_{k_r;t_r}
\\
&=\frakq_{t}^{t_0}.
\end{split}
\end{align}
Note that the factor $-\varepsilon_{i;t}$ in \eqref{2eq:pq1} is absorbed
in coordination with the factor in \eqref{2eq:factor1}.
\end{ex}

\begin{defn}[Equivalence/Consistency]
\
\par
\begin{itemize}
\item
Two scattering diagrams $\frakD$ and $\frakD'$ with a common initial seed $\fraks$ are
\emph{equivalent\/} if, for any curve $\gamma$ that is admissible for both $\frakD$ and $\frakD'$, the equality 
$\frakp_{\gamma,\frakD}=\frakp_{\gamma,\frakD'}$ holds. \index{scattering diagram!equivalent}
\item
A scattering diagram $\frakD$ is \emph{consistent\/} if for any admissible curve $\gamma$, \index{scattering diagram!consistent}
the associated wall-crossing automorphism $\frakp_{\gamma,\frakD}$  depends 
only on
the end points of $\gamma$.
\end{itemize}
\end{defn}

\begin{rem}
For a given scattering diagram,
one   obtains equivalent scattering diagrams
by repeating the following procedures:
\begin{itemize}
\item
To join a pair of walls 
$(\frakd_1,f)_n$ and $(\frakd_2,f)_n$
intersecting in their common face $\frakd_1\cap \frakd_2$ of codimension two
into a wall $(\frakd_1\cup \frakd_2,f)_n$,
or to split it conversely.
\item
To join a pair of walls $(\frakd,f_1)_n$ and $(\frakd,f_2)_n$ into a wall $(\frakd,f_1f_2)_n$,
or to split it conversely. One may join infinitely many walls
with a common support  as well
by taking the infinite product of wall functions, or split it conversely.
\item
To add or to remove a trivial wall $(\frakd, 1)$.
\end{itemize}
Applying these operations,
and especially removing all trivial walls,
one can obtain a scattering diagram
whose support is minimal among other equivalent diagrams.
\end{rem}

The following is the first fundamental theorem
 on scattering diagrams
in view of the application to cluster algebra theory.

\begin{thm}[{\cite[Theorems~1.12 \&1.28]{Gross14}}]
\label{2thm:scat1}
Let $\fraks$ be any initial seed for $\Gamma$.
Then, there exists a scattering diagram $\frakD_{\fraks}$
satisfying the following properties:
\begin{itemize}
\item[(a).] $\frakD_{\mathrm{in},\fraks}:=\{ (e_i^{\perp}, 1+\hat{y}_i)_{e_i} \mid i=1,\dots, n\}
\subset \frakD_{\fraks}$.
\item[(b).]  $\frakD_{\fraks}\setminus \frakD_{\mathrm{in},\fraks}$ consists only of outgoing walls, and their normal vectors  are not equal to $e_i$ for any $i$.
\item[(c).]  $\frakD_{\fraks}$ is consistent.
\end{itemize}
Moreover, such a scattering diagram is unique  up to the equivalence.
\end{thm}

We call $\frakD_{\fraks}$ a \emph{cluster scattering diagram\/} (CSD for short) for $\fraks$.
\index{scattering diagram!cluster}
\index{CSD|see{cluster scattering diagram}}
\index{cluster scattering diagram (CSD)}

\begin{rem}
The second statement of (b) is taken from 
\cite[Theorem~1.28, Remark~1.29]{Gross14},
while
all other statements are from \cite[Theorem~1.12]{Gross14}.
\end{rem}

From now on we assume that $\frakD_{\fraks}$ has the minimal support
 among other equivalent diagrams.
 Under the identification $M_{\bbR}\simeq_{\fraks} \bbR^n$,
we define $\frakS(B_{t_0}):=\mathrm{Supp}(\frakD_{\fraks})\subset \bbR^n$.
Now we see that the properties of the support diagram $\frakS(B_{t_0})$ given in
Proposition \ref{2prop:SD1} follow from Definition \ref{2defn:wall1} and
 Theorem \ref{2thm:scat1}.

 \begin{rem}
\label{22rem:equivalence1}
There is some redundancy of fixed data,
because we do not require the coprimeness of
$\d_1$, \dots, $\d_n$ here.
Suppose that the exchange matrix $B$
in \eqref{2eq:initials1}  is indecomposable.
For a fixed data $\Gamma$,
another  fixed data $\Gamma'$ is called a \emph{rescaling\/} of $\Gamma$ \index{rescaling (of a fixed data)}
if  there is some  rational number $\lambda>0$ such that
the following relations hold:
\begin{align}
N'&= N,\\
\label{22eq:scale1}
\{n_1, n_2 \}'&=\lambda \{n_1, n_2 \},\\
\label{23eq:scale2}
(N^{\circ})'&= \lambda^{-1} N^{\circ},\\
\label{23eq:scale3}
\d'_i &= \lambda^{-1} \d_i,\\
\label{23eq:scale4}
(M^{\circ})'&= \lambda M^{\circ}.
\end{align}
For such a pair,
 a seed $\fraks$ for $\Gamma$ is also regarded as a seed for $\Gamma'$.
 Let $\phi_{\lambda}: \bbk[[P]]\rightarrow {\bbk[[P']]}$, $x^m \mapsto x^{\lambda m}$
 be the induced algebra isomorphism.
 Then, we have, for $m\in P$,
 \begin{align}
\begin{split}
(\frakp'_{n,\phi_{\lambda}(f)}\circ \phi_{\lambda})(x^{ m})&=
\frakp'_{n,\phi_{\lambda}(f)}(x^{\lambda m})=
x^{\lambda m} \phi_{\lambda}(f)^{\langle \d'_0(n)n, \lambda m\rangle}
\\
&=x^{\lambda m}  \phi_{\lambda}(f)^{\langle \d(n)n,  m\rangle}
=( \phi_{\lambda}\circ \frakp_{n,f})(x^m).
\end{split}
\end{align}
Thus,  cluster scattering diagrams $\frakD_{\fraks}$ and  $\frakD'_{\fraks}$
for $\Gamma$ and $\Gamma'$, respectively, are identified under the above correspondence.
Moreover,
for  $x_i$, $y_i$, $b_{ij}$
 and
 $x'_i$, $y'_i$, $b'_{ij}$ in
\eqref{2eq:initials1} for $\Gamma$ and $\Gamma'$, respectively,
the following relations hold:
\begin{align}
x'_i = x_i^{\lambda},
\quad y'_i = y_i,
\quad b'_{ij}=b_{ij}. 
\end{align}
In particular, we see that a cluster scattering diagram $\frakD_{\fraks}$ depends only on
the matrix $B=(b_{ij})$ as expected. 
In general, when $B$ is not indecomposable, we may apply the above rescaling to
each indecomposable component of $B$.
\end{rem}

\subsection{Rank 2 examples of consistent scattering diagrams $\frakD_{\fraks}$}

\label{2subsec:rank2}

Let us present the scattering diagrams
$\frakD_{\fraks}$ in Theorem \ref{2thm:scat1}
in the  rank 2 case
under the identification $N\simeq_{\fraks} \bbZ^n$ and $M_{\bbR}\simeq_{\fraks} \bbR^n$,
following \cite{Gross14}.
Using the equivalence in Remark \ref{22rem:equivalence1},
we may assume that $\{e_2,e_1\}=1$ without loosing generality.
We use the same convention and notations in  Section \ref{2subsec:rank2G1}.

(I). Finite type.
The scattering diagram $\frakD_{\fraks}$ is equivalent to
$\frakD(\bfSigma_{\fraks,t_0})$ constructed in
Example \ref{2ex:scatter1}.
In other words, $\frakD(\bfSigma_{\fraks,t_0})$ itself is a consistent
scattering diagram.
\par
(a). Type $A_2$. Let
\begin{align}
\d_1=\d_2=1;\
\hat{y}_1=x_2,\ \hat{y}_2=x_1^{-1}.
\end{align}
The scattering diagram $\frakD(\bfSigma_{\fraks,t_0})$
consists of five distinct walls of the form in  \eqref{2eq:wall1}.
They are rearranged up to the equivalence of scattering diagrams as
 two incoming walls  and one outgoing walls as
 \begin{align}
 (\bfe_1^{\perp}, 1+\hat{y}_1)_{\bfe_1},
 \
 (\bfe_2^{\perp}, 1+\hat{y}_2)_{\bfe_2}, 
 \
 (\sigma((1,-1)), 1+\hat{y}_1\hat{y}_2)_{(1,1)}.
 \end{align}
 This is the scattering diagram $\frakD_{\fraks}$,
 which is
 depicted in Figure \ref{2fig:scat1}.
 The only property to be checked is the following consistency condition
 \begin{align}
 \label{2eq:consis1}
 \frakp_{\gamma_1, \frakD_{\fraks}}=
 \frakp_{\gamma_2, \frakD_{\fraks}}
 \end{align}
 for the admissible curves $\gamma_1$ and $\gamma_2$ in
 Figure \ref{2fig:scat1}.
 One can check it easily by the direct calculation as follows:
 \begin{figure}
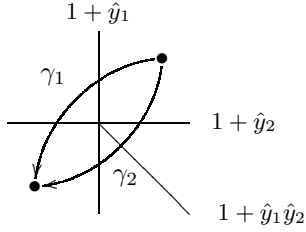

\centering
\leavevmode
\xy
0;/r1.2mm/:,
(0,12)*{\text{\small $ 1+\hat{y}_1$}};
(16,0)*{\text{\small $1+\hat{y}_2$}};
(18,-10)*{\text{\small $1+\hat{y}_1\hat{y}_2$}};
(-5,5)*{\gamma_1};
(3,-6)*{\gamma_2};
(7,7)*+{\bullet};(-7,-7)*+{\bullet};
(7,7)*+{};(-7,-7)*+{}
   **\crv{(0,6.7)&(-6.7,0)}
          ?>*\dir{>};
       (7,7)*+{};(-7,-7)*+{}
       **\crv{(6.7,0)&(0,-6.7)}
       ?>*\dir{>};
(0,0)="A",
\ar@{-}"A"+(0,0); "A"+(10,0)
\ar@{-} "A"+(0,0); "A"+(0,10)
\ar@{-} "A"+(0,0); "A"+(-10,0)
\ar@{-} "A"+(0,0); "A"+(0,-10)
\ar@{-} "A"+(0,0); "A"+(10,-10)
\endxy
\caption{Consistent scattering diagram $\frakD_{\fraks}$ of Type $A_2$.}
\label{2fig:scat1}
\end{figure}
 \begin{align}
 \begin{split}
  \frakp_{\gamma_1, \frakD_{\fraks}}(x_1)
  &=
    \frakp_{\bfe_2, 1+\hat{y}_2}( \frakp_{\bfe_1, 1+\hat{y}_1}(x_1))
    \\
   &=
    \frakp_{\bfe_2, 1+\hat{y}_2}( x_1(1+\hat{y}_1))
    \\
    &=
    x_1(1+\hat{y}_1(1+\hat{y}_2))
    \\
    &=
    x_1(1+\hat{y}_1+\hat{y}_1\hat{y}_2).
    \end{split}
    \end{align}
    \begin{align}
    \begin{split}
      \frakp_{\gamma_2, \frakD_{\fraks}}(x_1)   
  &=
    \frakp_{\bfe_1, 1+\hat{y}_1}(\frakp_{(1,1), 1+\hat{y}_1\hat{y}_2}( \frakp_{\bfe_2, 1+\hat{y}_2}(x_1)))
    \\
      &=
    \frakp_{\bfe_1, 1+\hat{y}_1}(\frakp_{(1,1), 1+\hat{y}_1\hat{y}_2}( x_1))
    \\
          &=
    \frakp_{\bfe_1, 1+\hat{y}_1}( x_1(1+\hat{y}_1\hat{y}_2))
    \\
    &=
    x_1(1+\hat{y}_1)(1+\hat{y}_1\hat{y}_2 (1+\hat{y}_1)^{-1})
    \\
        &=
    x_1(1+\hat{y}_1+\hat{y}_1\hat{y}_2).
 \end{split}
 \end{align}
 Similarly,
  \begin{align}
 \begin{split}
  \frakp_{\gamma_1, \frakD_{\fraks}}(x_2)
  &=
    \frakp_{\bfe_2, 1+\hat{y}_2}( \frakp_{\bfe_1, 1+\hat{y}_1}(x_2))
    \\
   &=
    \frakp_{\bfe_2, 1+\hat{y}_2}( x_2)
    \\
    &=
    x_2(1+\hat{y}_2).
        \end{split}
    \end{align}
    \begin{align}
    \begin{split}
      \frakp_{\gamma_2, \frakD_{\fraks}}(x_2)   
  &=
    \frakp_{\bfe_1, 1+\hat{y}_1}(\frakp_{(1,1), 1+\hat{y}_1\hat{y}_2}( \frakp_{\bfe_2, 1+\hat{y}_2}(x_2)))
    \\
      &=
    \frakp_{\bfe_1, 1+\hat{y}_1}(\frakp_{(1,1), 1+\hat{y}_1\hat{y}_2}
    ( x_2(1+\hat{y}_2)))
    \\
          &=
    \frakp_{\bfe_1, 1+\hat{y}_1}( x_2(1+\hat{y}_1\hat{y}_2)
    (1+\hat{y}_2(1+\hat{y}_1\hat{y}_2)^{-1})
    )
    \\
          &=
    \frakp_{\bfe_1, 1+\hat{y}_1}( x_2    (1+\hat{y}_2+\hat{y}_1\hat{y}_2)
    )
    \\
    &=
 x_2    (1+\hat{y}_2(1+\hat{y}_1)^{-1}+\hat{y}_1\hat{y}_2(1+\hat{y}_1)^{-1})
    \\
        &=
    x_2(1+\hat{y}_2).
 \end{split}
 \end{align}

It is also instructive to give an alternative derivation of
\eqref{2eq:consis1}
in view of Proposition \ref{2prop:detrop1}.
 Recall the action of a permutation $\sigma\in S_n$
 in \eqref{2eq:sigmax1}--\eqref{2eq:sigmag1}.
 We know from \eqref{2eq:A2G1} that
 \begin{align}
 \label{2eq:pentagon1}
 G_{t_2}=s_{12} G_{t_{-3}},
 \end{align}
 where $s_{12}$ is the transposition of 1 and 2.
Also, the following property (the \emph{pentagon periodicity\/}) is the first thing we learn in cluster algebra theory.
\index{pentagon periodicity}
 \begin{align}
 \label{2eq:pentagon2}
  \bfx_{t_2}=s_{12} \bfx_{t_{-3}}.
 \end{align}
Therefore, by  Proposition \ref{2prop:detrop1}, we have  (for $x$-variables)
\begin{align}
 \frakq_{t_2}^{t_0}=\frakq_{t_{-3}}^{t_0}.
\end{align}
Thus, by \eqref{2eq:pq2},
\begin{align}
\label{2eq:const1}
 \tilde\frakp_{t_2}^{t_0}=\tilde\frakp_{t_{-3}}^{t_{0}}.
\end{align}
Since
$\frakp_{\gamma_1, \frakD_{\fraks}}=( \tilde\frakp_{t_2}^{t_0})^{-1}$
and
$\frakp_{\gamma_2, \frakD_{\fraks}}=( \tilde\frakp_{t_{-3}}^{t_0})^{-1}$,
 we have
\eqref{2eq:consis1}.
Moreover, it clarifies  the meaning of the  consistency condition
\eqref{2eq:consis1};
namely, \emph{it guarantees the detropicalization  of $x$-variables}.
\par
(b). Type $B_2$.
Let
\begin{align}
\d_1=1, \ \d_2=2;\
\hat{y}_1=x_2^2,\ \hat{y}_2=x_1^{-1}.
\end{align}
The situation is the same as type $A_2$,
and the outgoing walls are
\begin{align}
\left(
\sigma((1,-2)),
1+\hat{y}_1\hat{y}_2
\right)_{(1,1)},
\
\left(
\sigma((1,-1)),
1+\hat{y}_1\hat{y}_2^2
\right)_{(1,2)}.
\end{align}
\par
(c). Type $G_2$.
Let
\begin{align}
\d_1=1, \ \d_2=3;\
\hat{y}_1=x_2^3,\ \hat{y}_2=x_1^{-1}.
\end{align}
The situation is the same as type $A_2$,
and the outgoing walls are
\begin{align}
\begin{split}
\left(
\sigma((1,-3)),
1+\hat{y}_1\hat{y}_2
\right)_{(1,1)},
\
\left(
\sigma((1,-2)),
1+\hat{y}_1^2\hat{y}_2^3
\right)_{(2,3)},
\\\left(
\sigma((2,-3)),
1+\hat{y}_1\hat{y}_2^2
\right)_{(1,2)},
\
\left(
\sigma((1,-1)),
1+\hat{y}_1\hat{y}_2^3
\right)_{(1,3)}.
\end{split}
\end{align}
\par
(II). Infinite type.
The consistent scattering diagram
$\frakD_{\fraks}$
is an extension of $\frakD(\bfSigma_{\fraks,t_0})$
such that
the supports of  additional walls are in the complement
of $|\Delta(\bfG^{t_0})|$.
In particular, $\frakD(\bfSigma_{\fraks,t_0})$ is a scattering diagram but it is not consistent.
\par
(d). Type $A_1^{(1)}$.
 Let
\begin{align}
\d_1=\d_2=2;\
\hat{y}_1=x_2^2,\ \hat{y}_2=x_1^{-2}.
\end{align}
There is one additional wall to $\frakD(\bfSigma_{\fraks,t_0})$ \cite[Example~1.15]{Gross14}, which is
\begin{align}
\label{2eq:A11}
\Biggl(
\sigma((1,-1)),
\left(
\sum_{k=0}^{\infty}
\hat{y}_1^k\hat{y}_2^k
\right)^2
\Biggr)_{(1,1)}.
\end{align}

(e). Type $A_2^{(2)}$.
Let
\begin{align}
\d_1=1, \ \d_2=4;\
\hat{y}_1=x_2^4,\ \hat{y}_2=x_1^{-1}.
\end{align}
There is one additional wall to $\frakD(\bfSigma_{\fraks,t_0})$ \cite[Theorem 3.4.]{Reading18}, which is
\begin{align}
\Biggl(
\sigma((1,-2)),
(1+\hat{y}_1\hat{y}_2^{2})
\left(
\sum_{k=0}^{\infty}
\hat{y}_1^{k}\hat{y}_2^{2k}
\right)^2
\Biggr)_{(1,2)}.
\end{align}
\par
(f). Non-affine type. There are infinitely many additional walls
inside the irrational cone spanned by  $\bfv$ and $\bfv'$
  in \eqref{2eq:ds1}.
The region is  informally called the \emph{Badlands}. We may also call it the \emph{dark side\/} \index{Badlands}
from the viewpoint of cluster patterns.
The explicit description of walls therein is not yet known.
However, for the skew-symmetric case $\d_1=\d_2$,
the existence of walls for every rational slope was proved in \cite[Example 7.10]{Davison19}.

\newpage
\section{Scattering diagrams and $F$-polynomials}
In this section we show how $F$-polynomials  and $x$-variables are related with
the consistent scattering diagram $\frakD_{\fraks}$.
The Laurent positivity  
follows from the positivity of theta functions in \cite{Gross14}.

\subsection{Mutations of scattering diagrams}

We continue to assume that  a fixed data $\Gamma$ satisfies Injectivity Assumption
throughout the section.

The following definition originates in \cite[\S2.1]{Fock03}.
\begin{defn}[Seed mutation] \index{mutation!seed (for a fixed data)}
Let $\fraks=(e_1,\dots,e_k)$ be any seed for a fixed data $\Gamma$.
For any $k=1,\dots,n$,
we define a new seed $\mu_k(\fraks)=\fraks'=(e_1',\dots,e_n')$
by
\begin{align}
\label{2eq:emut1}
e'_i
=
\begin{cases}
-e_k
& i=k,\\
e_i +[b_{ki}]_+ e_k & i\neq k,
\end{cases}
\quad
b_{ij}=\{\delta_i e_i,e_j\},
\end{align}
which is called the \emph{mutation\/} of $\fraks$ in direction $k$.
\end{defn}

The transformation \eqref{2eq:emut1} is written by a matrix
 form $P$ in Lemma \ref{2lem:square1}.
Thus, $e_1',\dots,e_n'$ are certainly a basis of $N$.
Also,
under the mutation \eqref{2eq:emut1},
The basis $\d_1e_1,\dots, \d_n e_n$ of $N^{\circ}$ mutates as
\begin{align}
\label{2eq:emut2}
\d_i e'_i
=
\begin{cases}
-\d_ke_k
& i=k,\\
\d_i e_i +[-b_{ik}]_+ \d_k e_k & i\neq k,
\end{cases}
\end{align}
where we used the skew-symmetry $b_{ki}\d_i=-b_{ik}\d_k$.
Thus, $\d_1e_1',\dots,\d_n e_n'$ are a basis of $N^{\circ}$
by the same reason as above.
Therefore, $\fraks'$ is certainly a seed in the sense of Definition
\ref{2defn:seed2}.

On can easily confirm that the mutation \eqref{2eq:emut2}
induces the following mutation for the basis $f_1,\dots,f_n$ of $M^{\circ}$,
which are dual to $\d_1 e_1, \dots, \d_n e_n$:
\begin{align}
\label{2eq:fmut1}
f'_i
=
\begin{cases}
\displaystyle
-f_k+\sum_{j=1}^n [-b_{jk}]_+ f_j
& i=k,\\
f_i  & i\neq k.
\end{cases}
\end{align}
Also, we define the matrix $B'=(b'_{ij})$, $b'_{ij}=\{ \d_i e'_i, e'_j\}$
as in \eqref{2eq:initials1}.
Then, $B'$ coincides with the matrix mutation of $B$ in direction
 $k$ by \eqref{2eq:bmut2} with $\varepsilon=1$.

\begin{rem}
The transformation \eqref{2eq:fmut1} already appeared in
\eqref{2eq:ff1}
under the identification $M_{\bbR} \simeq_{\fraks} \bbR^n$.
\end{rem}

\begin{rem}
\label{2rem:mu1}
For $\fraks'=\mu_k(\fraks)$,
$\mu_k(\fraks')= \fraks$ does not hold,
 because $b'_{ki}=-b_{ki}$.
Namely, $\mu_k$ is \emph{not\/} an involution on the set of all seeds for a fixed data $\Gamma$.
This causes a technical problem, for example, in the proof of
the forthcoming Theorem \ref{2thm:clasterscatter1}.
To remedy the situation,
we also define another mutation of a seed $\fraks'=\mu_{k}^-(\fraks)$ by
\begin{align}
\label{2eq:emut3}
e'_i
=
\begin{cases}
-e_k
& i=k,\\
e_i +[-b_{ki}]_+ e_k & i\neq k.
\end{cases}
\end{align}
Again, we define the matrix $B'=(b'_{ij})$, $b'_{ij}=\{ \d_i e'_i, e'_j\}$.
Then, $B'$ coincides with the matrix mutation of $B$ in direction
 $k$ by \eqref{2eq:bmut2} with $\varepsilon=-1$.
 Moreover, we have
\begin{align}
\label{2eq:muinv1}
 \mu_k^- ( \mu^+_k ( \fraks))=
 \mu_{k}^+ ( \mu_{k}^-(\fraks))=
 \fraks,
\end{align}
where $\mu_k^+=\mu_k$. We call $\mu_k^{\pm}$ the \emph{signed mutations\/} of $\fraks$. \index{mutation!signed}
They appeared, for example, in \cite{Iwaki14a}.
\end{rem}

For any  seed $\fraks=(e_1,\dots,e_n)$ and $k=1,\dots,n$,
we define a linear map
\begin{align}
\label{2eq:Sk4}
\begin{matrix}
S_{k;\fraks}:& M_{\bbR}& \rightarrow &M_{\bbR}\\
& m & \mapsto&
m+ \langle \d_k e_k, m\rangle p^*(e_k).
\end{matrix}
\end{align}
and its dual
\begin{align}
\label{2eq:Sk6}
\begin{matrix}
S^*_{k;\fraks}:& N& \rightarrow &N\\
& n & \mapsto&
n+ \{ \d_k e_k, n\}e_k.
\end{matrix}
\end{align}
The following equality can be easily verified:
\begin{align}
\langle S^*_{k;\fraks}(n), S_{k;\fraks}(x)\rangle
=\langle n, x\rangle.
\end{align}
Let
\begin{align}
 \calH_{k,+}=\{ m\in M_{\bbR} \mid \langle  e_k, m\rangle\geq 0\},
 \quad
 \calH_{k,-}=\{ m\in M_{\bbR} \mid \langle  e_k, m\rangle\leq 0\}.
\end{align}
Then, following \cite{Gross14},
we  define a piecewise-linear map
\begin{align}
\label{2eq:Tk4}
\begin{matrix}
T_{k;\fraks}:& M_{\bbR}& \rightarrow &M_{\bbR}\\
& m & \mapsto& 
\begin{cases}
S_{k;\fraks}(m) & m \in  \calH_{k,+},\\
m & m \in  \calH_{k,-}.
\end{cases}
\end{matrix}
\end{align}

\begin{rem} Under the identification $M_{\bbR} \simeq_{\fraks} \bbR^n$, it coincides with the map $T_{k;t_0}$
in \eqref{2eq:Tk1}.
\end{rem}

\begin{defn}[Mutation of $\frakD_{\fraks}$]
\label{2defn:mutD1}
For the consistent scattering diagram $\frakD_{\fraks}$ in  Theorem \ref{2thm:scat1} with minimal support,
we define a collection of walls $T_k(\frakD_{\fraks})$,
called the \emph{mutation of $\frakD_{\fraks}$ in direction $k$}, \index{mutation!of a cluster scattering diagram}
by collecting all walls obtained below:
\par
(a).
For each wall $(\frakd, f)_n$ of $\frakD_{\fraks}$
other than $(e_k^{\perp}, 1+\hat{y}_k)_{e_k}$,
let $\frakd_+=\frakd\cap \calH_{k,+}$,
and $\frakd_-=\frakd\cap \calH_{k,-}$.
We split the wall $(\frakd, f)_n$ into two walls
\begin{align}
(\frakd_+, f)_n,
\quad
(\frakd_-, f)_n,
\end{align}
where, if one of  $\frakd_+$ and  $\frakd_-$ is codimension more than 1,
we throw that part away.
Then,
we replace $(\frakd_+, f)_n$ with a new wall
$(S_{k;\fraks}(\frakd_+), S_{k;\fraks}(f))_{S^*_{k;\fraks}(n)}$,
where
 $S_{k;\fraks}(f)$ is the formal power series obtained from
$f$ by replacing each term $x^m$ ($m\in P$) with $x^{S_{k;\fraks}(m)}$,
while we leave $(\frakd_-, f)_n$ as it is.
\par
(b). Replace $(e_k^{\perp}, 1+\hat{y}_k)_{e_k}$ with  $(e'_k{}^{\perp}, 1+\hat{y}'_k)_{e'_k}$,
where $e'_k{}^{\perp}=e_k{}^{\perp}$ and $\hat{y}'_k=\hat{y}_k^{-1}$
by \eqref{2eq:emut1}.
\par
To summarize, in both (a) and (b),  each support of a  wall (after splitting) is transformed by
the piecewise-linear map $T_{k;\fraks}$.
On the other hand, each wall function in  (a)
obeys the same \emph{linear map\/} applied to its support,
while the case (b) is exceptional.
\end{defn}

\begin{rem}
\label{2rem:wall1}
Here we encounter a temporal  problem that the walls above are not
walls for the initial seed $\fraks$ by the following reasons:
\begin{itemize}
\item
The normal vectors of
walls $T_{k;\fraks}(\frakd)$ may not  belong to 
$N^+=N_{\fraks}^+$  in general,
because the coefficients of $e_k$ might be negative.
\item
Accordingly,
the wall functions
 $T_{k;\fraks}(f)$ may not  belong to 
$ \bbk[[P]]$ in general.
In particular,
the wall function
 $1+\hat{y}'_k=1+\hat{y}_k^{-1}$ does not  belong to 
$ \bbk[[P]]$.
\end{itemize}
These problems will be harmonically solved in the next theorem.
\end{rem}

Here is the second fundamental theorem on scattering diagrams
in view of the application to cluster algebra theory.

\begin{thm}[{\cite[Theorem~1.24]{Gross14}}]
\label{2thm:scat2}
The collection of walls $T_k(\frakD_{\fraks})$
is a scattering diagram for the seed $\fraks'=\mu_k(\fraks)$.
Moreover, $T_k(\frakD_{\fraks})$ is equivalent to
the consistent scattering diagram $\frakD_{\fraks'}$.
\end{thm}

Let us explain how Proposition
\ref{2prop:SD2} follows from this theorem.

We do the identification $M_{\bbR}\simeq_{\fraks} \bbR^n$ for
$\frakD_{\fraks}$,
and  $M_{\bbR}\simeq_{\fraks'} \bbR^n$ for
$\frakD_{\fraks'}$, respectively.
Under the change of a basis of $M_{\bbR}$
in \eqref{2eq:fmut1},
two identifications are related by the isomorphism
$\eta_{t_0}^{t_1}$ in \eqref{2eq:eta1}.
All relevant  maps are summarized  in the following
 commutative diagram,
 where
 $\varphi_{t_0}^{t_1}
 =\eta_{t_0}^{t_1}\circ T_{k;t_0}$
as in Proposition \ref{2prop:Gfan2}:
\begin{align}
\label{2eq:Tphi1}
  \xymatrix{
    M_{\bbR} \ar[r]^{T_{k;\fraks} } \ar[d]^{\simeq_{\fraks}} & M_{\bbR}  \ar[d]^{\simeq_{\fraks}}  \ar[dr]^{\simeq_{\fraks'}}   &\\
\bbR^n \ar[r]^{T_{k;t_0}}
\ar@/_25pt/[rr]^{\varphi_{t_0}^{t_1}} 
& \bbR^n \ar[r]^{\eta_{t_0}^{t_1}}& \bbR^n.
 }
\end{align}
Then, for the support $\frakd$ of each wall of $\frakD_{\fraks}$,
the map $\frakd\mapsto T_{k;\fraks}(\frakd)$
in $M_{\bbR}$
is identified with the map
$\frakd\mapsto \varphi_{t_0}^{t_1}(\frakd)$
in $\bbR^n$.
Thus, by forgetting the wall functions, we obtain Proposition
\ref{2prop:SD2}.

\subsection{Relation between $\frakD(\bfSigma_{\fraks,t_0})$ and $\frakD_{\fraks}$}

So far, 
what we know about
the relation 
 between
$\frakD(\bfSigma_{\fraks,t_0})$ and $\frakD_{\fraks}$
are
 the rank 2 examples
and the relations of the supports in Theorem \ref{2thm:SCprove1}.
Below we show  that the  walls of $\frakD(\bfSigma_{\fraks,t_0})$
 are  integrated into a part of
the consistent scattering diagram $\frakD_{\fraks}$
in an essential way.
To do that,
we apply a parallel construction of the mutation of $\frakD_{\fraks}$
to $\frakD(\bfSigma_{\fraks,t_0})$.

Throughout this subsection, we omit the normal vector $n\in N^+$ of a wall
$(\frakd, f)_n$ for simplicity, because it is recovered from $\frakd$ 
or $f$ without ambiguity.

For $t_0\in \bbT_n$ and $k=1,\dots,n$,
we define a linear map
\begin{align}
\label{2eq:Sk5}
\begin{matrix}
S_{k;t_0}:& \bbR^n& \rightarrow &\bbR^n\\
& \bfv & \mapsto&
(I+(B_{t_0})^{\bullet k})\bfv,
\end{matrix}
\end{align}
which correspond to the linear map $S_{k;\fraks}$ in \eqref{2eq:Sk4}
 under the identification
$M_{\bbR}\simeq_{\fraks} \bbR^n$.
Let
$\bbR^n_{k,+}$ and $\bbR^n_{k,-}$
be the ones in \eqref{2eq:Rk1}.
Then, 
the piecewise-linear map  \eqref{2eq:Tk1},
which corresponds to \eqref{2eq:Tk4},
is written as
\begin{align}
\label{2eq:Tk5}
\begin{matrix}
T_{k;t_0}:& {\bbR}^n & \rightarrow &{\bbR}^n \\
& \bfv & \mapsto& 
\begin{cases}
S_{k;t_0}(\bfv) & \bfv \in \bbR^n_{k,+},\\
\bfv & \bfv \in  \bbR^n_{k,-}.
\end{cases}
\end{matrix}
\end{align}

Let us explicitly denote the dependence of $\haty$-variables
to the initial vertex $t_0$ as
\begin{align}
 \hat{y}_{t_0}^{\bfv}:=x^{B_{t_0}\bfv}.
 \end{align}

\begin{defn}[Mutation of $\frakD(\bfSigma_{\fraks,t_0})$]
\label{2defn:mutD2}
For the collection of walls $\frakD(\bfSigma_{\fraks,t_0})$,
let
$T_k(\frakD(\bfSigma_{\fraks,t_0}))$ be the collection of walls 
$T_k(\frakd, f)$
obtained from each wall $(\frakd, f)$ of
$\frakD(\bfSigma_{\fraks,t_0})$
as below.
\par
(a). For each  wall $(\frakd,f)=
(\sigma_i(G_t^{t_0}),1+\hat{y}_{t_0}^{\bfc_{i;t}^{t_0+}})$
such that $ \bfc_{i;t}^{t_0+}\neq\bfe_k$,
we define
\begin{align}
T_k(\frakd,f)
=
\begin{cases}
(S_{k;t_0}(\frakd),S_{k;t_0}(f)) & \frakd\subset \bbR^n_{k,+},
\\
(\frakd,f)
& \frakd\subset \bbR^n_{k,-},
\end{cases}
\end{align}
where
\begin{align}
S_{k;t_0}(
1+\hat{y}_{t_0}^{\bfc_{i;t}^{t_0+}}
)
=
S_{k;t_0}(
1+x^{\hat\bfc_{i;t}^{t_0+}}
)
:=
1+x^{S_{k;t_0}(\hat\bfc_{i;t}^{t_0+})}.
\end{align}
Note that by Proposition \ref{2prop:int1}
either $ \frakd\subset \bbR^n_{k,+}$ or  $\frakd\subset \bbR^n_{k,-}$
occurs.
Thus,
we do not have to divide a wall.
\par
(b). For each wall
$
(\frakd,f)=(\sigma_i(G_t^{t_0}),1+\hat{y}_{t_0}^{\bfe_k} )$
we define
\begin{align}
T_k(\frakd,f)=
(\sigma_i(G_t^{t_0}),1+\hat{y}_{t_0}^{-\bfe_k} ).
\end{align}
\par
To summarize, in both (a) and (b)  each support $\frakd$ is transformed by
the piecewise-linear map $T_{k;t_0}$.
On the other hand, each wall function $f$ in  (a)
obeys the same \emph{linear map\/} applied to its support,
while the case (b) is exceptional.
\end{defn}

\begin{rem}
The same remark as Remark \ref{2rem:wall1} is applied to walls in the above.
\end{rem}

Let $t_1\in \bbT_n$ be the vertex that is 
$k$-adjacent to  the  initial vertex $t_0$.
Let $\eta_{t_0}^{t_1}$ be the linear map in \eqref{2eq:eta1}.
For any wall 
$(\frakd,f)\in T_k(\frakD(\bfSigma_{\fraks,t_0}))$,
we define a new wall by
\begin{align}
\label{2eq:walleta1}
 \eta_{t_0}^{t_1}(\frakd,f=1+x^{\bfv}):=
(\eta_{t_0}^{t_1}(\frakd), 1+x^{\eta_{t_0}^{t_1}(\bfv)}).
\end{align}

We have a parallel result to Theorem \ref{2thm:scat2}.
Also, this is an upgrade of
Proposition \ref{2prop:Gfan1} including wall functions.

\begin{prop}
\label{2prop:cscat2}
The walls  $ \eta_{t_0}^{t_1}(\frakd,f)$
in \eqref{2eq:walleta1} 
are  walls of $\frakD(\bfSigma_{\fraks',t_1})$
with $\fraks'=\mu_k(\fraks)$.
Moreover, it induces the following bijections of walls:
\begin{align}
\label{2eq:wallbj1}
\begin{matrix}
\frakD(\bfSigma_{\fraks,t_0})
&
\buildrel
T_k
\over
\rightarrow
&
T_k(\frakD(\bfSigma_{\fraks,t_0}))
&
\buildrel
\eta_{t_0}^{t_1}
\over
\rightarrow
&
\frakD(\bfSigma_{\fraks',t_1})\\
(\sigma_{i}(G_{t}^{t_0}), 1+\haty_{t_0}^{\bfc_{i;t}^{t_0+}})
&\mapsto
&
T_k(\sigma_{i}(G_{t}^{t_0}), 1+\haty_{t_0}^{\bfc_{i;t}^{t_0+}})
&\mapsto&
(\sigma_{i}(G_{t}^{t_1}), 1+\haty_{t_1}^{\bfc_{i;t}^{t_1+}}),
\end{matrix}
\end{align}
where
\begin{align}
\label{2eq:yxhat1}
\hat{y}_{t_0}^{\bfc_{i;t}^{t_0+}}=x^{\hat\bfc_{i;t}^{t_0+}},
\quad
\hat{y}_{t_1}^{\bfc_{i;t}^{t_1+}}=x^{\hat\bfc_{i;t}^{t_1+}}.
\end{align}

\end{prop}

\begin{proof}
We only need to show  the correspondence
by $\eta_{t_0}^{t_1}$ in \eqref{2eq:wallbj1}.
First we  prove the claim for the supports.
Under the map  $\eta_{t_0}^{t_1}\circ T_k$,
the support $\sigma_i(G_t^{t_0})$
maps to $\eta_{t_0}^{t_1}(T_{k;t_0}(\sigma_i(G_t^{t_0}))$.
By Propositions \ref{2prop:Gfan1} and \ref{2prop:Gfan2},
it equals to $\sigma_i(G_t^{t_1})$.
\par
Next 
we  prove the claim for the wall functions.
Let us consider the case (a) in  Definition \ref{2defn:mutD2},
where $ \bfc_{i;t}^{t_0}\neq\pm\bfe_k$ is assumed.
 Under 
the map $\eta_{t_0}^{t_1}\circ T_k$,
the $\hat{c}^+$-vector $\hat\bfc_{i;t}^{t_0+}
=\varepsilon_{i;t}^{t_0}\hat\bfc_{i;t}^{t_0}$
 obeys the same linear transformations
as its support $\sigma_i(G_t^{t_0})$.
Therefore, by comparing \eqref{2eq:gmut6} and \eqref{2eq:gmut12}, 
its image
equals to
$\varepsilon_{i;t}^{t_0}\hat\bfc_{i;t}^{t_1}$.
Now we  need to compare $\varepsilon_{i;t}^{t_0}$
and $\varepsilon_{i;t}^{t_1}$.
Here we repeat a similar argument 
after \eqref{2eq:Jid2} in the proof of Proposition
\ref{2prop:dual2}.
By \eqref{2eq:cmut6},
$C_t^{t_0}$ and $C_t^{t_1}$ differ only in their $k$th rows.
Since we assume $ \bfc_{i;t}^{t_0}\neq\pm\bfe_k$,
there is at least one nonzero element in the $i$th column
of $C_t^{t_0}$ other than at the $k$-th row.
Thus, we have
$\varepsilon_{i;t}^{t_0}=\varepsilon_{i;t}^{t_1}$,
and we conclude that 
 $\hat\bfc_{i;t}^{t_0+}$
maps to $\hat\bfc_{i;t}^{t_1+}$ as desired.
Now we consider the case (b) in  Definition \ref{2defn:mutD2},
where $ \bfc_{i;t}^{t_0}=\pm\bfe_k$ is assumed.
It is enough to prove the equality
\begin{align}
\label{2eq:ccc1}
\eta_{t_0}^{t_1}(-\hat\bfc_{i;t}^{t_0+})
=
\hat\bfc_{i;t}^{t_1+}.
\end{align}
Again by \eqref{2eq:cmut6},
we have $\bfc_{i;t}^{t_1}=\mp\bfe_k$.
Therefore, $\bfc_{i;t}^{t_0+}=\bfc_{i;t}^{t_1+}=\bfe_k$.
Thus, the equality \eqref{2eq:ccc1} is equivalent to
\begin{align}
\label{2eq:ccc2}
\eta_{t_0}^{t_1}(-\bfb_{k;t_0})
=
\bfb_{k;t_1},
\end{align}
which can be proved as
\begin{align}
 \eta_{t_0}^{t_1}(-\bfb_{k;t_0})=-(J_k+[-B_{t_0}]_+^{\bullet k})\bfb_{k;t_0}
 =-\bfb_{k;t_0}=\bfb_{k;t_1}.
 \end{align}
 (This proof also explains the necessity of the exception in the case (b)
in view of the exceptional change of  tropical signs
 $\varepsilon_{i;t}^{t_0}=-\varepsilon_{i;t}^{t_1}$ therein.)
\end{proof}

\begin{ex}[cf.~Example \ref{2ex:pl1}]
Let us clarify 
Definition  \ref{2defn:mutD2}
and
Proposition 
\ref{2prop:cscat2} explicitly for type $A_2$
based on the convention in Section \ref{2subsec:rank2G1}.
The relevant data  are as follows:
\begin{gather}
B_{t_0}=
\begin{pmatrix}
0 & -1\\
1 & 0
\end{pmatrix},
\quad
B_{t_1}=B_{t_{-1}}=
\begin{pmatrix}
0 & 1\\
-1 & 0
\end{pmatrix},
\\
\haty_{1;t_0}=x_2,
\
\haty_{2;t_0}=x_1^{-1};
\
\haty_{1;t_1}=\haty_{1;t_{-1}}=x_2^{-1},
\
\haty_{2;t_1}=\haty_{2;t_{-1}}=x_1,
\\
S_{1;t_0}=
\begin{pmatrix}
1 & 0\\
1 & 1
\end{pmatrix},
\quad
S_{2;t_0}=
\begin{pmatrix}
1 & -1\\
0 & 1
\end{pmatrix},
\\
\eta_{t_0}^{t_1}=
\begin{pmatrix}
-1 & 0\\
0 & 1
\end{pmatrix},
\quad
\eta_{t_0}^{t_{-1}}=
\begin{pmatrix}
1 & 1\\
0 & -1
\end{pmatrix}.
\end{gather}
The results are depicted  in Figure \ref{2fig:Gmut1},
where $\dot{G}_{t_{0}}^{t}:=T_{k;t_0}({G}_{t_{0}}^{t})$ therein.
See also Figure \ref{2fig:Gmut2}.
\end{ex}

\begin{figure}
\centering
\leavevmode
\begin{xy}
(0,-20)*{\text{$\frakD(\bfSigma_{\fraks,t_0})$}},
(5, 5)*{\text{\small $G_{t_{0}}^{t_0}$}},
(-5, 5)*{\text{\small $G_{t_{1}}^{t_0}$}},
(-5, -5)*{\text{\small $G_{t_{2}}^{t_0}$}},
(3.5, -8.5)*{\text{\small $G_{t_{3}}^{t_0}$}},
(8, -3.5)*{\text{\small $G_{t_{4}}^{t_0}$}},
(2,14)*{\text{\small $1+\hat{y}_{1;t_0}$}};
(18,0)*{ \text{\small $1+\hat{y}_{2;t_0}$}};
(13,-14)*{\text{\small $1+\hat{y}_{1;t_0}\hat{y}_{2;t_0}$}};
(0,0)="A",
\ar "A"+(0,0); "A"+(10,0)
\ar "A"+(0,0); "A"+(0,10)
\ar@{-} "A"+(0,0); "A"+(-10,0)
\ar@{-} "A"+(0,0); "A"+(0,-10)
\ar@{-} "A"+(0,0); "A"+(10,-10)
\end{xy}
\hskip15pt
\begin{xy}
(0,-20)*{\text{$T_1(\frakD(\bfSigma_{\fraks,t_0}))$}},
(3.5, 8.5)*{\text{\small $\dot{G}_{t_{0}}^{t_0}$}},
(-5, 5)*{\text{\small $\dot{G}_{t_{1}}^{t_0}$}},
(-5, -5)*{\text{\small $\dot{G}_{t_{2}}^{t_0}$}},
(5, -5)*{\text{\small $\dot{G}_{t_{3}}^{t_0}$}},
(8, 3)*{\text{\small $\dot{G}_{t_{4}}^{t_0}$}},
(-1,14)*{\text{\small $1+\hat{y}_{1;t_0}^{-1}$}};
(17,14)*{\text{\small $1+\hat{y}_{1;t_0}^{-1}\hat{y}_{2;t_0}$}};
(18,0)*{\text{ \small $1+\hat{y}_{2;t_0}$}};
(0,0)="A"
\ar "A"+(0,0); "A"+(10,0)
\ar "A"+(0,0); "A"+(0,10)
\ar@{-} "A"+(0,0); "A"+(-10,0)
\ar@{-} "A"+(0,0); "A"+(0,-10)
\ar@{-} "A"+(0,0); "A"+(10,10)
\end{xy}
\hskip10pt
\begin{xy}
(0,-20)*{\text{$\frakD(\bfSigma_{\mu_1(\fraks),t_1})$}},
(5, 5)*{\text{\small $G_{t_{1}}^{t_1}$}},
(5, -5)*{\text{\small $G_{t_{2}}^{t_1}$}},
(-5, -5)*{\text{\small $G_{t_{3}}^{t_1}$}},
(-3.5, 8.5)*{\text{\small $G_{t_{0}}^{t_1}$}},
(-8, 3.5)*{\text{\small $G_{t_{4}}^{t_1}$}},
(4,14)*{\text{\small $1+\hat{y}_{1;t_1}$}};
(17,0)*{\text{ \small $1+\hat{y}_{2;t_1}$}};
(-12,14)*{\text{\small $1+\hat{y}_{1;t_1}\hat{y}_{2;t_1}$}};
(0,0)="A",
\ar "A"+(0,0); "A"+(10,0)
\ar "A"+(0,0); "A"+(0,10)
\ar@{-} "A"+(0,0); "A"+(-10,0)
\ar@{-} "A"+(0,0); "A"+(0,-10)
\ar@{-} "A"+(0,0); "A"+(-10,10)
\end{xy}
\vskip5pt
\leavevmode
\hskip113pt
\begin{xy}
(0,-20)*{\text{$T_2(\frakD(\bfSigma_{\fraks,t_0}))$}},
(2, 5)*{\text{\small $\dot{G}_{t_{0}}^{t_0}$}},
(-8, 3.5)*{\text{\small $\dot{G}_{t_{1}}^{t_0}$}},
(-5, -5)*{\text{\small $\dot{G}_{t_{2}}^{t_0}$}},
(3.5, -8.5)*{\text{\small $\dot{G}_{t_{3}}^{t_0}$}},
(8, -3.5)*{\text{\small $\dot{G}_{t_{4}}^{t_0}$}},
(-3,-14)*{\text{\small $1+\hat{y}_{1;t_0}$}};
(18,0)*{ \text{\small $1+\hat{y}_{2;t_0}^{-1}$}};
(14,-14)*{\text{\small $1+\hat{y}_{1;t_0}\hat{y}_{2;t_0}$}};
(0,0)="A",
\ar "A"+(0,0); "A"+(10,0)
\ar@{-} "A"+(0,0); "A"+(-10,10)
\ar@{-} "A"+(0,0); "A"+(-10,0)
\ar@{-} "A"+(0,0); "A"+(0,-10)
\ar@{-} "A"+(0,0); "A"+(10,-10)
\end{xy}
\hskip9pt
\begin{xy}
(0,-20)*{\text{$\frakD(\bfSigma_{\mu_2(\fraks),t_{-1}})$}},
(12, 5)*{\text{\small $G_{t_{4}}^{t_{-1}}=G_{t_{-1}}^{t_{-1}}$}},
(5, -5)*{\text{\small $G_{t_{0}}^{t_{-1}}$}},
(-5, -5)*{\text{\small $G_{t_{1}}^{t_{-1}}$}},
(-3.5, 8.5)*{\text{\small $G_{t_{3}}^{t_{-1}}$}},
(-8.5, 3.5)*{\text{\small $G_{t_{2}}^{t_{-1}}$}},
(6 ,14)*{\text{\small $1+\hat{y}_{1;t_{-1}}$}};
(17,0)*{\text{\small $1+\hat{y}_{2;t_{-1}}$}};
(-13,14)*{\text{\small $1+\hat{y}_{1;t_{-1}}\hat{y}_{2;t_{-1}}$}};
(0,0)="A",
\ar "A"+(0,0); "A"+(10,0)
\ar "A"+(0,0); "A"+(0,10)
\ar@{-} "A"+(0,0); "A"+(-10,0)
\ar@{-} "A"+(0,0); "A"+(0,-10)
\ar@{-} "A"+(0,0); "A"+(-10,10)
\end{xy}
\caption{Mutations of $\frakD(\bfSigma_{\fraks,t_0})$ for type $A_2$.}
\label{2fig:Gmut1}
\end{figure}

To describe the relation between $\frakD(\bfSigma_{\fraks,t_0})$ and $\frakD_{\fraks}$,
we introduce the following notion.
\begin{defn} \index{wall!essential}
For  any wall $(\frakd, f)$  and a scattering diagram $\frakD$  for a common  seed $\fraks$,
we say that   $(\frakd, f)$ is \emph{essential  to $\frakD$\/}
if there is a scattering diagram $\frakD'$ that is equivalent to $\frakD$
satisfying
the following condition:
\begin{itemize}
\item
The wall  $(\frakd, f)$ is a wall of $\frakD'$;
moreover, for any wall $(\frakd', f')$ of $\frakD'$
other than  $(\frakd, f)$, it holds that
$\dim \frakd\cap \frakd' <n-1$.
\end{itemize}

\end{defn}

Combining Theorem \ref{2thm:scat2} and
Proposition \ref{2prop:cscat2},
we obtain the following  relation
between $\frakD(\bfSigma_{\fraks,t_0})$ and $\frakD_{\fraks}$.
(The claim (b) was already given in Theorem \ref{2thm:SCprove1} (a),
but it is included here again for the summary.)

\begin{thm}
\label{2thm:clasterscatter1}
Under the identification $M_{\bbR}\simeq_{\fraks} \bbR^n$,
the following facts hold:
\par
(a).
 Every wall 
$(\sigma_{i}(G_{t}^{t_0}), 1+\hat{y}_{t_0}^{\bfc_{i;t}^{t_0+}})$
of $\frakD(\bfSigma_{\fraks,t_0})$ is essential to 
$\frakD_{\fraks}$.

\par
(b). For any $t\in \bbT_n$, the set $\sigma^{\circ}(G_t^{t_0})$, which is a
chamber of $\mathrm{Supp}(\frakD(\bfSigma_{\fraks,t_0}))$,
is also a chamber of $ \mathrm{Supp}(\frakD_{\fraks})$.
\end{thm}

\begin{proof}
We only need to prove (a).
Basically, we prove the claim 
by the induction on $t_0$.
However, as explained in Remark \ref{2rem:mu1},
the mutation $\mu_k(\fraks)$ in \eqref{2eq:emut1} is not involutive.
So, we pay special attention to it.
Let $t_0,t\in \bbT_n$ be any pair of vertices.
Let $d=d(t_0,t)$ be their distance in $\bbT_n$.
Then,
there is a unique sequence of vertices
$t_0, t_1, \dots, t_{d}=t$
that are sequentially adjacent with
edges labeled by $k_1, \dots, k_d$
 such that 
$d(t_0,t_r)=r$.
First, we assign a seed $\fraks=\fraks_0$ at the vertex $t_0$,
which is the target of the proof.
Then, we assign a seed $\fraks_r$ at the vertex $t_r$
recursively as $\fraks_{r+1}=\mu_{k_{r+1}}^-(\fraks_r)$ by the signed mutation 
in \eqref{2eq:emut3}.
Then, by \eqref{2eq:muinv1}, we have
$\fraks_{r}=\mu_{k_{r+1}}(\fraks_{r+1})$.
Now, the claim (a) follows from the following claim.
\begin{claim}
Let $t$ be as above.
Then, for any $r=0,\dots,d$,
 each wall 
$(\sigma_{i}(G_{t}^{t_r}), 1+\hat{y}_{t_r}^{\bfc_{i;t}^{t_r+}})$
of $\frakD(\bfSigma_{\fraks_r,t_r})$ is essential to 
$\frakD_{\fraks_r}$.
\end{claim}
For each $i=1,\dots,n$,
we prove the claim by the finite induction on $r$  from $d$ to $0$.
\par
First, consider the case $r=d$, namely, $t_d=t$.
We  split the support of the wall $(e_{i}^{\perp}, 1+\hat{y}_{t_d}^{\bfe_{i}})$
of $\frakD_{\fraks_d}$
in Theorem \ref{2thm:scat1}
into  the orthants
\begin{align}
e_{i}^{\perp}=\bigcup_{\kappa_1,\dots,\kappa_{i-1}, \kappa_{i+1},
\dots,\kappa_n=\pm1} \sigma
(\kappa_1e_1, {\dots}, \kappa_{i-1} e_{i-1},  \kappa_{i+1}e_{i+1}, \dots, \kappa_ne_n),
\end{align}
so that the one with $\kappa_1=\dots = \kappa_n=1$
is $\sigma_i(G_{t}^{t_d})=\sigma_i(I)$.
By Theorem \ref{2thm:scat1}, there is no other wall of $\frakD_{\fraks_d}$ whose normal vector is 
$\bfe_{i}$.
Thus, $(\sigma_i(G_{t}^{t_d}), 1+\hat{y}_{t_d}^{\bfe_{i}})$ is essential to $\frakD_{\fraks_d}$.
\par
Next, suppose that the claim is true 
for some $r$.
Namely, for the wall
$(\frakd, f):=(\sigma_{i}(G_{t}^{t_r}), 1+\haty_{t_r}^{\bfc_{i;t}^{t_r+}})$
of $\frakD(\bfSigma_{\fraks_r,t_r})$,
there is a scattering diagram $\frakD'$
that is equivalent to $\frakD_{\fraks_r}$
such that the following holds:
\begin{itemize}
\item
Under the identification $M_{\bbR}\simeq_{\fraks_r} \bbR^n$,
the wall  $(\frakd, f)$ is a wall of $\frakD'$;
moreover, for any wall $(\frakd', f')$ of $\frakD'$
other than  $(\frakd, f)$, it holds that
$\dim \frakd\cap \frakd' <n-1$.
\end{itemize}
We apply the same construction of $T_{k_r}(\frakD_{\fraks_r})$ to
$\frakD'$ and obtain $T_{k_r}(\frakD')$.
Since the construction of $T_{k_r}(\frakD')$ is compatible with
that of $T_{k_r}(\frakD(\bfSigma_{\fraks_r,t_r}))$,
 we have the property:
\begin{itemize}
\item
Under the identification $M_{\bbR}\simeq_{\fraks_r} \bbR^n$,
the wall  $T_{k_r}(\frakd, f)$ is a wall of $T_{k_r}(\frakD')$;
moreover, for any wall $(\frakd', f')$ of $T_{k_r}(\frakD')$
other than  $T_{k_r}(\frakd, f)$, it holds that
$\dim T_{{k_r};t_r}(\frakd)\cap \frakd' <n-1$.
\end{itemize}
By construction, $T_{k_r}(\frakD')$  is equivalent to
$T_{k_r}(\frakD_{\fraks_r})$.
Also,
by Theorem \ref{2thm:scat2},
$T_{k_r}(\frakD_{\fraks_r})$ is equivalent to $\frakD_{\fraks_{r-1}}$,
where we recall $\mu_{k_r}(\fraks_r)=\fraks_{r-1}$.
Thus, 
$T_{k_r}(\frakD')$ is also equivalent to $\frakD_{\fraks_{r-1}}$.
Therefore, the wall  $T_{k_r}(\frakd, f)$ is essential
 to $\frakD_{\fraks_{r-1}}$.
Now, we change the identification
from
 $M_{\bbR}\simeq_{\fraks_r} \bbR^n$
 to $M_{\bbR}\simeq_{\fraks_{r-1}} \bbR^n$
 by the linear isomorphism $\eta_{t_r}^{t_{r-1}}$.
 Then, by Proposition \ref{2prop:cscat2},
 the wall $T_{k_r}(\frakd, f)$ is identified with
 $(\sigma_{i}(G_{t}^{t_{r-1}}), 1+\haty_{t_{r-1}}^{\bfc_{i;t}^{t_{r-1}+}})$.
 Therefore,
we conclude that
the wall $(\sigma_{i}(G_{t}^{t_{r-1}}), 1+\haty_{t_{r-1}}^{\bfc_{i;t}^{t_{r-1}+}})$
 is essential
to $\frakD_{\fraks_{r-1}}$.
This completes the proof
of the claim.
\end{proof}

\begin{rem}
\label{2rem:finite1}
The following remark is due to Nathan Reading.
Theorem \ref{2thm:clasterscatter1} does not mean that
$\frakD(\bfSigma_{\fraks,t_0})$ satisfies the finiteness condition.
For example, consider the initial exchange matrix
\begin{align}
B_{t_0}=
\begin{pmatrix}
0 & -1 & -1\\
1 & 0 & -1\\
1 & 1& 0
\end{pmatrix}.
\end{align}
(This matrix is singular, but it is irrelevant to the point here.)
Then, 
there are 
infinitely many distinct faces $\sigma_i(G_t)$   sharing
the common normal vector $(0,1,0)$,
thus, also with the common wall function.
Therefore,  in this case $\frakD(\bfSigma_{\fraks,t_0})$ is not  a scattering diagram
in the sense of Definition \ref{2defn:scat1}.
However, Theorem \ref{2thm:clasterscatter1} guarantees that
one can combine them into finitely many walls
to satisfy the finiteness condition.
After this, $\frakD(\bfSigma_{\fraks,t_0})$ becomes
a scattering diagram.
\end{rem}

\subsection{$F$-polynomials and wall-crossing automorphisms }

Based on Theorem \ref{2thm:clasterscatter1}, now we can identify
$F$-polynomials with wall-crossing automorphisms
of $\frakD_{\fraks}$.

So far, a wall-crossing automorphism $\frakp_{\gamma,\frakD}$ is defined to
act on $\bbk[[P]]$ by \eqref{2eq:fpm1} as a ring automorphism.
Now we consider a $\bbk[[P]]$-module
$x^m\bbk[[P]]$ for any $m\in M^{\circ}$
such that the action $\frakp_{\gamma,\frakD}(x^m)$
is defined by the same formula \eqref{2eq:fpm1}.

Let us recall the situation in Example \ref{2ex:p1}.
Consider a sequence of vertices $t_0$, $t_1$, \dots, $t_{r+1}=t\in \bbT_n$
such that
they  are sequentially adjacent with edges labeled by $k_0$, \dots, $k_{r}$.
Let $\gamma_t^{t_0}$ be an admissible curve therein.
Namely, $\gamma_t^{t_0}$ starts  in $\sigma^{\circ}(G_t)$,
passing through $G$-cones in $\bbR^n$, and it ends in $\sigma^{\circ}(G_{t_0})$.
Under the identification $M_{\bbR}\simeq_{\fraks} \bbR^n$,
we define
\begin{align}
\label{2eq:fp7}
\frakp_{t}^{t_0}:=\frakp_{\gamma_t^{t_0}, \frakD_{\fraks}}.
\end{align}

The following theorem clarifies the relation
between $F$-polynomials and
 wall-crossing automorphisms in the scattering diagram $\frakD_{\fraks}$.

\begin{thm}[{\cite[Theorem~5.6]{Reading17}, \cite[Theorem~2.9]{Reading18}}]
\label{2thm:F1}
Under the identification $M_{\bbR}\simeq_{\fraks} \bbR^n$,
the following facts hold.
\par
(a).  The wall-crossing  automorphism $\frakp_{t}^{t_0}$ in \eqref{2eq:fp7}
 for $\frakD_{\fraks}$
 coincides with 
the wall-crossing  automorphism $\tilde\frakp_{t}^{t_0}$ in \eqref{2eq:pq2}
 for  $\frakD(\bfSigma_{\fraks,t_0})$.
\par
(b).
The following formula holds
in
$x^{\bfg_{i;t}}\bbk[[P]]$:
\begin{align}
 \label{2eq:fp6}
 \frakp_t^{t_0}(x^{\bfg_{i;t}})
 &=
 x^{\bfg_{i;t}}
 F_{i;t}(\hat\bfy).
  \end{align}
\end{thm}
\begin{proof}
(a).
By  Theorems \ref{2thm:clasterscatter1},
one may assume that,
up to the equivalence,
the walls of $\frakD_{\fraks}$ passed by the above curve $\gamma_t^{t_0}$
all belong to $\frakD(\bfSigma_{\fraks,t_0})$.
Thus, we have
\begin{align}
\tilde\frakp_{t}^{t_0}=
\frakp_{\gamma_t^{t_0}, \frakD(\bfSigma_{\fraks,t_0})}
=\frakp_{\gamma_t^{t_0}, \frakD_{\fraks}}
=\frakp_{t}^{t_0}.
\end{align}
(b).
Recall that $\tilde\frakp_t^{t_0}=\frakq_t^{t_0}$
by \eqref{2eq:pq2}.
Therefore, by \eqref{2eq:fp5}, we have the equality.
\end{proof}

\begin{rem}
\label{2rem:F1}
By Injectivity Assumption, the initial exchange matrix $B$
is nonsingular. Thus, $\hat{y}_1$, \dots, $\hat{y}_n$
are algebraically independent.
Therefore, \eqref{2eq:fp6} uniquely determines  $F$-polynomials
$F_{i;t}(\bfy)$ themselves.
\end{rem}
\subsection{Theta functions and Laurent positivity}
Let us present the outline of the proof of the Laurent positivity
by \cite{Gross14}.

The following object was introduced in \cite{Gross14}
 to provide
 a  combinatorial description of wall-crossing automorphisms.

\begin{defn}[Broken line] \index{broken line}
\label{2defn:broken1}
Let $\frakD$ be any scattering diagram with  respect to  a given initial seed $\fraks$.
Let $m_0\in M^{\circ}\setminus \{0\}$, and let $Q\in M_{\bbR}\setminus
\mathrm{Supp}(\frakD)$.
A \emph{broken line for $m_0$ with endpoint $Q$} is a piecewise-linear
curve $\gamma:(-\infty, 0]\rightarrow M_{\bbR}\setminus
\mathrm{Sing}(\frakD)$ satisfying the following properties:
\begin{itemize}
\item[(1)]
The endpoint $\gamma(0)$ is $Q$.
\item[(2)]
There are numbers $-\infty=t_0 < t_1 < t_2 < \cdots < t_{r+1}=0$ 
$(r\geq 0)$ such that
$\gamma$
 is linear in each interval $I_j=(t_{j},t_{j+1})$,
 while it breaks at each $t_j$  for $j=1,\dots,r$.
 Moreover, $\gamma$ passes  some walls of $\frakD$ at each
 break point $t_j$.
 (It is possible that $\gamma$ passes some walls in an interval $I_j$ without breaking.)
 \item[(3)]
To each $I_j$, a monomial $c_j x^{m_j}\in \bbk[M^{\circ}]$ is attached.
In particular, to $I_0$, a monomial $x^{m_0}$ is attached,
where $m_0$ is the given data.
\item[(4)] The velocity $\gamma'$ is $-m_j$ in the interval $I_j$.
\item[(5)] Let $\gamma_j$ be a segment of $\gamma$ for
the interval
$(t_j-\delta,t_j+\delta)$ with sufficiently small $\delta>0$ such that $\gamma_j$ passes walls
only at $t_j$.
Then, $c_j x^{m_j}$ is a monomial in
$
\frakp_{\gamma_j, \frakD}(c_{j-1}x^{m_{j-1}})$.
\end{itemize}
Finally, we define 
\begin{align}
\mathrm{Mono}(\gamma):=c_r x^{m_r}.
\end{align}
Below we implicitly assume that $Q$ is in a general position so that
it has the maximal set of broken lines in the neighborhood of $Q$.
Namely, no broken line for $Q'$ near $Q$ drops out accidentally by intersecting 
$\mathrm{Sing}(\frakD)$ when $Q'$ approaches to $Q$.
\end{defn}

\begin{rem}
\label{2rem:bend1}
At  $t_j$, the velocity changes
from $-m_j$ to  $-m_{j+1}=-m_j - k_j p^*(n_j)$, where $n_j$ is the common normal vector
of the walls at $t_j$, and $k_j$ is some positive integer.
Since $p^*(n_j)\in n_j^{\perp}$,
the broken line always \emph{crosses\/}  walls,
\emph{not reflects nor stops} at them as stated in  Condition (2).
\end{rem}

\begin{rem}[{\cite[Remark~3.2]{Gross14}}]
\label{2rem:positive1}
In  Condition (5)  of Definition \ref{2defn:broken1},
suppose that $\gamma_j$ crosses possibly multiple walls
$(\frakd_{\lambda}, f_{{\lambda}})_{n_j}$ $(\lambda\in \Lambda_j)$
at $t_j$.
Since $\gamma$ does not intersect $\mathrm{Sing}(\frakD)$,
these walls share the common normal vector $n_j\in N^+_{\mathrm{pr}}$.
Note that $\gamma'=-m_{j-1}$ just before $\gamma$ crosses these walls.
Then,
the intersection signs 
in \eqref{2eq:factor1}
are given by
\begin{align}
\label{2eq:factor2}
\epsilon_j
=
\begin{cases}
1 & \langle n_{j}, m_{j-1}\rangle>0,\\
-1 & \langle n_{j}, m_{j-1}\rangle<0.\\
\end{cases}
\end{align}
Thus, $\epsilon_j \langle n_{j}, m_{j-1}\rangle =|\langle n_{j}, m_{j-1}\rangle|$.
Therefore,  by \eqref{2eq:fpm1}, \eqref{2eq:pdef1}, and \eqref{2eq:pdef2},
we have
\begin{align}
\frakp_{\gamma_j, \frakD}(c_{j-1}x^{m_{j-1}})
=c_{j-1}x^{m_{j-1}
}
\prod_{\lambda\in \Lambda_j}
f_{{\lambda}}^{| \langle \d(n_j)n_j, m_{j-1}\rangle|}.
\end{align}
It is crucial that there is \emph{no division\/} in this expression
for the forthcoming positivity of theta functions.
\end{rem}

\begin{defn}[Theta function] \index{theta function}
Under the same assumption and notations in Definition
\ref{2defn:broken1},
the \emph{theta function $\vartheta_{Q,m_0}$} for $m_0$ with endpoint $Q$  is defined by
\begin{align}
\vartheta_{Q,m_0}
=\sum_{\gamma} \mathrm{Mono}(\gamma),
\end{align}
where the sum is over all broken lines for $m_0$ with endpoint $Q$.
We also define
\begin{align}
\label{2eq:t01}
\vartheta_{Q,0}=1.
\end{align}
\end{defn}

Let us quote the basic properties of theta functions from
\cite{Gross14}.

\begin{prop}[{\cite[Theorem~3.4]{Gross14}}]
For any scattering diagram $\frakD$, we have
\begin{align}
\vartheta_{Q,m_0}\in x^{m_0} \bbk[[P]].
\end{align}
\end{prop}

\begin{prop}[{\cite[Lemma~4.8]{Carl10}, \cite[Theorem~3.5]{Gross14}}]
\label{2prop:theta1}
Let $\frakD$ be a consistent scattering diagram.
Let $m_0\in M^{\circ}$ and $Q,Q'\in M_{\bbR}\setminus
\mathrm{Supp}(\frakD)$.
Then, for any admissible curve $\gamma$ from $Q$ to $Q'$,
we have
\begin{align}
\vartheta_{Q',m_0}=\frakp_{\gamma,\frakD}(\vartheta_{Q,m_0}).
\end{align}
\end{prop}

Now let us specialize to the scattering diagram $\frakD_{\fraks}$
in Theorem \ref{2thm:scat1}.
\begin{prop}[{\cite[Corollary~3.9]{Gross14}}]
\label{2prop:theta2}
Let $\frakD_{\fraks}$ be the scattering diagram
in Theorem \ref{2thm:scat1}.
Let us identify $M_{\bbR}\simeq_{\fraks} \bbR^n$.
For any $G$-cone $\sigma(G_t)$,
let $\bfm_0\in \sigma(G_t)\cap \bbZ^n$
and $Q\in \sigma^{\circ}(G_t)$.
Then, we have
\begin{align}
\vartheta_{Q,\bfm_0}=x^{\bfm_0}.
\end{align}
\end{prop}

From Propositions \ref{2prop:theta1}, \ref{2prop:theta2},
we have the third fundamental theorem on scattering diagrams
in view of the application to cluster algebra theory.

\begin{thm}
\label{2thm:theta5}
Let $\frakD_{\fraks}$ be the scattering diagram
in Theorem \ref{2thm:scat1}.
Let us identify $M_{\bbR}\simeq_{\fraks} \bbR^n$.
For any $G$-cone $\sigma(G_t)$,
let $\bfm_0\in \sigma(G_t)\cap \bbZ^n$
and $Q\in \sigma^{\circ}(G_{t_0})$.
Then, for any admissible curve $\gamma$ from any point in $\sigma^{\circ}(G_{t})$ to $Q$,
we have
\begin{align}
\label{2eq:pgdx1}
\vartheta_{Q,\bfm_0}=\frakp_{\gamma,\frakD}(x^{\bfm_0}).
\end{align}
\end{thm}

From Theorems \ref{2thm:F1} and \ref{2thm:theta5},
we have the following identification of  $x$-variables
in $\bfSigma_{\fraks,t_0}$
with theta functions.

\begin{thm}[{\cite[Theorem~4.9]{Gross14}}]
\label{2thm:theta4}
Let $\frakD_{\fraks}$ be the scattering diagram
in Theorem \ref{2thm:scat1}.
Let us identify $M_{\bbR}\simeq_{\fraks} \bbR^n$.
Then, for any $t\in \bbT_n$ and   $Q\in \sigma^{\circ}(G_{t_0})$,
we have
\begin{align}
 \label{2eq:theta1}
\vartheta_{Q,\bfg_{i;t}}
 &=
 x^{\bfg_{i;t}}
 F_{i;t}(\hat\bfy)
 =x_{i;t}.
  \end{align}
  \end{thm}
  
  \begin{proof}
  We set  $\bfm_0=\bfg_{i;t}$ in \eqref{2eq:pgdx1}.
  Then, comparing  it with    \eqref{2eq:fp6},
  we obtain the equality.
  \end{proof}

  The following positivity result holds.

\begin{prop}[{\cite[Theorem~1.13]{Gross14}}]
\label{2prop:theta3}
The scattering diagram $\frakD_{\fraks}$
in Theorem \ref{2thm:scat1} is equivalent to
a scattering diagram such that every wall function
has a form
\begin{align}
\label{2eq:f1}
f=(1+x^{p^*(n)})^c
\end{align}
for some (not necessarily primitive) $n\in N^+$ and a positive integer $c$.
\end{prop}

\begin{rem}
By Theorem \ref{2thm:clasterscatter1},
the complexity of  wall functions with  non-primitive $n$ or $c\neq 1$ occurs 
only outside $|\Delta(\bfG)|$, i.e., in the dark side.
\end{rem}

\begin{ex}
The wall function in \eqref{2eq:A11}
can be split into the form \eqref{2eq:f1} by
\begin{align}
\Biggl(
\sum_{k=0}^{\infty}
\hat{y}_1^k\hat{y}_2^k
\Biggr)^2
=
\prod_{j=0}^{\infty}
\left(
1+
\hat{y}_1^{2^j}\hat{y}_2^{2^j}
\right)^2.
\end{align}
\end{ex}

By the definition of broken lines, Remark \ref{2rem:positive1},
and Proposition \ref{2prop:theta3},
we have the fourth fundamental theorem on scattering diagrams
in view of the application to cluster algebra theory.

\begin{thm}[{\cite[Theorem~1.13 \& Remark 3.2]{Gross14}}]
\label{2thm:positivetheta1}
For the scattering diagram $\frakD_{\fraks}$  
in Theorem \ref{2thm:scat1},
every theta function $\vartheta_{Q,m_0}
\in x^{m_0} \bbk[[P]]$
has only nonnegative integer coefficients.
\end{thm}

Finally, we obtain the Laurent positivity in Theorem
  \ref{2thm:Fpositive1}.

\begin{thm}[{Laurent positivity, \cite[Theorem 4.10]{Gross14}}] \index{Laurent positivity}
For any cluster patterns $\bfSigma$ and a given initial vertex $t_0$,
every $F$-polynomial $F_{i;t}(\bfy)\in \bbZ[\bfy]$
has no negative coefficients.
\end{thm}
\begin{proof}
Recall that we have assumed Injectivity Assumption for $\Gamma$.
For a nonsingular $B$-pattern, 
it is a corollary of  Theorems  \ref{2thm:positivetheta1}, \ref{2thm:theta4},
and Remark \ref{2rem:F1}.
The singular case reduces to the nonsingular case
by Proposition \ref{2prop:principal1}.
\end{proof}

\subsection{Linear independence of cluster monomials}

Let us present another important consequence of  
Theorem \ref{2thm:theta4}.

Following \cite{Fomin03c}, we introduce the notion of cluster monomials.
\begin{defn}[Cluster monomials]
\index{cluster monomial}
For a given cluster pattern $\bfSigma$,
a {\em cluster monomial} at $t\in \bbT_n$ is a monomial
(including 1)
in the cluster variables $x_{1;t}$, \dots,  $x_{n;t}$ at $t$. 
\end{defn}

Two cluster variables taken from different $t$ and $t'$ in $\bbT_n$ may coincide.
For example, if $t$ and $t'$ are $k$-adjacent, a cluster monomial at $t$ that does
not contain $x_{k;t}$ as a factor is also a cluster monomial at $t'$.
We regard them as the same cluster monomial.
We are interested in the collection of all \emph{distinct} cluster monomials taken from all $t\in \bbT_n$.

The following  property was conjectured by \cite[Conj.~4.16]{Fomin03c}, then partially proved by
 \cite{Caldero05, Derksen10, Geiss10, Plamondon10b, Cerulli11, Cerulli12} for various classes,
and proved by \cite{Gross14} in full generality.
See also Corollary \ref{2cor:lin1}.

\begin{thm}[{\cite[Theorem 7.20]{Gross14}}]
\label{2thm:mono1}
For a given cluster pattern $\bfSigma$,
all distinct cluster monomials are linearly independent over $\bbZ$.
\end{thm}

Below we present a proof of Theorem \ref{2thm:mono1} following \cite{Gross14}.

First, we  express  cluster monomials with theta functions
in the same way as cluster variables given in Theorem \ref{2thm:theta4}.
A cluster monomial at $t$ has the following form
\begin{align}
x_t^{\bfa}=\prod_{i=1}^n x_{i;t}^{a_i}
\quad
(\bfa=(a_i) \in \bbZ_{\geq 0}^n).
\end{align}
We define the $g$-vector $\bfg_t^{\bfa}$ of $x_t^{\bfa}$ by
\begin{align}
\bfg_t^{\bfa}=\sum_{i=1}^n a_i \bfg_{i;t}.
\end{align}
By definition,
this is in the $G$-cone $\sigma(G_{t})$.
Under the situation of Theorem \ref{2thm:theta4},
we have
\begin{align}
\label{2eq:monot1}
\vartheta_{Q,\bfg_t^{\bfa}}
=\frakp_{\gamma,\frakD}( x^{\bfg_t^{\bfa}})
=\frakp_{\gamma,\frakD}\biggl(\prod_{i=1}^n (x^{\bfg_{i;t}})^{a_i}\biggr)=
\prod_{i=1}^n \frakp_{\gamma,\frakD}( x^{\bfg_{i;t}})^{a_i}
=x_t^{\bfa},
\end{align}
where we used \eqref{2eq:theta1} in the last equality.

By the Laurent phenomenon, we know that $\vartheta_{Q,\bfg_t^{\bfa}}=x_t^{\bfa}$ has a Laurent polynomial expression in $\bfx$.
We say that a Laurent monomial in $\bfx$ is \emph{proper} if it has a negative exponent for some $x_i$.
\index{proper Laurent monomial}

The following lemma is a key to prove Theorem \ref{2thm:mono1}.
\begin{lem}[{In the proof of \cite[Theorem 7.20]{Gross14}}]
\label{2lem:pL1}
In the Laurent polynomial expression of the theta function
$\vartheta_{Q,\bfg_t^{\bfa}}$ in \eqref{2eq:monot1},
every Laurent monomial in $\bfx$ is proper if $\bfg_t^{\bfa} \not\in \sigma(G_{t_0})
\cap \bbZ^n=\bbZ_{\geq 0}^n$.
\end{lem}
\begin{proof}
We may assume that $\bfg_t^{\bfa}\neq \bfzero$.
Suppose that $\vartheta_{Q,\bfg_t^{\bfa}}$ contains a non-proper Laurent monomial $M$ in $\bfx$.
Then, there is a broken line $\gamma$ for $\bfg_t^{\bfa}$ with endpoint $Q$ such that 
the final velocity $\gamma'(0)$ does not have any positive component.
Since $\gamma'(0)$ is not zero, it is a negative vector.
 Such a broken line has no break point; therefore, it is a ray ending at $Q$ with constant 
 velocity $-\bfg_t^{\bfa}$.
Thus, $\bfg_t^{\bfa}$ is a positive vector.
\end{proof}

Let us rephrase the above result in terms of  cluster monomials.

\begin{lem}
\label{2lem:pL2}
Let $\bfSigma$ be  a given cluster pattern.
In the Laurent polynomial expression of 
a cluster monomial $x_t^{\bfa}$ for $\bfSigma$ at $t$,
every Laurent monomial in $\bfx$ is proper 
  if $x_t^{\bfa}$ does not coincide with any cluster monomial at $t_0$.
\end{lem}
\begin{proof}
First, assume that the underlying $B$-pattern $\bfB$ is nonsingular.
Suppose that $x_t^{\bfa} $ contains a non-proper Laurent monomial $M$ in $\bfx$.
Then, by  \eqref{2eq:monot1} and Lemma \ref{2lem:pL1},
$x_t^{\bfa}=\vartheta_{Q,\bfg_t^{\bfa}}$ and
$\bfg_t^{\bfa}$ belongs to $\sigma(G_{t_0})$.
Thus, we can write $\bfg_t^{\bfa}=\sum_{i=1}^n b_i \bfe_i$ ($b_i\in \bbZ_{\geq 0}$).
Then, we obtain
\begin{align}
\label{2eq:monot2}
x_t^{\bfa}=
\vartheta_{Q,\bfg_t^{\bfa}}
=\frakp_{\gamma,\frakD}( x^{\bfg_t^{\bfa}})
=\frakp_{\gamma,\frakD}\biggl(\, \prod_{i=1}^n x_i^{b_i}\biggr)=
\prod_{i=1}^n (\frakp_{\gamma,\frakD}( x_i))^{b_i}
=x^{\bfg_t^{\bfa}}.
\end{align}
Thus, $x_t^{\bfa}$ coincides with a cluster monomial at $t_0$.

Next, assume that $\bfB$ is singular.
Consider the principal extension  $\overline{\bfB}$.
By Proposition \ref{2prop:principal1} and the separation formula \eqref{2eq:sep1}, 
$x_t^{\bfa}$ for  $\overline{\bfB}$ is given by
\begin{align}
\label{2eq:xta1}
x_t^{\bfa}&=
\prod_{i=1}^n
\biggl(\,
\prod_{j=1}^n
x_j^{g_{ji;t}}
\biggr)^{a_i}
F_{i;t}(\hat{\bfy})^{a_i}
\quad
(i=1,\dots,n),
\end{align}
where $G_t$ and $F_{i;t}$ are a $G$-matrix and an $F$-polynomial for $\bfB$.
Thus, this expression formally coincides with the one for $\bfB$.
However, in the formula,  for $\overline{\bfB}$, we have
\begin{align}
\hat{y}_i=\sum_{j=1}^{2n} x_j^{b_{ji}}=x_{n+i}\sum_{j=1}^n x_j^{b_{ji}}
\quad
(i=1,\dots,n),
\end{align}
and,  for $\bfB$, we have $\hat{y}_i=\sum_{j=1}^n x_j^{b_{ji}}$.
To extract the result for $\bfB$,
we note the following facts:
\begin{itemize}
\item[(a)]
Any cluster monomial 
$x_t^{\bfa}$ for  ${\bfB}$
is obtained from the corresponding
$x_t^{\bfa}$ for  $\overline{\bfB}$
by the specialization $x_{n+1}=\cdots = x_{2n}=1$.
\item[(b)]
Since $F$-polynomials are polynomials in $\hat{\bfy}$,
the exponents of $x_{n+i}$ ($i=1,\dots,n$) are nonnegative for every Laurent monomial in
the RHS of \eqref{2eq:xta1} for $\overline{\bfB}$.
Therefore, all the above proper Laurent monomials
are still proper under the specialization $x_{n+1}=\cdots = x_{2n}=1$.
\end{itemize}
Now, suppose that 
 $x_t^{\bfa} $ for $B$ contains a non-proper Laurent monomial
in $\bfx=(x_1,\dots,x_n)$.
Then, by (b),
the corresponding
 $x_t^{\bfa} $ for $\overline B$ also contains a non-proper Laurent monomial
in $\tilde \bfx=(x_1,\dots,x_{2n})$.
Thus,
by the first part of the proof,
 $x_t^{\bfa} $ for $\overline B$ coincides with a 
 cluster monomial for  $\overline B$ at $t_0$.
 Therefore, by (a),
 $x_t^{\bfa} $ for $ B$ coincides with a 
 cluster monomial for  $ B$ at $t_0$. 
\end{proof}

\index{proper Laurent monomial property}
The property presented in Lemma \ref{2lem:pL2} is called 
the {\em proper Laurent monomial property} in \cite{Cerulli11}. 
It was shown in \cite[Theorem 6.4]{Cerulli11} that
the linear independence in Theorem \ref{2thm:mono1} immediately follows
form this property.
For the reader's convenience, we reproduce the proof here.

\begin{proof}[Proof of Theorem \ref{2thm:mono1}]
Consider a $\bbZ$-linear relation among distinct cluster monomials.
One can organize it in the following form:
\begin{align}
\label{2eq:monorel1}
\sum_{i} c_i x_{t_i}^{\bfa_i}
=
\sum_i c'_i x^{\bfa'_i},
\end{align}
where each cluster monomial $x_{t_i}^{\bfa_i}$ 
in the LHS does not coincide with any cluster monomial  at $t_0$.
Also, $t_i=t_j$ may happen for $i\neq j$.
Then, by Lemma \ref{2lem:pL2},
 the LHS is a sum of proper Laurent monomials.
 Thus,
to have the equality,
the  both  sides of \eqref{2eq:monorel1} vanish.
Then, it follows that $c'_i=0$ for any $i$
by the linear independence of the cluster monomials at $t_0$, which is obvious.
One can repeat this argument by changing the initial vertex $t_0$ to other vertices in $\bbT_n$
and show that $c_i=0$ for any $i$.
\end{proof}

\begin{rem}
Using the same idea, the linear independence of a wider class of
theta functions including cluster monomials (called the \emph{theta basis}) was proved in \cite[Theorem 7.20]{Gross14}.
\end{rem}

One can immediately extend Theorem \ref{2thm:mono1}
to a  cluster pattern with coefficients in the sense of \cite{Fomin07}.
This was the original conjecture given by \cite{Fomin07}.

\begin{cor}
\label{2cor:lin1}
For any cluster pattern $\bfSigma$
with  coefficients in a given semifield $\bbP$,
all distinct cluster monomials are linearly independent over $\bbZ\bbP$.
\end{cor}
\begin{proof}
Suppose that  there is a linear relation among cluster monomials over $\bbZ\bbP$ for some cluster pattern with  coefficients.
Then, by trivializing the coefficients, we obtain a linear relation over $\bbZ$ for the cluster pattern without coefficients.
This is a contradiction.
\end{proof}

\subsection{Remarks on singular case}
One can actually remove  Injectivity Assumption of $\Gamma$,
or equivalently, the assumption of the nonsingularity of the corresponding $B$-pattern,
 and still obtain
an analogous result to Theorem \ref{2thm:theta4} directly.
Since this is beyond the scope of Part II, we only give a sketch
of the outline here.

On the $F$-polynomial side, 
we need to modify the result \eqref{2eq:fp5}.
We consider a cluster pattern
with \emph{principal coefficients\/} at $t_0$ in \cite{Fomin07}.
Let $\bbP=\mathrm{Trop}(\bfy)$ be the tropical semifield of $\bfy=\bfy_{t_0}$.
Let $\bbZ\bbP$ be the group ring of $\bbP$,
and let $\bbQ\bbP$ be the fraction field of $\bbZ\bbP$.
The $\haty$-variables are defined by
\begin{align}
\label{2eq:yhatp1}
\haty_{i;t}=y_{i;t}\prod_{j=1}^n x_{j;t}^{b_{ji;t}}.
\end{align}
In particular, let $\hat\bfy=(\hat y_i)$ be the initial $\haty$-variables,
\begin{align}
\haty_{i}=\haty_{i;t_0}=y_{i}\prod_{j=1}^n x_{j}^{b_{ji}},
\end{align}
which are algebraically independent in the ambient field $\calF_X=(\bbQ\bbP)(\bfx)$
whether $B$ is nonsingular or singular.
Let $t, t'\in \bbT_n$ be vertices  that are $k$-adjacent.
The $y$-variables transform in $\bbP$ by
\begin{align}
\label{2eq:ymut31}
y_{i;t'}
&=
\begin{cases}
\displaystyle
y_{k;t}^{-1}
& i=k,
\\
y_{i;t} y_{k;t}^{[\varepsilon b_{ki;t}]_+} (1\oplus y_{k;t}^{\varepsilon})^{-b_{ki;t}}
&i\neq k.
\end{cases}
\end{align}
Meanwhile,
one can regard
mutations of $x$-variables
as  isomorphisms of fields 
 as follows,
where $\bfx_{t}$, $\bfx_{t'}$ are $n$-tuple of formal variables:
\begin{align}
\mu_{k;t}:(\bbQ\bbP)(\bfx_{t'}) & \rightarrow (\bbQ\bbP)(\bfx_t),
\\
\label{2eq:xmut31}
\mu_{k;t}(x_{i;t'}) &=
\begin{cases}
\displaystyle
x_{k;t}^{-1}\Biggl(\, \prod_{j=1}^n x_{j;t}^{[-\varepsilon b_{jk;t}]_+}
\Biggr)
\frac{
 1+\hat{y}_{k;t}^{\varepsilon}
}{
 1\oplus \hat{y}_{k;t}^{\varepsilon}
}
& i=k,
\\
x_{i;t}
&i\neq k.
\end{cases}
\end{align}
Compare it with \eqref{2eq:xmut3}.
Note that, for the tropical sign $\varepsilon_{k;t}$, we have
\begin{align}
1\oplus y_{k;t}^{\varepsilon_{k;t}}=1,
\end{align}
thus, the factor $1\oplus \hat{y}_{k;t}^{\varepsilon}$ 
in \eqref{2eq:ymut31} and  \eqref{2eq:xmut31} disappears
by setting $\varepsilon=\varepsilon_{k;t}$.
Then, we just repeat the same decomposition \eqref{2eq:decom1}
and so on, and we obtain the same formula as \eqref {2eq:fp5},
 \begin{align}
 \label{2eq:fp52}
 \frakq_t^{t_0}(x^{\bfg_{i;t}})
 &=
 x^{\bfg_{i;t}}
 F_{i;t}(\hat\bfy),
 \end{align}
but now $\haty_1$, \dots, $\haty_n$ are  algebraically independent as mentioned.

Meanwhile, on the scattering diagram side, following  \cite[Appendix B]{Gross14},
we consider the extension of 
 the lattice $M^{\circ}$ of a fixed data $\Gamma$,
\begin{align}
\tilde M^{\circ} = M^{\circ}\oplus N.
\end{align}
Then, we introduce a parallel map  to  $p^*$ in \eqref{2eq:p*1},
\begin{gather}
\label{2eq:p*11}
\tilde p^* :  N \rightarrow \tilde M^{\circ},
\\
( \tilde p^*(n))(n'+m') = \{ n', n\}+ \langle n, m \rangle,
\quad (n'\in N^{\circ}, m'\in M).
\end{gather}
We have
\begin{align}
\label{2eq:pe12}
 \tilde p^*(e_j)=\sum_{i=1}^n b_{ij}f_i + e_j,
 \end{align}
 so that the \emph{map $\tilde p^*$ is injective whether $B$ is singular or not}.
 Also,
 under the identification \eqref{2eq:initials1}, we have
 \begin{align}
\label{2eq:pe22}
x^{\tilde p^*(e_i)}=y_i\prod_{j=1}^n x_j^{b_{ji}}=: \haty_i,
\end{align}
which coincides with \eqref{2eq:yhatp1}.
We also need to extend the monoid $P$ in \eqref{2eq:P1} with $\tilde P\subset \tilde M^{\circ}$
as follows:
\begin{itemize}
\item[(i).]
$\tilde P=\sigma \cap \tilde M^{\circ}$,
where $\sigma $ is a $2r$-dimensional strongly convex cone in $\tilde M^{\circ}_{\bbR}$.
\item[(ii).]
$\tilde p^*(e_1), \cdots, \tilde p^*(e_r)\in \tilde P$.
\end{itemize}
Such $\tilde P$ is not unique at all, and we choose one arbitrarily.
The result does not depend on the choice of $\tilde P$.
Then, for a wall $(\frakd, f)_n$, we replace the wall function $f$ in \eqref{2eq:fd1} with
\begin{align}
\label{2eq:fd2}
f=1+\sum_{k=1}^{\infty} c_k x^{k\tilde p^*(n)} \in {\bbk[[\tilde P]]},
\end{align}
while $n$ and $\frakd$ remain the same.
Then, under this modification, we have a parallel result for Theorems \ref{2thm:theta5}
and \ref{2thm:theta4} without assuming Injectivity Assumption.
More details will be found in Part III.

\newpage
\section{Some applications}

Let us give some applications of the results
and the techniques presented so far.

\subsection{Detropicalization revisited}
\label{2subsec:synchro2}
We return to the situation in Section
\ref{2subsec:synchro1}.

 As mentioned in Remark \ref{rem:detrop1},
 the following result was given for $\sigma=\mathrm{id}$
 by \cite{Cao17},
 and  for general $\sigma$ by
  \cite{Nakanishi19},
  where both
 rely on the Laurent positivity.
 Here, we give an alternative proof based on
 Theorem  \ref{2thm:F1} and
the consistency of the scattering diagram $\frakD_{\fraks}$,
without relying on the Laurent positivity.

 \begin{thm}[{Detropicalization \cite[Lemma~2.4 \& Theorem~2.5]{Cao17}, \cite[Theorem~5.2]{Nakanishi19}}]
 \label{2thm:detrop1} \index{detropicalization}
  Let $\bfSigma$ be any cluster pattern of rank $n$,
 and let $t_0$ be a given initial vertex.
Then, for any $t,t'\in \bbT_n$ and any permutation
$\nu\in S_n$, the following facts hold:
 \begin{itemize}
\item[(a).] $G_{t}=\nu G_{t'}\ \Longrightarrow\ \bfx_{t}=\nu \bfx_{t'}$.
\par
\item[(b).] $C_{t}=\nu C_{t'}\ \Longrightarrow\ \bfy_{t}=\nu \bfy_{t'}$.
\end{itemize}
\end{thm}

\begin{proof}
First, assume that the underlying $B$-pattern $\bfB$ is nonsingular.
\par
\par
(a). Assume that $G_t=\nu G_{t'}$.
Then, $\sigma(G_t)=\sigma( G_{t'})$.
Therefore, for $\frakp_{t}^{t_0}$ in \eqref{2eq:fp7},
we have
\begin{align}
\label{2eq:pp1}
\frakp_{t}^{t_0}=\frakp_{t'}^{t_0}
\end{align}
by the consistency of the scattering diagram $\frakD_{\fraks}$.
Then, we have $\frakq_{t}^{t_0}=\frakq_{t'}^{t_0}$
(for $x$-variables)
by Theorem \ref{2thm:F1} (a) and  \eqref{2eq:pq2}.
Therefore, 
 $\bfx_{t}=\nu \bfx_{t'}$ by Proposition \ref{2prop:detrop1}.
Alternatively, by Theorem \ref{2thm:F1} (b) and \eqref{2eq:pp1},
we have
\begin{align}
 \label{2eq:fp11}
   \frakp_t^{t_0}(x^{\bfg_{i;t}})
 &=
 x^{\bfg_{i;t}}
 F_{i;t}(\hat\bfy)=x_{i;t},
 \\
  \label{2eq:fp12}
 \frakp_t^{t_0}(x^{\bfg_{i;t}})
 &=
 \frakp_{t'}^{t_0}(x^{\bfg_{\nu^{-1}(i);t'}})
 =
  x^{\bfg_{\nu^{-1}(i);t'}}
 F_{\nu^{-1}(i);t'}(\hat\bfy)
 =
 x_{\nu^{-1}(i);t'}.
  \end{align}
Thus, we have $\bfx_{t}=\nu \bfx_{t'}$.
\par
(b).
Assume that $C_t=\nu C_{t'}$.
Then, by the duality \eqref{2eq:dual1}
and \eqref{2eq:cbc1},
we obtain \begin{align}
\d_{\nu(i)}=\d_i,
\
G_t=\nu G_{t'},
\
B_t=\nu B_{t'}.
\end{align}
See \cite[Prop.~4.4 \& Cor.~4.5]{Nakanishi19} for details.
Then, from \eqref{2eq:fp11} and \eqref{2eq:fp12},
we have
\begin{align}
 F_{i;t}(\hat\bfy)=
F_{\nu^{-1}(i);t'}(\hat\bfy).
\end{align}
By Remark  \ref{2rem:F1}, we have
\begin{align}
 F_{i;t}(\bfy)=
F_{\nu^{-1}(i);t'}(\bfy).
\end{align}
Thus, by the separation formula \eqref{2eq:sep2},
we obtain
 $\bfy_{t}=\nu \bfy_{t'}$.
\par
Next, we consider the case when $\bfB$ is singular.
Thanks to Proposition \ref{2prop:principal1},
 it reduces to the nonsingular case as follows:
 \begin{align}
 G_t=\nu G_{t'}
 \ \Longrightarrow\
  \overline{G}_t=\nu \overline{G}_{t'}
 \ \Longrightarrow\
  \overline{\bfx}_t=\nu \overline{\bfx}_{t'}
   \ \Longrightarrow\
  {\bfx}_t=\nu {\bfx}_{t'},
  \\
   C_t=\nu C_{t'}
 \ \Longrightarrow\
  \overline{C}_t=\nu \overline{C}_{t'}
 \ \Longrightarrow\
  \overline{\bfy}_t=\nu \overline{\bfy}_{t'}
   \ \Longrightarrow\
  {\bfy}_t=\nu {\bfy}_{t'},
 \end{align}
 where $  \overline{\bfSigma}=\{\overline{\Sigma}_t=(\overline{\bfx}_t, \overline{\bfy}_t,
 \overline{B}_t)\}_{t\in \bbT_n}$ is a cluster pattern 
 in Proposition \ref{2prop:principal1}.
\end{proof}

\subsection{Bijection between $g$-vectors and $x$-variables}

Here we  prove a  statement,
which sharpens Theorem
\ref{2thm:detrop1} (a).
This was proved
by \cite{Cerulli12} in the skew-symmetric case
with the representation/categorical method.
Also, the implication $\Longrightarrow$ was proved by
\cite[Theorem~3.2]{Cao18} in the skew-symmetrizable case
with principal coefficients.
In contrast to the proof of Theorem
\ref{2thm:detrop1},
here we rely on the Laurent positivity.

\begin{thm}
\label{2thm:gx1}
Let $\bfSigma$ be any cluster pattern,
 and let $t_0$ be a given initial vertex.
Then, we have
\begin{align}
\bfg_{i;t}^{t_0}=\bfg_{i';t'}^{t_0}
\quad
\Longleftrightarrow
\quad
x_{i;t}=x_{i';t'}.
\end{align}
\end{thm}

\begin{proof}
($\Longleftarrow$)
First we apply the separation formula \eqref{2eq:sep1} to $x_{i';t'}$ 
with respect to the initial vertex $t$.
Then,  by assumption, we have
\begin{align}
x_{i';t'}=\Biggl(\, \prod_{j=1} x_{j;t}^{g_{ji';t'}^{t}}\Biggr)F_{i';t'}^t(\hat\bfy_t)=x_{i;t}.
\end{align}
By the equivalent statement of the sign-coherence
in Conjecture \ref{2conj:Fconst1},
which is now proved,
 the constant term of $F_{i';t'}^t(\bfy_t)$ is 1.
Then, by the Laurent positivity, we have $F_{i;t'}^t(\bfy_t)=1$
and also $\bfg_{i';t'}^t=\bfe_i=\bfg_{i;t}^t$.
Then, applying the bijection in Proposition \ref{2prop:Gfanphi1} repeatedly,
we obtain $\bfg_{i';t'}^{t_0}=\bfg_{i;t}^{t_0}$.
\par
($\Longrightarrow$)
From the assumption,
we obtain
$\bfg_{i';t'}^t=\bfe_i$
by the opposite procedure to the above.
Thus, we have
\begin{align}
x_{i';t'}=x_{i;t}F_{i';t'}^t(\hat\bfy_t).
\end{align}
Changing the role of $t$ and $t'$, we also have
\begin{align}
x_{i;t}=
x_{i';t'}F_{i;t}^{t'}(\hat\bfy_{t'}).
\end{align}
From two equalities, we obtain
\begin{align}
F_{i';t'}^t(\hat\bfy_t)F_{i;t}^{t'}(\hat\bfy'_{t})=1.
\end{align}
Again,  by the constant term 1 of $F$-polynomials and the Laurent positivity,
we have $F_{i';t'}^t(\bfy_t)=F_{i;t}^t(\bfy_{t})=1$.
Therefore, $x_{i;t}=x_{i';t'}$.
\par
($\Longrightarrow$) Let us give an alternative proof using the consistency of
the scattering diagram
$\frakD_{\fraks}$ without using the Laurent positivity.
Assume that $\bfB$ is nonsingular.
Let us identify $M_{\bbR}\simeq_{\fraks} \bbR^n$.
Then, we have
\begin{align}
\label{2eq:xp1}
x_{i;t}=\frakp_{t}^{t_0}(x^{\bfg_{i;t}}),\quad
x_{i';t'}=\frakp_{t'}^{t_0}(x^{\bfg_{i';t'}})=\frakp_{t'}^{t_0}(x^{\bfg_{i;t}}).
\end{align}
By assumption, two cones $\sigma(G_{t})$, $\sigma(G_{t'})$ have
the ray
 $\sigma(\bfg_{i;t})=\sigma(\bfg_{i';t'})$
 as a common face.
Let $\gamma_{t}^{t'}$ be an admissible curve 
from $\sigma^{\circ}(G_{t})$ to $\sigma^{\circ}(G_{t'})$
for $\frakD_{\fraks}$.
(It may  not be completely contained 
in $|\Delta(\bfG^{t_0})|$.)
For each degree $\ell$,
by ignoring all walls of $\frakD_{\fraks}$ whose wall functions are trivial modulo  $\widehat{J}^{\ell}$,
one can take $\gamma_{t}^{t'}$ close enough to the ray
$\sigma(\bfg_{i;t})$ so that $\gamma_{t}^{t'}$
passes through only walls of $\frakD_{\fraks}$ that contain $\sigma(\bfg_{i;t})$.
Since the normal vectors of these walls are orthogonal to $\bfg_{i;t}$,
the contribution to the wall-crossing automorphism $\frakp_{\gamma_{t}^{t'}, \frakD_{\fraks}}$ from each wall is trivial.
Therefore,
we have
\begin{align}
\label{2eq:xp21}
\frakp_{\gamma_{t}^{t'}, \frakD_{\fraks}}(x^{\bfg_{i;t}})
\equiv x^{\bfg_{i;t}}
\mod \widehat{J}^{\ell}.
\end{align}
Since $\ell$ is arbitrary, this implies that
\begin{align}
\label{2eq:xp2}
\frakp_{\gamma_{t}^{t'}, \frakD_{\fraks}}(x^{\bfg_{i;t}})
= x^{\bfg_{i;t}}.
\end{align}
By the consistency of $\frakD_{\fraks}$, we have
\begin{align}
\label{2eq:xp3}
\frakp_{t'}^{t_0}\circ\frakp_{\gamma_{t}^{t'}}=\frakp_{t}^{t_0}.
\end{align}
By combining \eqref{2eq:xp1}--\eqref{2eq:xp3}, we obtain $x_{i;t}=x_{i';t'}$.
Next, we consider the case when $\bfB$ is singular.
Thanks to Proposition \ref{2prop:principal1},
 it reduces to the nonsingular case as follows:
\begin{align}
\bfg_{i;t}^{t_0}=\bfg_{i';t'}^{t_0}
\ \Longrightarrow\
\overline\bfg_{i;t}^{t_0}=\overline\bfg_{i';t'}^{t_0}
\ \Longrightarrow\
\overline{x}_{i;t}=\overline{x}_{i';t'}
\ \Longrightarrow\
x_{i;t}=x_{i';t'}.
\end{align}
\end{proof}

\begin{rem}
The first proof of the implication $\Longrightarrow$ is 
a little simplified version of the proof of
 \cite[Theorem~3.2]{Cao18}.
\end{rem}

\begin{rem}
\label{2rem:coef1}
The above proof (as for the implication $\Longrightarrow$, the first one)
is applicable to any cluster pattern \emph{with  arbitrary coefficients\/}
in the sense of \cite{Fomin07},
with a little care of coefficients.
Originally in \cite{Fomin07}, a $g$-vector is uniquely  associated with each $x$-variable only
for cluster patterns with geometric type satisfying a certain condition in
 \cite[Eq.~(7.10)]{Fomin07}. Now we can safely associate a $g$-vector with  each $x$-variable  for any cluster pattern
\emph{with  arbitrary coefficients}.
\end{rem}

\begin{rem} As for $y$-variables and $c$-vectors,
the implication
\begin{align}
\bfc_{i;t}^{t_0}=\bfc_{i';t'}^{t_0}
\quad
\Longleftarrow
\quad
y_{i;t}=y_{i';t'}
\end{align}
holds by Proposition \ref{2prop:ytrop1}.
However, the opposite implication does not hold. For example, for type $A_2$ case in
Section \ref{2subsec:rank2G1}, we have
\begin{align}
y_{2;t_{0}}=y_2, \quad
y_{2;t_{1}}=y_2(1+y_1),
\end{align}
which is already a counterexample.
\end{rem}

One can rephrase Theorem \ref{2thm:gx1}
using the notion of the cluster complex in \cite{Fomin03a}.
\begin{defn}[Cluster complex] \index{cluster complex}
The \emph{cluster complex\/} $\Delta(\bfSigma)$ of a cluster pattern $\bfSigma$ is the
simplicial complex whose vertices are $x$-variables (cluster variables)
and whose simplices are nonempty subsets of clusters.
\end{defn}

We define a parallel notion for a  $G$-pattern $\bfG^{t_0}$ of $\bfSigma$,
where we encounter  the conflict of notation with a $G$-fan $\Delta(\bfG^{t_0})$.
However, two notions are essentially equivalent, so that we abuse the notation.

\begin{defn}[$G$-complex] \index{$G$-complex}
The \emph{$G$-complex\/} $\Delta(\bfG^{t_0})$ of a $G$-pattern $\bfG^{t_0}$ is the
simplicial complex whose vertices are $g$-vectors in $\bfG^{t_0}$
and whose simplices are  nonempty sets of  $g$-vectors belonging to
a common $G$-matrix.
\end{defn}
One can geometrically identify a simplex in the $G$-complex $\Delta(\bfG^{t_0})$ with
the intersection $\sigma\cap S^{n-1}$ 
of a
cone  $\sigma$ in the $G$-fan $\Delta(\bfG^{t_0})$
and the unit sphere $S^{n-1}$ in $\bbR^n$.

We have a corollary of Theorem \ref{2thm:gx1}.

\begin{cor}
\label{2cor:isocomp1}
For any cluster pattern $\bfSigma$ and any vertex $t_0$, we have
an isomorphism of simplicial complexes
\begin{align}
\label{2eq:comiso1}
\begin{matrix}
\Delta(\bfSigma)&\simeq&\Delta(\bfG^{t_0})
\\
x_{i;t}& \longleftrightarrow & \bfg_{i;t}.
\end{matrix}
\end{align}
\end{cor}

Due to Remark \ref{2rem:coef1},
we can extend the result to any cluster pattern with arbitrary coefficients
in the sense of \cite{Fomin07}. 
This extends the result of \cite[Cor.~5.6]{Cerulli12} in the skew-symmetric case.

\begin{cor}
\label{2cor:cciso1}
For any cluster pattern $\bfSigma$
with arbitrary coefficients
in the sense of \cite{Fomin07}
 and any vertex $t_0$, we have
 the same isomorphism as  \eqref{2eq:comiso1}.
\end{cor}

We have a further corollary of Corollary \ref{2cor:cciso1}
and Theorem \ref{2thm:Gfan1}.
The equivalence $(a)\Longleftrightarrow (d)$ below
is a part of a conjecture by \cite[Conj. 4.14(b)]{Fomin03c},
and it was proved by
 \cite{Gekhtman07} with algebraic arguments.
 Here we replace the proof with geometric arguments.
 
\begin{cor}[{cf.~\cite[Theorem~5]{Gekhtman07}}].
For any cluster pattern $\bfSigma$
with arbitrary coefficients
in the sense of \cite{Fomin07}
 and any vertex $t_0$,
 the following conditions for $t,t'\in \bbT_n$ are equivalent:
 \begin{itemize}
 \item[(a).]
 $\bfx_t$ and $\bfx_{t'}$
 contain exactly $n-1$ common elements (as a set),
where $n$ is the rank of $\bfSigma$.
 \item[(b).]
$\sigma(G_t)$ and $\sigma(G_{t'})$ intersects in
their common face of codimension one.
 \item[(c).]
 There are some $t''\in \bbT_n$
 that is  adjacent to $t$ 
 and some permutation  $\nu\in S_n$ such that
 $G_{t'}=\nu G_{t''}$.
  \item[(d).]
 There are some $t''\in \bbT_n$
 that is  adjacent  to $t$ 
 and some permutation  $\nu\in S_n$ such that
 $\bfx_{t'}=\nu\bfx_{t''}$.
 \end{itemize}
\end{cor}
\begin{proof}
The equivalences $(a)\Longleftrightarrow (b)$ and $(c)\Longleftrightarrow (d)$ 
are due to Corollary \ref{2cor:cciso1},
while $(b)\Longleftrightarrow (c)$ is due to Theorem \ref{2thm:Gfan1}.
\end{proof}
